\documentclass[a4paper,11pt]{amsart}
\usepackage{cases}

\usepackage{mathrsfs}
\usepackage{amsfonts,amssymb,amscd,amsthm,amsmath,graphicx}
\usepackage{longtable}
\usepackage[all]{xy}

\usepackage{graphicx}

\theoremstyle{plain}
\newtheorem{theorem}{Theorem}[section]
\newtheorem{lemma}[theorem]{Lemma}
\newtheorem{corollary}[theorem]{Corollary}
\newtheorem{proposition}[theorem]{Proposition}

\newtheorem{definition}[theorem]{Definition}
\newtheorem{remark}[theorem]{Remark}
\newtheorem{example}[theorem]{Example}

\newtheorem{question}{Question}[section]

\def\A{\operatorname{A}}
\def\B{\operatorname{B}}
\def\C{\operatorname{C}}
\def\BC{\operatorname{BC}}
\def\D{\operatorname{D}}
\def\E{\operatorname{E}}
\def\F{\operatorname{F}}
\def\G{\operatorname{G}}

\def\PSU{\operatorname{PSU}}

\def\SO{\operatorname{SO}}
\def\Sp{\operatorname{Sp}}
\def\SU{\operatorname{SU}}
\def\U{\operatorname{U}}

\def\Ad{\operatorname{Ad}}

\def\Aut{\operatorname{Aut}}

\def\det{\operatorname{det}}

\def\diag{\operatorname{diag}}

\def\Eva{\operatorname{Eva}}

\def\Hom{\operatorname{Hom}}

\def\Int{\operatorname{Int}}

\def\ker{\operatorname{ker}}

\def\Lie{\operatorname{Lie}}

\def\max{\operatorname{max}}

\def\Out{\operatorname{Out}}

\def\R{\operatorname{R}}

\def\rank{\operatorname{rank}}
\def\RD{\operatorname{RD}}
\def\Rep{\operatorname{Rep}}

\def\sgn{\operatorname{sgn}}
\def\sign{\operatorname{sign}}
\def\span{\operatorname{span}}

\def\Spin{\operatorname{Spin}}
\def\st{\operatorname{st}}
\def\Stab{\operatorname{Stab}}
\def\supp{\operatorname{supp}}

\newcommand{\bbC}{\mathbb{C}}

\newcommand{\bbR}{\mathbb{R}}
\newcommand{\bbZ}{\mathbb{Z}}

\newcommand{\bbQ}{\mathbb{Q}}

\newcommand{\fre}{\mathfrak{e}}

\newcommand{\frg}{\mathfrak{g}}
\newcommand{\frh}{\mathfrak{h}}

\newcommand{\frp}{\mathfrak{p}}

\newcommand{\frt}{\mathfrak{t}}
\newcommand{\fru}{\mathfrak{u}}

\begin{document}

\title{On the dimension datum problem and the linear dependence problem}

\author{Jun Yu}

\address{School of Mathematics,
Institute for Advanced Study,
Einstein Drive,
Fuld Hall,
Princeton, NJ 08540}
\email{junyu@math.ias.edu}

\abstract{The dimension datum of a closed subgroup of a compact Lie group is the sequence of
invariant dimensions of irreducible representations by restriction. In this article we classify
closed connected subgroups with equal dimension data or linearly dependent dimension
data. This classification should have applications to the isospectral geometry and automorphic
form theory. We also study the equality/linear dependence of not necessarily connected subgroups
of unitary group acting irreducibly on the natural representation.}
\endabstract



\maketitle

\tableofcontents

\section{Introduction}

Given a compact Lie group $G$, the dimension datum $\mathscr{D}_{H}$ of a closed subgroup $H$ of $G$ is
the map from $\widehat{G}$ to $\bbZ$, $$\mathscr{D}_{H}: V\mapsto\dim V^{H},$$ where $\widehat{G}$ is
the set of equivalence classes of irreducible finite-dimensional complex linear representations of $G$
and $V^{H}$ is the subspace of $H$-invariant vectors in a complex linear representation $V$ of $G$.

In number theory, dimension data arose from the determination of some monodromy groups
(cf. \cite{Katz}). In the theory of automorphic forms, Langlands (\cite{Langlands}) has suggested
to use dimension data as a key ingredient in his program ``Beyond Endoscopy''. The idea is to
use the dimension datum to identify the conjectural subgroup ${}^\lambda H_\pi\subset {}^L \mathscr{G}$
associated to an automorphic representation $\pi$ of $\mathscr{G}(\mathbb{A})$, where ${}^L\mathscr{G}$
is the $L$-group of $\mathscr{G}$ in the form of \cite[2.4(2)]{Borel}. In differential geometry, the
dimension datum $\mathscr{D}_H$ is related to the spectrum of the Laplace operator on the homogeneous
Riemannian manifold $G/H$.

Let $H\subset G$ be compact Lie groups. The {\it dimension datum problem} asks the following question.
\begin{question}\label{Q:dimension data}
To what extent is \(H\) (up to \(G\)-conjugacy) determined by its dimension datum \(\mathscr{D}_H\) ?
\end{question}

In \cite{Larsen-Pink}, Larsen and Pink considered the dimension data of connected semisimple subgroups.
They showed that the subgroups are determined up to isomorphism by their dimension data (cf. \cite{Larsen-Pink},
Theorem 1), and not up to conjugacy in general (cf. \cite{Larsen-Pink}, Theorem 3). Moreover, they showed
that the subgroups $H$ are determined up to conjugacy if $G=\SU(n)$ and the inclusions
$H\hookrightarrow\SU(n)$ give irreducible representations of $H$ on $\bbC^{n}$ (cf. \cite{Larsen-Pink},
Theorem 2).

In \cite{Langlands}, Subsections 1.1 and 1.6, Langlands raised another question (cf. \cite{An-Yu-Yu},
Question 5.3) about dimension data. If this question has an affirmative answer, then it will facilitate
a way of dealing with the dimension data of ${}^\lambda H_\pi$ using trace formulas. However, Langlands
suspected (cf. \cite{Langlands}, discussions following Equation (14)) that in general this question has an
affirmative answer. As observed in \cite{An-Yu-Yu}, this question proposed by Langlands is equivalent
to the following question that we shall call the {\it linear dependence problem}.

\begin{question}\label{Q:linear dependence}
Given a list of finitely many closed subgroups $H_1,\dots,H_{n}$ of $G$ with $\mathscr{D}_{H_{i}}\neq
\mathscr{D}_{H_{j}}$ for any $i\neq j$, are $\mathscr{D}_{H_1},\dots,\mathscr{D}_{H_{n}}$ linearly
independent?
\end{question}

In \cite{An-Yu-Yu}, jointly with Jinpeng An and Jiu-Kang Yu, we have given counter-examples to show
that the affirmative answer to Question \ref{Q:dimension data} (or \ref{Q:linear dependence}) is not
always positive. In this paper, we classify closed connected subgroups with equal dimension data
or linearly dependent dimension data.The method of this classification is as follows.

Given a compact Lie group $G$ we choose a bi-invariant Riemannian metric $m$ on it. For a closed
connected torus $T$ in $G$, we introduce root systems on $T$. Among the root systems on $T$, there is a maximal
one, denoted by $\Psi_{T}$, which contains all other root systems on $T$. For each reduced root system $\Phi$
on $T$ and a finite group $W$ acting on $\frt_0=\Lie T$ satisfying some condition (cf. Definition
\ref{D:characters}), we define a character $F_{\Phi,W}$. In particular, this definition applies to the finite
group $\Gamma^{\circ}:=N_{G}(T)/C_{G}(T)$. We define another root system $\Psi'_{T}$ on $T$ as the sub-root
system of $\Psi'_{T}$ generated by the root systems of closed connected subgroups $H$ of $G$ with $T$ a maximal
torus of $H$. Moreover we show that the Weyl group of $\Psi'_{T}$ is a sub-group of $\Gamma^{\circ}$. By
Propositions \ref{P: ST support} and \ref{P:group-root system}, we show that the dimension datum problem (or
the linear dependence problem) reduces to comparing the characters
$\{F_{\Phi,\Gamma^{\circ}}|\ \Phi\subset\Psi_{T}\}$ (or finding linear relations among them). We propose
Questions \ref{Q:equal-character} and \ref{Q:dependent-character}. They are concerned with the characters
associated to reduced sub-root systems of a given root system $\Psi$. In the next two sections we solve these
two questions completely. We reduce both questions to the case where $\Psi$ is an irreducible root system. In
the case that $\Psi$ is a classical irreducible root system, the classification of sub-root systems of $\Psi$
is well-known. In \cite{Larsen-Pink}, Larsen and Pink defined an algebra isomorphism $E$ taking characters of
sub-root systems to polynomials (cf. Lemma \ref{Isomorphism E} and Definition \ref{D:abcd}). In this case we
solve Questions \ref{Q:equal-character} and \ref{Q:dependent-character} by getting all algebraic relations among
polynomials in the image of $E$. The relations of some polynomials given in Propositions \ref{P:A=BB and A=CD}
and \ref{P:A-BB'CD} are most important. In the case that $\Psi$ is an exceptional irreducible root system,
Oshima classified sub-root systems of $\Psi$ (cf. \cite{Oshima}). In \cite{Larsen-Pink}, the authors defined a
weight $2\delta'_{\Phi}$ for each reduced sub-root system $\Phi$ of $\Psi$. We give the formulas of these weights
in the Tables 2-9. We also define and calculate a generating function for each reduced sub-root system of $\Psi$
(cf. Definition \ref{D:f}). With these, we are able to find and prove all linear relations among the characters
of reduced sub-root systems of $\Psi$. Therefore we solve Questions \ref{Q:equal-character} and
\ref{Q:dependent-character} completely.

Given a connected closed torus $T$ in $G$, the finite group $\Gamma^{\circ}$ palys an important role for
the dimension data of closed connected subgroups $H$ of $G$ with $T$ a maximal torus of $H$. In Proposition
\ref{P:Gamma0-Psi'} we show that $\Gamma^{\circ}\sup W_{\Psi'_{T}}$, where $\Psi'_{T}$ is the sub-root system
of $\Psi_{T}$ generated by root systems of closed connected subgroups $H$ of $G$ with $T$ a maximal torus of $H$.
On the other hand, in Proposition \ref{P:Gamma0-2} we give a construction showing that the finite group
$\Gamma^{\circ}$ could be arbitrary providing that it containing $W_{\Psi'_{T}}$. These supplements show that
our solution to Questions \ref{Q:equal-character} and \ref{Q:dependent-character} actually give all examples
of closed connected subgroups with equal dimension data or linearly dependent dimension data.

In the last section, we study the equalities and linear relations among dimension data of closed subgroups
of $\U(n)$ acting irreducibly on $\bbC^{n}$. For nonconnected irreducible subgroups, we give ineteresting
examples with equal or linearly dependent dimension data for $n=16$ and $n=12$ respectively. We also show
that the dimension data of closed connected irredicible subgroups are actually linearly independent. This
strengthens a theorem of Larsen and Pink. The proof is identically theirs. The writing of this section is
inspired by questions of Peter Sarnak.

The following theorem is a simple consequence of our classification of subgroups with equal dimension data. It
follows from Proposition \ref{P:group-root system}, Theorem \ref{T:equal character-nonsimple}, Theorem
\ref{T:character-classical} and Theorem \ref{character-exceptional}.

\begin{theorem}\label{T:dimension-algebra}
If two compact Lie groups $H,H'$ have inclusions to a compact Lie group $G$ with the same dimension data,
then after replacing each ideal of $\Lie H$ and $\Lie H'$ isomorphic to $\mathfrak{u}(2n+1)$ ($n\geq 1$)
by an ideal isomorphic to $\mathfrak{sp}(n)\oplus\mathfrak{so}(2n+2)$, the resulting Lie algebras are
isomorphic.
\end{theorem}

The classification of subgroups with linearly dependent dimension data is more complicated. The
following theorem is a simple consequence of this classification. It follows from Propositions
\ref{P:linear-A}, \ref{P:linear-BCD} and \ref{P:linear-exceptional}. In the case of type $\B_2$, $\B_3$
or $\G_2$, we also have a linear independence result. The proof requires some consideration of possible
connected full rank subgroups.

\begin{theorem}\label{T:linear-type A}
Given $G=\SU(n)$ or a compact connected Lie group isogeneous to it, for any list $\{H_1,H_2,\dots,H_{s}\}$
of non-conjugate connected closed full rank subgroups of $G$, the dimension data
$\mathscr{D}_{H_1},\mathscr{D}_{H_2},\cdots,\mathscr{D}_{H_s}$ are linearly independent. Given a comapct
connected Lie group $G$ with a simple Lie algebra $\frg_0$, if $\frg_0$ is not of type $\A_{n}$, $\B_2$,
$\B_3$ or $\G_2$, then there exist non-conjugate connected closed full rank subgroups $H_1$, $H_2$,...,$H_s$
of $G$ with linearly dependent dimension data.
\end{theorem}

The organization of this paper is as follows. In Section \ref{S:root systems}, after recalling
the definitions of root datum and root system, we define root system in a lattice. In Proposition
\ref{P:maximal root system} we show that a given lattice contains a unique maximal root system in it.
In Section \ref{S:characters}, given a compact Lie group $G$ with a biinvariant Riemannian metric
$m$, we define root systems on $T$ and a maximal one $\Psi_{T}$ among them. For each reduced root system
$\Phi$ on $T$ and a finite group $W$ acting on $T$ satisfying some condition, we define a character
$F_{\Phi,W}$. Moreover, we discuss properties of these characters $\{F_{\Phi,W}\}$. Most importantly,
in Propositions \ref{P: ST support} and \ref{P:group-root system}, we reduce the dimension datum problem
and the linear dependence problem to compare these characters and getting linear relations among them.
In Section \ref{S:conjugacy}, we discuss the relations between several finite groups $\Gamma^{\circ}$,
$\Gamma$, $W_{\Psi'_{T}}$, $W_{\Psi_{T}}$ and $\Aut(\Psi_{T})$. In particualr we prove that
$W_{\Psi'_{T}}\subset\Gamma^{\circ}$. Moreover, we discuss the connection between conjugacy
relations of root systems of subgroups with regard to some of these finite groups and relations of the
subgroups. In Section \ref{S:formulation of questions}, we formulate two questions in terms of characters of
reduced sub-root systems of a given root system. If the finite group $\Gamma^{\circ}$ contains $W_{\Psi_{T}}$,
then these two questions are equivalent to the dimension datum problem and the linear dependence problem,
respectively. In Section \ref{S:sub-root systems}, we discuss the classification of reduced sub-root systems of a
given irreducible root system. In Section \ref{S:leading terms}, given an exceptional irreducible root system
$\Psi_0$, we give the formulas of the weights $\{2\delta'_{\Phi}|\ \Phi\subset\Psi_0\}$.
In Section \ref{S:dimension-equal}, given a root system $\Psi$, we classify reduced sub-root systems
of $\Psi$ with equal characters $F_{\Phi,\Aut(\Psi)}$. This solves Question \ref{Q:equal-character}.
In Section \ref{S:dimension-dependent}, given a root system $\Psi$, we classify reduced sub-root
systems of $\Psi$ with linearly dependent characters $F_{\Phi,W(\Psi)}$. This solves Question
\ref{Q:dependent-character}. In Section \ref{S:Gamma0}, given a root system $\Psi'$ and a finite group $W$
containing $W_{\Psi}$, we give examples of a compact connected simple Lie group $G$ and a connected closed
torus $T$ in  $G$ satisfying that: $\dim T=\rank\Psi'$, $\Psi'\subset\Psi_{T}$ and $\Psi'$ is stable under
$\Gamma^{\circ}$, $\Gamma^{\circ}=W$ as groups acting on $\Psi'$ and each reduced sub-root system $\Phi$ of
$\Psi'$ equals the root system of a closed connected subgprup $H$ with $T$ a maximal torus of $H$.
In Section \ref{S:Irreducible}, we give interesting examples of irreducible subgroups of $\U(n)$ with
equal dimension data or linearly dependent dimension data, and show that connected irreducible subgroups of
$\U(n)$ actually have linearly independent dimension data.

\smallskip

\noindent{\it Notation and conventions.} Given a compact Lie group $G$,
\begin{enumerate}
\item denote by $G_0$ the subgroup of connected component of $G$ containing the identity element and
$[G,G]$ the commutator subgroup.
\item Let $\frg_0=\Lie G$ be the Lie algebra of $G$.
\item Write $G^{\natural}$ for the set of conjugacy classes in $G$;
\item Denote by $\widehat{G}$ the set of equivalence classes of irreducible finite-dimensional complex
linear representations of $G$.
\item Let $\mu_{G}$ be the unique Haar measure on $G$ with $\int_{G} 1\mu_{G}=1$.
\item Write $V^{G}$ for the subspace of $G$-invariant vectors in a complex linear representation
$V$ of $G$.
\item Given a maximal torus $S$ of $G$, let $W_{G}=N_{G}(S)/C_{G}(S)$ be the Weyl group of $G$. Then the
quotient space $S/W_{G}$ is a connected component of $G^{\natural}$. Moreover, in the case that $G$ is
connected, they are identical.
\item Given a closed subgroup $H$ of $G$, denote by $\st_{H}$ the push-measure on $G^{\natural}$ of
$\mu_{H}$ under the composition map $H\hookrightarrow G\longrightarrow G^{\natural}$. It is called the
{\it Sato-Tate measure} of the subgroup $H$.
\item For any complex linear representation $\rho$ of $G$, $\mathscr{D}_{H}(\rho)=\int_{H}\chi_{\rho}(x)
\mu_{H}=\st_{H}(\chi_{\rho}^{\natural})$, where $\chi_{\rho}^{\natural}$ is a continuous function on
$G^{\natural}$ induced from the character function $\chi_{\rho}$ of $\rho$. In this way,
the dimension datum $\mathscr{D}_{H}$ and the Sato-Tate measure $\st_{H}$ determine each other.
\end{enumerate}

Given a compact Lie group $G$ and a closed connected torus $T$, let $\Phi(G,T)$ be the set of
non-zero weights of $\frg=\frg_0\otimes_{\bbR}\bbC$ as a complex linear representation of $T$.
In Section \ref{S:root systems}, we explain that in some cases $\Phi(G,T)$ is a root system.

Given a lattice $L$,
\begin{enumerate}
\item write $L_{\mathbb{Q}}=L\otimes_{\mathbb{Z}}\mathbb{Q}$ for the rational vector space
generated by $L$.
\item Let $\bbQ[L]$ be the group ring of $L$ over the field $\bbQ$.
\item For any element $\lambda\in L$, denote by $[\lambda]$ the corresponding element in
$\bbQ[L]$. Then, for any $\lambda,\mu\in L$ and $n\in\bbZ$, $[\lambda][\mu]=[\lambda+\mu]$
and $[\lambda]^{n}=[n\lambda]$.
\end{enumerate}

Given an abstract root system $\Phi$ (cf. Definition \ref{D:abstract root system}),
\begin{enumerate}
\item denote by $\bbZ\Phi$ the root lattice spanned by $\Phi$ and by $\bbQ\Phi=\bbZ\Phi\otimes_{\bbZ}\bbQ$
the rational vector space spanned by $\Phi$.
\item Choosing a positive definite inner product $(\cdot,\cdot)_{m}$ on $\bbQ\Phi$ inducing the cusp product
on $\Phi$ (which always exists and is unique if and only if $\Phi$ is irreducible), write
\[\Lambda_{\Phi}=\big\{\lambda\in\bbQ\Phi|\ \frac{2(\lambda,\alpha)_{m}}{(\alpha,\alpha)_{m}}\in\bbZ,\
\forall\alpha\in\Phi\big\}\] for the lattice of integral weights. One can show that $\Lambda_{\Phi}$ does not
depend on the choice of the inner product $m$ on $\bbQ\Phi$ if it induces the cusp product on $\Phi$.
\item Let $\Phi^{\circ}$ be the subset of short vectors in $\Phi$: for any $\alpha\in\Phi$,
$\alpha\in\Phi^{\circ}$ if and only if for any other $\beta\in\Phi$, $|\beta|\geq|\alpha|$ or
$\langle\alpha,\beta\rangle=0$. One can show that $\Phi^{\circ}$ is a sub-root system of $\Phi$.
\item If $V$ is a rational (or real) vector space with a positive definite inner product $(\cdot,\cdot)_{m}$
containing $\Phi$, let \[\Lambda_{\Phi}(V)=\big\{\lambda\in V|\ \frac{2(\lambda,\alpha)_{m}}{(\alpha,\alpha)_{m}}
\in\bbZ,\ \forall\alpha\in\Phi\big\}.\] Then $\Lambda_{\Phi}(V)$ is the direct sum of $\Lambda_{\Phi}$ and
the linear subspace of vectors in $V$ orthogonal to all vectors in $\Phi$.
\item Given a subset $X$ of $\Phi$, we call the minimal sub-root system of $\Phi$ containing $X$ (which exists and
is unique) the sub-root system generated by $X$, and denote it by $\langle X\rangle$.
\end{enumerate}

Write \[T_{k}=\{\diag\{z_1,z_2,\dots,z_{k}\}|\ |z_1|=|z_2|=\cdots=|z_{k}|=1\}.\]

We follow Bourbaki numbering to order the simple roots (cf. \cite{Bourbaki}, Pages 265-300). Write
$\omega_{i}$ for the $i$-th fundamental weight.

\noindent{\bf Acknowledgements.} This article is a sequel of \cite{An-Yu-Yu}. The author is grateful to Jinpeng
An and Jiu-Kang Yu for the collaboration. He is also grateful to Brent Doran, Richard Pink, Gopal Prasad,
Peter Sarnak and Jiu-Kang for helpful discussions and suggestions.

\section{Root system in a lattice}\label{S:root systems}

In this section, after recalling the definitions of root datum and root system in a Euclidean vector space,
we define root system in a lattice. Moreover, given a lattice, we show that it possesses a unique root
system in it which contains all other root systems in it. Here lattice means a finite rank free abelian
group with a positive definite inner product. Our discussion mostly follows \cite{Bourbaki}, \cite{Knapp}
and \cite{Springer}. However there exist minor differences between our definitions and definitions in each
of them. For example, our definition of abstract root system is different from that in \cite{Knapp}, and
we allow our root systems to be neither reduced nor semisimple, in contrast to all of the above references.


Defined as in \cite{Springer}, a root datum is a quadruple $(X,R,X^{\ast},R^{\ast})$ together
with some additional structure (duality and the root-coroot correspondence).
\begin{definition}\label{D:root datum}
A root datum consists of a quadruple $(X,R,X^{\ast},R^{\ast})$, where
\begin{enumerate}
\item  $X$ and $X^{\ast}$ are free abelian groups of finite rank together with a perfect pairing between
them with values in $\bbZ$ which we denote by $(,)$ (in other words, each is identified with the dual
lattice of the other).
\item $R$ is a finite subset of $X$ and $R^{\ast}$ is a finite subset of $X^{\ast}$ and there is a
bijection from $R$ onto $R^{\ast}$, denoted by $\alpha\mapsto\alpha^{\ast}$.
\item For each $\alpha$, $(\alpha,\alpha^{\ast})=2$.
\item For each $\alpha$, the map taking $x$ to $x-(x,\alpha^{\ast})\alpha$ induces an automorphism of the root
datum (in other words it maps $R$ to $R$ and the induced action on $X^{\ast}$ maps $R^{\ast}$ to $R^{\ast}$).
\end{enumerate}

Two root data $(X_1,R_1,X_1^{\ast},R_1^{\ast})$ and $(X_2,R_2,X_2^{\ast},R_2^{\ast})$ are called isomorphic
if there exists a linear isomorphism $f: X_1\rightarrow X_2$ such that $f(R_1)=f(R_2)$,
$(f^{\ast})^{-1}(R_1^{\ast})=R_2^{\ast}$ and $f(\alpha)^{\ast}=((f^{\ast})^{-1})(\alpha^{\ast})$ for any
$\alpha\in R_1$. Here, $f^{\ast}: X_2^{\ast}\rightarrow X_1^{\ast}$ is the dual linear map of $f$ and
$(f^{\ast})^{-1}$ is its inverse.
\end{definition}

The elements of $R$ are called {\it roots}, and the elements of $R^{\ast}$ are called {\it coroots}. If $R$
does not contain $2\alpha$ for any $\alpha$ in $R$, then the root datum is called {\it reduced}. To a connected
compact Lie group $G$ with a maximal torus $S$, we can associate a reduced root datum $\RD(G,S)$, whose
isomorphism class $\RD(G)$ depends only on $G$. Two connected compact Lie groups $G,G'$ are isomorphic if and
only if $\RD(G)$ and $\RD(G')$ are isomorphic.

\smallskip

A root system in a Euclidean vector space is a finite set with some additional structure.
\begin{definition}\label{D:root system}
Let $V$ be a finite-dimensional Euclidean vector space, with an inner product $m$ denoted by
$(\cdot,\cdot)_{m}$. A root system in $V$ is a finite set $\Phi$ of non-zero vectors (called roots) in $V$ that
satisfy the following conditions:
\begin{enumerate}
\item For any two roots $\alpha$ and $\beta$, the element $\beta-\frac{2(\beta,\alpha)_{m}}
{(\alpha,\alpha)_{m}}\alpha\in\Phi$.
\item (\textbf{Integrality}) For any two roots $\alpha$ and $\beta$, the number $\frac{2(\beta,\alpha)_{m}}
{(\alpha,\alpha)_{m}}$ is an integer.
\end{enumerate}

Given a root system $\Phi$ in a Euclidean vector space $V$ with an inner product $m$, we call
$s_{\alpha}: V\rightarrow V$ defined by
\[s_{\alpha}(\lambda)=\lambda-\frac{2(\lambda,\alpha)_{m}}{(\alpha,\alpha)_{m}}\alpha,\ \forall\lambda\in V\]
the reflection corresponding to the root $\alpha$.
\end{definition}

Moreover, a root system $\Phi$ in a Euclidean vector space $V$ is called {\it semisimple} if the roots
span $V$. It is called {\it reduced} if the only scalar multiples of a root $x\in\Phi$ that belong to $\Phi$
are $x$  and $-x$.

To a connected compact semisimple Lie group $G$ with a maximal torus $S$, we can associate a reduced
semisimple root system $\R(G,S)$, whose isomorphism class $\RD(G)$ depends only on $G$. Two connected
compact semisimple Lie groups $G,G'$ having isomorphic universal covers if and only if $\R(G)$ and $\R(G')$
are isomorphic.

\smallskip

There are several different notions of lattice in the literature.

\begin{definition}\label{D:lattice}
A lattice is a finite rank free abelian group with a positive definite inner product.
\end{definition}

Now we define root systems in a lattice.
\begin{definition}\label{D:root system in a lattice}
Let $L$ be a lattice with a positive definite inner product $m$ denoted by $(\cdot,\cdot)_{m}$. A root
system in $L$ is a finite set $\Phi$ of non-zero vectors (called roots) in $L$ that satisfy the following
conditions:
\begin{enumerate}
\item For any two roots $\alpha$ and $\beta$, the element $\beta-\frac{2(\beta,\alpha)_{m}}{(\alpha,\alpha)_{m}}
\alpha\in\Phi$.
\item (\textbf{Strong integrality}) For any root $\alpha$ and any vector $\lambda\in L$, the number
$\frac{2(\lambda,\alpha)_{m}}{(\alpha,\alpha)_{m}}$ is an integer.
\end{enumerate}


Given a root system $\Phi$ in a lattice $L$ with an inner product $m$, we call $s_{\alpha}: L\rightarrow L$
defined by \[s_{\alpha}(\lambda)=\lambda-\frac{2(\lambda,\alpha)_{m}}{(\alpha,\alpha)_{m}}\alpha,\
\forall\lambda\in L\] the reflection corresponding to the root $\alpha$.
\end{definition}

Moreover, a root system $\Phi$ in a lattice $L$ is called {\it semisimple} if the roots span
$V=L\otimes_{\bbZ}\bbR$. It is called {\it reduced} if the only scalar multiples of a root $x\in\Phi$
that belong to $\Phi$ are $x$  and $-x$.

If $\Phi\subset L$ is a root system in $L$, then it is a root system in $V=L\otimes_{\bbZ}\bbR$.
Conversely, if $\Phi\subset L$ is a root system in $V$, it is not necessarily a root system in
$L$. For example, let $L=\bbZ e_1$ be a rank $1$ lattice with $(e_1,e_1)>0$ and $V=\bbR e_1$. Then for
any $k\in\bbZ_{>0}$, $\{\pm{k}e_1\}$ is a root system in $V$. However $\{\pm{k}e_1\}$ is a root system
in $L$ only if $k=1$ or $2$. This example indicates that strong integrality is strictly stronger than
integrality.

\begin{definition}\label{D:abstract root system}
Given a root system $\Phi$ in a Euclidean vector space or a lattice with a positive definite inner
product $m$, we call \[\langle y,x\rangle=\frac{2(x,y)_{m}}{(x,x)_{m}}\] the cusp product on $\Phi$.
Overlooking the Eulidean space or the lattice containg $\Phi$, we call $\Phi$ with the cusp product on
it an abstract root system.
\end{definition}

Different with the above definiton, an abstract root system as defined in the book \cite{Knapp} is simply
a semisimple root system in a Euclidean vector space in our sense.

\begin{proposition}\label{P:cusp product}
Given an abstract root system $\Phi$, for two roots $\alpha\in\Phi$ and $\beta\in\Phi$, let
$k_1=\max\{k\in\bbZ|\ \beta+k\alpha\in\Phi\}$ and $k_2=\max\{k\in\bbZ|\ \beta-k\alpha\in\Phi\}$.
Then $\langle\beta,\alpha\rangle=k_2-k_1$.
\end{proposition}
\begin{proof}
Applying the reflection $s_{\alpha}$, we get
$\langle\beta+k_1\alpha,\alpha\rangle+\langle\beta-k_2\alpha,\alpha\rangle=0$. Thus
$\langle\beta,\alpha\rangle=k_2-k_1$ follows from $\langle\alpha,\alpha\rangle=2$.
\end{proof}


\begin{definition}\label{D:irreducible root system}
An abstract root system $\Phi$ with a cusp product $\langle\cdot,\cdot\rangle$ is called irreducible if
there exist no non-trivial disjoint unions $\Phi=\Phi_1\bigcup\Phi_2$ such that
$\langle\alpha_1,\alpha_2\rangle=0$ for any $\alpha_1\in\Phi_1$ and $\alpha_2\in\Phi_2$. A root system
$\Phi$ in a Euclidean vector space or in a lattice is called irreducible if it is semisimple and irreducible
as an abstract root system.
\end{definition}

\begin{definition}\label{D:Automorphism group}
Given a root system $\Phi$ in a Euclidean vector space $V$ with a positive definite inner
product $m$, denote by $\Aut(\Phi,m)$ the group of linear isomorphisms of $V$ stabilizing $\Phi$ and
preserving the inner product $m$, and by $W_{\Phi}$ the group of linear isomorphisms of $V$ generated
by $\{s_{\alpha}|\ \alpha\in\Phi\}$.

Given a root system $\Phi$ in a lattice $L$ with a positive definite
inner product $m$, write $\Aut(\Phi,m)$ for the group of linear isomorphisms of $L$ stabilizing $\Phi$
and preserving the inner product $m$, and $W_{\Phi}$ for the group of linear isomorphisms of $L$ generated
by $\{s_{\alpha}|\ \alpha\in\Phi\}$.

Given an abstract root system $\Phi$ with a cusp product
$\langle\cdot,\cdot\rangle$, let $\Aut(\Phi)$ be the group of permutations on $\Phi$ preserving its cusp
product, and by $W_{\Phi}$ the group of permutations on $\Phi$ generated by $\{s_{\alpha}|\ \alpha\in\Phi\}$.
\end{definition}

By definitions we have \[W_{\Phi}\subset\Aut(\Phi,m),\] and the restriction on $\Phi$ gives a homomorphism
\[\pi: \Aut(\Phi,m)\rightarrow\Aut(\Phi).\]

Given a root system $\Phi$ in a Euclidean vector space $V$, we have the following
characterization of lattices $L$ in $V$ such that $\Phi$ is also a root system in $L$.

\begin{proposition}\label{P:root system-lattice}
Given a Euclidean vector space $V$ with a positive definite inner product $m$ denoted by $(\cdot,\cdot)$,
let $\Phi$ be a root system in $V$. Then for any lattice $L\subset V$, $\Phi$ is a root system in $L$
if and only if $\bbZ\Phi\subset L\subset\Lambda_{\Phi}(V)$.
\end{proposition}

\begin{proof}
Suppose $\Phi$ is a root system in $L$. Since $\Phi\subset L$, $\bbZ\Phi\subset L$. On the other hand,
by the condition strong integrality in Definition \ref{D:root system in a lattice},
$L\subset\Lambda_{\Phi}(V)$. Hence $\bbZ\Phi\subset L\subset\Lambda_{\Phi}(V)$.

Suppose $\bbZ\Phi\subset L\subset\Lambda_{\Phi}(V)$. Then $\Phi\subset L$ and strong integrality holds.
On the other hand, since $\Phi$ is a root system in $V$, it is stable under the reflections
$\{s_{x}|\ x\in\Phi\}$. Hence $\Phi$ is a root system in $L$.
\end{proof}

Recall that a root system $\Phi$ in a Euclidean space (or in a lattice) with an inner product $m$ is called
{\it simply laced} if it is irreducible and $(\alpha,\alpha)_{m}=(\beta,\beta)_{m}$ for any $\alpha,\beta\in\Phi$.
Similarly, an abstrct root system $\Phi$ is called simply laced if it is irreducible and its cusp product takes
values in the set $\{0,1,-1\}$.



An embedding (resp. isomorphism) of root systems in lattices
\[f: (L,\Phi,m)\longrightarrow (L',\Phi',m')\] is a $\bbZ$-linear bijection $f: L\rightarrow L'$ which is
an isometry with respect to $(m,m')$ and satisfying that $f(\Phi)\subset\Phi'$ (resp. $f(\Phi)=\Phi'$). If
$L = L'$, $m=m'$ and $f$ is the identity, we simply say that $\Phi$ is a sub-root system of $\Phi'$.
Similarly, we define embeddings and sub-root systems for root systems in Euclidean vector spaces.

We remark that in the literature, root systems in a Euclidean vector space are often required to be semisimple
and/or reduced (e.g. \cite{Humphreys} and \cite{Knapp}). We require neither. And it is essential for us to
include the inner product in our definition of root systems in a Euclidean vector space or in a lattice.

For simplicity, we abbreviate as root system for a root system in a Euclidean vector space, a root system
in a lattice or an abstract root system simply when it is clear from context which is intended. A root system
in a lattice $(L,\Phi,m)$ (or a root system in a Euclidean vector space $(V,\Phi,m)$) will be denoted simply
by $\Phi$ in the case that the lattice $L$ and the inner product $m$ (or the Euclidean vector space $V$ and
the inner product $m$) are clear from the context.

\begin{proposition}\label{P:maximal root system}
Given a lattice $L$ with a pisitve definite inner product $m$ denoted by $(\cdot,\cdot)_{m}$, there exists a root system
$\Psi_{L}$ in $L$ containing all other root systems in $L$. The root system $\Psi_{L}$ is uniquely determined
by this characterization.
\end{proposition}

\begin{proof}
Define \[\Psi_{L}=\big\{0\neq\alpha\in L|\ \frac{2(\lambda,\alpha)_{m}}{(\alpha,\alpha)_{m}}\in\bbZ,
\forall\lambda\in\Lambda\big\}.\] We show that $\Psi_{L}$ satisfies the desired conclusion.

First we show that $\Psi_{L}$ is a root system in $L$. Obviously $\Psi_{L}$ satisfies the condition
$(2)$ in Definition \ref{D:root system in a lattice}. For any $\alpha,\beta\in\Psi$ and any
$\lambda\in L$, \begin{eqnarray*}&&\frac{2(\lambda,s_{\alpha}(\beta))_{m}}
{(s_{\alpha}(\beta),s_{\alpha}(\beta))_{m}}
\\&=&\frac{2(s_{\alpha}(\lambda),\beta)_{m}}{(\beta,\beta)_{m}}
\\&=&\frac{2(\lambda-\frac{2(\lambda,\alpha)_{m}}{(\alpha,\alpha)_{m}}\alpha,\beta)_{m}}{(\beta,\beta)_{m}}
\\&=&\frac{2(\lambda,\beta)_{m}}{(\beta,\beta)_{m}}-
\frac{2(\lambda,\alpha)_{m}}{(\alpha,\alpha)_{m}}\frac{2(\alpha,\beta)_{m}}{(\beta,\beta)_{m}}.
\end{eqnarray*}
Since $\alpha,\beta\in\Psi_{L}$ and $\alpha,\lambda\in L$, the number in the last line is an integer.
Hence the vector $s_{\alpha}(\beta)$ is contained in $\Psi_{L}$. This proves the condition $(1)$.
Therefore $\Psi_{L}$ is a root system in $L$.

By the condition strong integrality in Definition \ref{D:root system in a lattice}, $\Psi_{L}$ contains all other
root systems in $L$. This indicates that $\Psi_{L}$ is unique with this characterization.
\end{proof}

From Proposition \ref{P:maximal root system} and its proof, we see that the  condition strong integrality
really matters. In the following proposition, we give some examples of root systems in lattices. These
root systems will be used in the proof of Proposition \ref{P:construction-simple}.





\begin{proposition}\label{P:root system}
Let $G$ be a connected compact Lie group with a biinvariant Riemannian metric $m$, $T$ a connected closed
torus in $G$, $L=\Hom(T,\U(1))$ the weight lattice, and $m$ the induced positive definite inner product
on $L$. Then the set $\Phi(G,T)$ of non-zero $T$-weights in $\frg=\frg_0\otimes_{\bbR}\bbC$ is a root
system in $L$ in the case that
\begin{enumerate}
\item $T$ is a maximal torus of $G$,
\item there exists an involutive automorphism $\theta$ of $\frg_0$ and $\Lie T$ is a maximal abelian
subspace of $\frg_0^{-\theta}=\{X\in\frg_0|\theta(X)=-X\}$, or
\item $\frg_0=\mathfrak{so(8)}$ and $\Lie T$ is conjugate to a Cartan subalgebra of a (unique up to conjugacy)
subalgebra of $\frg_0$ isomorphic to $\frg_2$.
\end{enumerate}
\end{proposition}

\begin{proof}
We prove $(1)$ first. It is well-known that $\Phi(G,T)$ is a root system in $L\otimes_{\bbZ}\bbR=
\Hom_{\bbR}(\Lie T,i\bbR)$. Thus it satisfies the condition $(1)$ in Definition \ref{D:root system in a lattice}.
We need to show that it satisfies the condition $(2)$. Write $r=\dim T=\rank G$.
For each $\alpha\in\Phi(G,T)$, there exists a homomorphism (cf. \cite{Knapp}, Pages 143-149)
\[f_{\alpha}:\SU(2)\times\U(1)^{r-1}\longrightarrow G\] such that $f_{\alpha}(\U(1)^{r-1})=(\ker\alpha)_0$,
$f_{\alpha}(T'\times \U(1)^{r-1})=T$, and the complexified Lie algebra of $f_{\alpha}(\SU(2))$ is the
$\mathfrak{sl}_2$ subalgebra corresponding to the root $\alpha$. Here $T'$ is the subgroup of $\SU(2)$ of
diagonal matrices. We have $T'\cong\U(1)$ and this gives a linear character $\alpha_0$ of $T'$. We can normalize
$\alpha_0$ such that $\alpha\circ (f_{\alpha}|_{T'})$ equals $2\alpha_0$ . Let
\[W=N_{G}(T)/C_{G}(T),\] \[n_{\alpha}=f_{\alpha}(\left(\begin{array}{cc}0&1\\-1&0\\\end{array}\right))\]
and \[s_{\alpha}=\Ad(n_{\alpha})\in W.\] For any $\lambda\in L$, let
$k\alpha_0=\lambda\circ(f_{\alpha}|_{T'})$. Then $k\in\bbZ$. We prove that $(2\lambda-k\alpha,\alpha)_{m}=0$.
This is equivalent to $s_{\alpha}(2\lambda-k\alpha)=2\lambda-k\alpha$, and also equivalent to
\begin{equation}\label{Eq:H} (2\lambda-k\alpha)(s_{\alpha}(H))=(2\lambda-k\alpha)(H)\end{equation} for any
$H\in\Lie T$. Write $H$ as $H=H_1+H_2$ where $H_1\in\Lie (f_{\alpha}(T'))$ and $H_2\in\Lie(\ker\alpha)$. Then we
have $s_{\alpha}(H_1)=-H_1$ and $s_{\alpha}(H_2)=H_2$. By this Equation (\ref{Eq:H}) is equivalent to
$(2\lambda-k\alpha)(H_1)=0$. This follows from $k\alpha_0=\lambda\circ(f_{\alpha}|_{T'})$ and
$\alpha\circ(f_{\alpha}|_{T'})=2\alpha_0$. Thus \[\frac{2(\lambda,\alpha)_{m}}{(\alpha,\alpha)_{m}}
=\frac{2(\frac{k}{2}\alpha,\alpha)_{m}}{(\alpha,\alpha)_{m}}=k\in\bbZ.\] Hence $\Phi(G,T)$ satisfies the
condition $(2)$ in Definition \ref{D:root system in a lattice}.

The proof for $(2)$ is similar as the proof for $(1)$. In \cite{Knapp}, Pages 379-380, it is proven that
$\Phi(G,T)$ is a root system in $L\otimes_{\bbZ}\bbR=\Hom_{\bbR}(\Lie T,i\bbR)$. Thus it satisfies the condition
$(1)$ in Definition \ref{D:root system in a lattice}. Similar as in the above proof, there is an reflection
$s_{\alpha}\in N_{G}(T)/C_{G}(T)$. Using it we are able to show that $\Phi(G,T)$ satisfies the condition $(2)$ in
Definition \ref{D:root system in a lattice}.

For $(3)$, let $H$ denote a closed subgroup of $H$ isomorphic to $\G_2$ and with $T$ a maximal torus of $H$.
Then we have $\Phi(G,T)=\Phi(H,T)$. In this way $(3)$ follows from $(1)$.
\end{proof}

In the thory of integral lattice, there is a notion of root system of a lattice, which is different with
our discussion here.

\begin{remark}\label{R:lattice-root system}
Given a lattice $L$ with a positive definite inner product $m$ denoted by $(\cdot,\cdot)$, it is called
an integral lattice if $(\lambda_1,\lambda_2)\in\bbZ$ for any $\lambda_1,\lambda_2\in L$. In this case
the set $X=\{\lambda\in L|\ (\lambda,\lambda)=1\textrm{ or }2\}$ is called the root system of $L$.
In general $X$ is a proper subset of $\Psi_{L}$.
\end{remark}

\begin{remark}\label{R:RD-AR}
Root systems (in Euclidean vector spaces) classify compact connected Lie groups up to isogeny, however
root data classify them up to isomorphism. Hence the difference between root datum and root system in a
Euclidean vector space is clear. In another direction, we think that root datum and root system in a
lattice should be equivalent objects. That is to say, there should exist a construction from a root datum
to a root system in a lattice, and vice versa.
\end{remark}

\section{The Larsen-Pink method}\label{S:characters}

In this section, given a compact Lie group $G$ with a biinvariant Riemannian metric $m$ and a closed
connected torus $T$ in $G$, we define root systems on $T$ and get two specific
root systems $\Psi_{T}$ and $\Psi'_{T}$. Moreover, we define a character $F_{\Phi,W}$ for each root
system $\Phi$ on $T$ and a finite group $W$ acting on the Lie algebra of $T$ and satisfying some
condition (cf. Definition \ref{D:characters}). We study properties of the character $F_{\Phi,W}$ and
their relation with dimension data. In this way, we reduce the dimension datum problem and the linear
dependence problem to comparing these characters and getting linear relations among them.

Let $G$ be a compact Lie group with a biinvariant Riemannian metric $m$. For a closed connected torus
$T$ contained in $G$, let \[\Lambda_{T}=\Hom(T,U(1))\] be the integral weight lattice of $T$
and $\Lambda_{\bbQ,T}=\Lambda_{T}\otimes_{\bbZ}\bbQ$ the $\bbQ$-weight space of $T$. By restriction the Riemannian
metric $m$ induces an inner product on the Lie algebra $\frt_0=\Lie T$ of $T$. Denote it by $m$ as well.
Write $(\frt_0)^{\ast}=\Hom_{\bbR}(\frt_0,\bbR)$ and $(\frt_0)_{\bbC}^{\ast}=\Hom_{\bbR}(\frt_0,\bbC)$.
Then $(\frt_0)^{\ast}\subset(\frt_0)_{\bbC}^{\ast}$ and
$(\frt_0)_{\bbC}^{\ast}=(\frt_0)^{\ast}\oplus i(\frt_0)^{\ast}$. Since $m$ is positive definite, in
particular non-degenerate, the induced linear map
\[p: \frt_0\longrightarrow(\frt_0)^{\ast}=\Hom_{\bbR}(\frt_0,\bbR)\] is a linear isomorphism.
For any $\lambda,\mu\in(\frt_0)^{\ast}$, let
\[(\lambda,\mu)_{m}=-(p^{-1}\lambda,p^{-1}\mu)_{m}.\] Then $(\cdot,\cdot)_{m}$ is a negative definite inner
product on $(\frt_0)^{\ast}$. For any $\lambda_1,\lambda_2,\mu_1,\mu_2\in(\frt_0)^{\ast}$, let
\[(\lambda_1+i\lambda_2,\mu_1+i\mu_2)_{m}=((\lambda_1,\mu_1)_{m}-(\lambda_2,\mu_2)_{m})+
i((\lambda_1,\mu_2)_{m}+(\lambda_2,\mu_1)_{m}).\] Then $(\cdot,\cdot)_{m}$ is a non-degenerate symmetric
bilinear form on $(\frt_0)_{\bbC}^{\ast}$. It is positive definite on $i(\frt_0)^{\ast}$ by restriction.
Since $\Lambda\subset\Lambda_{\bbQ}\subset i\frt_0^{\ast}$, they inherit this inner product.

\begin{definition}\label{D:Gamma-Gamma0}
Let $\Gamma$ be the group of automorphisms of $T$ preserving the Riemannian metric $m$ on it, and
let \[\Gamma^{\circ}=N_{G}(T)/C_{G}(T).\]
\end{definition}

By this definiton we have $\Gamma^{\circ}\subset\Gamma$. The group $\Gamma$ has an equivalent definition
as the group of automorphisms of the lattie $\Lambda_{T}$ preversing the inner product $m$ on $\Lambda_{T}$.
Choose a maximal torus $S$ of $G$ containing $T$ and let $W_{G}=N_{G}(S)/C_{G}(S)$ be the Weyl group of $G$.
Then $\Gamma^{\circ}$ has an equivalent definition as the image of $\Stab _{W_{G}}(T)$ in $\Aut(T)$.
Given a non-zero element $\alpha\in\Lambda_{T}$, define $s_{\alpha}:i\frt_0^{\ast}\longrightarrow
i\frt_0^{\ast}$ by \[s_{\alpha}(\lambda)=\lambda-\frac{2(\lambda,\alpha)_{m}}{(\alpha,\alpha)_{m}}\alpha,\
\forall\lambda\in i\frt_0^{\ast}.\] Then $s_{\alpha}$ is a reflection on $i\frt_0^{\ast}$.

\begin{definition}\label{D:subroot}
A subset $\Phi\subset\Lambda_{T}$ is called a root system on $T$ if for any $\alpha\in\Phi$ and
$\lambda\in\Lambda_{T}$, \[s_{\alpha}(\Phi)=\Phi\] and \[\frac{2(\lambda,\alpha)_{m}}{(\alpha,\alpha)_{m}}
\in\bbZ.\] In other words, a subset $\Phi\subset\Lambda_{T}$ is a root system on $T$ if it is a root system
in $\Lambda_{T}$.
\end{definition}

\begin{definition}\label{D:maxroot}
Let \[\Psi_{T}=\big\{0\neq\alpha\in\Lambda_{T}|\ \frac{2(\lambda,\alpha)_{m}}{(\alpha,\alpha)_{m}}\in\bbZ,\
\forall\lambda\in\Lambda_{T}\big\}.\]
\end{definition}

By Proposition \ref{P:maximal root system}, $\Psi_{T}$ is the unique maximal root system in the lattice
$\Lambda_{T}$. By the definition $\Psi_{T}$ is stable under the action of $\Gamma=\Aut(\Lambda_{T},m)$. Thus
we have \[\Gamma=\Aut(\Psi_{T},m),\] where $\Aut(\Psi_{T},m)$, the automorphism of the root system $\Psi_{T}$
in the lattice $\Lambda_{T}$, is defined in Definition \ref{D:Automorphism group}.

\begin{proposition}\label{P:Psi}
For any closed connected subgroup $H$ of $G$ with $T$ a maximal torus of $H$, the root system of $H$,
$\Phi(H)=\Phi(H,T)$ is a root system on $T$. Moreover, it is a reduced root system.
Defined as above, $\Psi_{T}$ is a root system on $T$ and it contains all root systems on $T$. In particular
$\Phi(H,T)\subset\Psi_{T}$ for any closed subgroup $H$ of $G$ with $T$ a maximal torus of $H$.
\end{proposition}

\begin{proof}
The fact that subset $\Phi(H,T)$ is a root system on $T$ is proven in Proposition \ref{P:root system}. The
fact that  the subset $\Psi_{T}$ is a root system on $T$ and contains all other root systems on $T$ is
proved in Proposition \ref{P:maximal root system}.
\end{proof}

In general the rank of $\Psi_{T}$ can be any integer between $0$ and $\dim T$. It happens that
$\rank\Psi_{T}=\dim T$ if and only if there exists a root system on $T$ of rank equal to $\dim T$. In particular
if there exists a semisimple closed subgroup $H$ of $G$ with $T$ a maximal torus of $H$, then $\rank\Psi_{T}$
is equal to $\dim T$. We choose and fix an ordering on $\Lambda_{T}$ so that we get a positive system
$(\Psi_{T})^{+}$.

\begin{remark}\label{R:compare Psi}
In \cite{Larsen-Pink}, Larsen-Pink defined a root system $\Psi$ from a dimension datum. In general this
$\Psi$ is a root system on $T$ and is strictly contained in our $\Psi_{T}$.
\end{remark}

The root system $\Psi_{T}$ depends on the torus $T$ and a biinvariant Riemannian metric $m|_{T}$. If we replace
$G$ by a group isogenous of it, then the lattice $\Lambda_{T}$ becomes another lattice isogeneous to it. The root
system $\Psi_{T}$ will change probably after this modification. On the other hand, $\Psi_{T}$ is also sensitive
with the biinvariant Riemannian metric $m|_{T}$. In a different direction, we can define another root
system $\Psi'_{T}$ which depends on $T$ and the the isogeny class of $G$. However, if we do not know $G$ and
$T$ explicitly, it is hard to determine $\Psi'_{T}$.

\begin{definition}\label{D:Psi prime}
Define $\Psi'_{T}$ as the sub-root system of $\Psi_{T}$ generated by $\{\Phi(H,T)\}$, where $H$
runs through all closed connected subgroups $H$ of $G$ with $T$ a maximal torus of $H$.
\end{definition}

It is clear that $\Psi'_{T}$ is stable under $\Gamma^{\circ}$. However it is not stable under $
W_{\Psi_{T}}$, $\Gamma$ or $\Aut(\Psi_{T})$ in general.

\begin{remark}\label{R:Psi-simple group}
With $\Psi_{T}$, we are able to consider all root systems $$\{\Phi(H,T)|\ T \textrm{ is a maximal
torus of } H\}$$ together by viewing them as sub-root systems of $\Psi_{T}$. In the case that $G$ is a
connected simple Lie group and $T$ is a maximal torus of $G$, let $\Psi_0=\Phi(G,T)$ be the root
system of $G$. Then $\Psi_0\subset\Psi_{T}$.
If $\Psi_0$ is simply laced, then for any sub-root system $\Phi$ of $\Psi_0$ there exists a closed
subgroup $H$ of $G$ with $T$ a maximal torus of $H$ such that $\Phi(H,T)=\Phi$. If $\Psi_0$ is not
simply laced, then not every sub-root system $\Phi$ of $\Psi_0$ is of the form $\Phi(H,T)=\Phi$
where $H$ is a closed subgroup of $G$ with $T$ a maximal torus of $H$.
In the case that $\Psi_0$ is of type $\A_{n}$ ($n\geq 4$), $\E_6$, $\E_7$, $\E_8$, $\F_4$ or $\G_2$, it must be that
$\Psi_{T}=\Psi_0$. However, if $\Psi_0$ is of type $\A_2$, $\B_n$, $\C_{n}$ or $\D_{n}$ ($n\geq 3$), then
$\Psi_{T}$ may be strictly larger than $\Psi_0$. Precisely to say, in the case that $\Psi_0=\A_{2}$, it is possible
that $\Psi_{T}=\G_{2}$. In the case that $\Psi_0=\B_{n}$ or $\C_{n}$, it is possible that $\Psi_{T}=\BC_{n}$. In the
case that $\Psi_0=\D_{n}$, it is possible that $\Psi_{T}=\B_{n}$, $\C_{n}$ or $\BC_{n}$.
\end{remark}

\begin{lemma}\label{L:reflection group}
Given a root system $\Psi$ and its Weyl group $W_{\Psi}$, let $X$ be a non-empty subset of $\Psi$ and
$W$ be the subgroup of $W_{\Psi}$ generated by reflections corresponding to elements in $X$.
If there exist no proper sub-root systems of $\Psi$ containing $X$, then $W=W_{\Psi}$.
\end{lemma}

\begin{proof}
Let $X'=\{\alpha\in\Psi|\ s_{\alpha}\in W\}$. For any two roots $\alpha,\beta$ contained in $X'$, since
\[s_{-\alpha}=s_{\alpha}\in W\] and \[s_{s_{\alpha}(\beta)}=s_{\alpha}s_{\beta}s_{\alpha}\in W,\] we get
$-\alpha\in X'$ and $s_{\alpha}(\beta)\in X'$. That means, $X'$ is a sub-root system of $\Psi$. On the
other hand, we have $X\subset X'$. By the condition that there exist no proper sub-root systems of $\Psi$
containing $X$, we get $X'=\Psi$. Therefore $W=W_{\Psi}$.
\end{proof}


\begin{proposition}\label{P:Gamma0-Psi'}
Given a compact Lie group $G$ and a biinvariant Riemannian metric $m$ on $G$, for any clsoed connected
torus $T$ in $G$, we have $W_{\Psi'_{T}}\subset\Gamma^{\circ}$
\end{proposition}

\begin{proof}
Recall that $\Psi'_{T}$ is the sub-root system of $\Psi_{T}$ generated by $\{\Phi(H,T)\subset\Psi_{T}\}$,
where $H$ run over closed connected subgroups of $G$ with $T$ a maximal torus of $H$. Let $X$ be the subset
of $\Psi'_{T}$ defined as: $\alpha\in X$ if and only if there exists a closed connected subgroup $H$ of $G$
with $T$ a maiximal tours of $H$ and the root system of $H$ being $\{\pm{\alpha}\}$. Let $W$ be the subgroup
of $W_{\Psi'_{T}}$ generated by reflections corresponding to elements in $X$.

We show that $W\subset\Gamma^{\circ}$ and there exist no proper sub-root systems of $\Psi'_{T}$ containing $X$.
For any $\alpha\in X$, there exists a closed connected subgroup $H$ of $G$ with $T$ a maximal torus of $H$
and the root system of $H$ being $\{\pm{\alpha}\}$. Thus there exists a finite surjection
\[p:\SU(2)\times\U(1)^{r-1}\longrightarrow H.\] Let
\[n_{\alpha}=p(\left(\begin{array}{cc}0&1\\-1&0\\\end{array}\right),1),\]
where $\left(\begin{array}{cc}0&1\\-1&0\\\end{array}\right)\in\SU(2)$ and $1\in\U(1)^{r-1}$. Then
$n_{\alpha}\in N_{G}(T)$ and $\Ad(n_{\alpha})|_{T}=s_{\alpha}$. Hence $s_{\alpha}\in\Gamma^{\circ}$.
Therefore $W\subset\Gamma^{\circ}$.
Suppose there exists a proper sub-root system $\Psi'$ of $\Psi'_{T}$ containing $X$. For any closed
connected subgroup $H$ of $G$ with $T$ a maximla torus of $H$, one has that each root of $\Phi(H,T)$ is
contained in $X$. Hence $\Phi(H,T)\subset X\subset\Psi'$. This contradicts to the condition that
the root systems $\Phi(H,T)$ generate $\Psi'_{T}$.

By the above two facts and Lemma \ref{L:reflection group} we get $W_{\Psi'_{T}}=W\subset\Gamma^{\circ}$.
\end{proof}

\begin{corollary}\label{C:Psi'-A1}
For any root $\alpha$ in $\Psi'_{T}$, there exists a closed connected subgroup $H$ of $G$ with $T$ a maximal
torus of $H$ and having root system equal to $\{\pm{\alpha}\}$. The abstract root system $\Psi'_{T}$ depends
only on the group $G$ and the connected torus $T$, not on the biinvariatn Riemannian metric $m$.
\end{corollary}
\begin{proof}
Conisdering the subset $X$ defined in the above proof for Proposition \ref{P:Gamma0-Psi'}, it is obviously
$\Gamma^{\circ}$ invariant. Since $\Gamma^{\circ}\supset W_{\Psi'_{T}}$ as the above proof showed, $X$ is a
union of some $W_{\Psi'_{T}}$ on $\Psi'_{T}$. Thus $X$ is a sub-root system of $\Psi'_{T}$. On the other hand,
the above proof showed that $X$ generates $\Psi'_{T}$. Therefore $X=\Psi'_{T}$. This is just the first statement
in the conclusion of the corollary. Moreover, by Proposition \ref{P:cusp product} the cusp product on $\Psi'_{T}$
is determined. Therefore, the abstract root system $\Psi'_{T}$ is determined by $G$ and $T$, indpendent with
the biinvariant Riemannian metric $m$.
\end{proof}

\begin{remark}\label{R:Psi'-Phi}
Suppose the abstract root system $\Psi'_{T}$ is given, an interesting question is to determine the sub-root
systems $\Phi$ of $\Psi'_{T}$ so that there exists a closed connected subgroup $H$ of $G$ with $T$ a maximal
torus of $H$ and having root system equal to $\Phi$. From the above corollary, we know that each rank-$1$ sub-root
system of $\Psi'_{T}$ is the root system of a closed connected subgroup. For sub-root systems of higher rank, it is
unclear to the author which sub-root systems of $\Psi'_{T}$ are root systems of closed connected subgroups. Moreover,
is it possible that any sub-root system of $\Psi'_{T}$ of rank larger than one is not the root system of a closed
connected subgroup?
\end{remark}

\smallskip

Given a {\it reduced root system} $\Phi$ on $T$, let \[\delta_{\Phi}=\frac{1}{2}\sum_{\alpha\in\Phi^{+}}
\alpha.\] Let \[\delta'_{\Phi}\] the unique dominant weight with respect to $(\Psi_{T})^{+}$ in the orbit
$W_{\Psi_{T}}\delta_{\Phi}=\{\gamma\delta_{\Phi}:\gamma\in W_{\Psi_{T}}\}$. Here $W_{\Psi_{T}}$ is the
Weyl group of the root system $\Psi_{T}$ and $\Phi^{+}=\Phi\cap(\Psi_{T})^{+}$.

\begin{definition}\label{D:characters}
Define a character on $T$, \[A_{\Phi}=\sum_{w\in  W_{\Phi}}\sign(w)[\delta_{\Phi}-w\delta_{\Phi}].\]

Moreover, for any finite group $W$ between $W_{\Phi}$ and $\Aut(\Lambda_{\bbQ,T},m)$, define another
character on $T$ \[F_{\Phi,W}=\frac{1}{|W|}\sum_{\gamma\in W}\gamma(A_{\Phi}).\]
\end{definition}

\begin{definition}\label{D:chi-lambda}
Given a finite subgroup $W$ of $\Aut(\Lambda_{\bbQ,T},m)$ and an integral weight $\lambda\in\Lambda_{T}$,
let \[\chi^{\ast}_{\lambda,W}=\frac{1}{|W|}\sum_{\gamma\in W}[\gamma\lambda]\in\bbQ[\Lambda_{\bbQ,T}].\]
\end{definition}

The character $\chi^{\ast}_{\lambda,W}$ depends only on the orbit $W\lambda=\{\gamma\lambda:\gamma\in W\}$.
If $W\subset\Aut(\Lambda_{T},m)$, then $\chi^{\ast}_{\lambda,W}\in\bbQ[\Lambda_{T}]$ for any
$\lambda\in\Lambda_{T}$. Moreover, the set $\big\{\chi^{\ast}_{\lambda,W}|\lambda\in\Lambda_{T}^{1}\big\}$ is a
basis of $\bbQ[\Lambda_{T}]^{W}$, where $\Lambda_{T}^{1}$ is a set of representatives of $W$ orbits in $\Lambda_{T}$.
In the case that $\rank\Psi_{T}=\dim T$ and $W=W_{\Psi_{T}}$, we can choose $\Lambda_{T}^{1}=\Lambda_{T}^{+}$. Here,
$\Lambda_{T}^{+}$ is the set of dominant integral weights in $\Lambda_{T}$ with respect to $(\Psi_{T})^{+}$.

The characters $F_{\Phi,W}$ have the following property.
\begin{proposition}\label{P:characters}
Each character $F_{\Phi,W}$ is a linear combination of $\{\chi^{\ast}_{\lambda,W}|\ \lambda\in
\Lambda_{\mathbb{Q},T}\}$ with integer coefficients and the constant term of $F_{\Phi,W}$ is $1$. With
$\{\chi^{\ast}_{\lambda,W}|\ \lambda\in\Lambda_{\mathbb{Q},T}\}$ as a basis, the minimal length terms of
$F_{\Phi,W}-1$ are of the form $\{c_{\alpha}\cdot\chi^{\ast}_{\alpha,W}|\ c_{\alpha}\in\bbZ,\alpha
\in\Phi^{\circ}\}$ and the unique longest term of $F_{\Phi,W}$ is of the form
$c\cdot\chi^{\ast}_{2\delta_{\Phi},W}$ where $c=\pm{1}$.
\end{proposition}

\begin{proof}
Since $W$ contains $W_{\Phi}$, \begin{eqnarray*}&&F_{\Phi,W}\\&=&\frac{1}{|W|}\sum_{\gamma\in W}
\gamma(A_{\Phi})\\&=&\frac{1}{|W|}\sum_{\gamma\in W}\sum_{w\in W_{\Phi}}\sign(w)(\gamma[\delta_{\Phi}-
w\delta_{\Phi}])\\&=&\sum_{w\in W_{\Phi}}\sign(w)\chi^{\ast}_{\delta_{\Phi}-w\delta_{\Phi},W}.
\end{eqnarray*}

One has $|\delta_{\Phi}-w\delta_{\Phi}|^{2}=\langle 2\delta_{\Phi},\delta_{\Phi}-w\delta_{\Phi}\rangle$
and \[\delta_{\Phi}-w\delta_{\Phi}=\frac{1}{2}\sum_{\alpha\in\Phi^{+}}\alpha-\frac{1}{2}\sum_{\alpha\in
\Phi^{+}}w\alpha=\sum_{\alpha\in\Phi^{+}\cap w^{-1}\Phi^{-}}\alpha.\] Then for $w\neq 1$,
$\delta_{\Phi}-w\delta_{\Phi}$ is of shortest length exactly when $w=s_{\alpha}$ for $\alpha$, a short
root of minimal length; and it is of longest length exactly when $w\Phi^{+}=\Phi^{-}$, i.e., $w=w_0$
is the unique longest element in $W$.
\end{proof}

If $W_{\Psi_{T}}\subset W$, then $\chi^{\ast}_{2\delta_{\Phi},W}=\chi^{\ast}_{2\delta'_{\Phi},W}$. The weight
$2\delta'_{\Phi}$ is defined in \cite{Larsen-Pink} and it is observed there that $\chi^{\ast}_{2\delta'_{\Phi},W}$
is the unique leading term of $F_{\Phi,W}$ if $W_{\Psi_{T}}\subset W$.

\smallskip

The following two propositions connect Sato-Tate measures $\st_{H}$ and the characters
$\{F_{\Phi,\Gamma^{\circ}}\}$. The importance of them is: we are able to tackle dimension data by studying
the algebraic objects $F_{\Phi,\Gamma^{\circ}}$. Proposition \ref{P:group-root system}
reduces the dimension datum problem and the linear dependence problem to comparing the characters
$\{F_{\Phi,\Gamma^{\circ}}|\ \Phi\subset\Psi_{T}\}$ and getting linear relations among them. After this,
combinatorial classification of reduced sub-root systems and algebraic method treating these
characters come into force.


The proof of the following proposition can be found in \cite{Larsen-Pink}, Section 1 or \cite{An-Yu-Yu},
Section 4. For completeness, we give a sketch of the proof.

\begin{proposition}\label{P: ST support}
For a closed subgroup $H$ of $G$, the support of the Sato-Tate measure $\st_{H}$ is the set of conjugacy classes
contained in the set $\{gxg^{-1}|\ g\in G, x\in H\}$. If $H$ is connected, then for a maximal torus $T$ of $H$
contained in a maximal torus $S$ of $G$, oen has \[\big(\bigcup_{w=nS\in W_{G}}nTn^{-1}\big)/W_{G}\subset S/W_{G}
\subset G^{\natural},\] $S/W_{G}$ being a connected component of $G^{\natural}$ and
$\big(\bigcup_{w=nS\in W_{G}}nTn^{-1}\big)/W_{G}$ being equal to the support of $\st_{H}$.
Again if $H$ is connected, then for the natural map $\pi: T\rightarrow \supp(\st_{H})$, oen has
\[\pi^{\ast}(\st_{H})=F_{\Phi,\Gamma^{\circ}}(t),\] where $\Phi=\Phi(H)$.
\end{proposition}

\begin{proof}
The first and the second statements are clear.

Given $H$ being connected, for a conjugation invariant continuous function $f$ on $G$, one has
\[\int_{H} fd\mu_{H}=\int_{H} fd\mu_{H}.\] By the Weyl integration formula,
\[\int_{H} fd\mu_{H}=\frac{1}{|W_{H}|}\int_{T}f(t)F_{\Phi}(t)dt,\]
where $F_{\Phi}(t)$ is the Weyl product.
Writing this expression $\Gamma^{\circ}$ invariant, we get
\begin{eqnarray*}&&\frac{1}{|W_{H}|}\int_{T}f(t)F_{\Phi}(t)dt\\&=&
\frac{1}{|W_{H}|}\frac{1}{|Stab_{W_{G}}(T)|}\sum_{\gamma=nS\in W_{G},nTn^{-1}=T}
\int_{T}f(n^{-1}tn)F_{\Phi}(n^{-1}tn)dt\\&=&
\frac{1}{|W_{H}|}\frac{1}{|Stab_{W_{G}}(T)|}\sum_{\gamma=nS\in W_{G}, nTn^{-1}=T}
\int_{T}f(t)F_{\Phi}(n^{-1}tn)dt\\&=&\frac{1}{|W_{H}|}\frac{1}{|\Gamma^{\circ}|}
\int_{T}f(t)\big(\sum_{\gamma\in\Gamma^{\circ}}F_{\Phi}(\gamma^{-1}(t))\big)dt.
\end{eqnarray*}


Moreover, we have \begin{eqnarray*}&&F_{\Phi}(t)\\&=&\prod_{\alpha\in\Phi}\big(1-[\alpha]\big)
\\&=&\prod_{\alpha\in\Phi^{+}}\big([\frac{-\alpha}{2}]-[\frac{\alpha}{2}]\big)\big([\frac{\alpha}{2}]-
[\frac{-\alpha}{2}]\big)\\&=&\big(\sum_{w\in W_{\Phi}}\sgn(w)[-w\delta]\big)\big(\sum_{w\in W_{\Phi}}
\sgn(w)[w\delta]\big)\\&=&\sum_{w,\tau\in W}\sgn(w)\sgn(\tau)[-w\delta+\tau\delta]\\&=&
\sum_{w,\tau\in W_{\Phi}}\sgn(w)[-\tau w\delta+\tau\delta]\quad (\textrm{use }
w\rightarrow\tau w)\\&=&\sum_{\tau\in W_{\Phi}}\tau\big(\sum_{w\in W_{\Phi}}\sgn(w)
[\delta-w\delta]\big)\end{eqnarray*} and
 \begin{eqnarray*}&&\frac{1}{|W_{H}|}\frac{1}{|\Gamma^{\circ}|}
\sum_{\gamma\in\Gamma^{\circ}}F_{\Phi}(\gamma^{-1}(t))\\&=&\frac{1}{|W_{H}|}
\frac{1}{|\Gamma^{\circ}|}\sum_{\gamma\in\Gamma^{\circ}}\gamma(F_{\Phi}(t))
\\&=&\frac{1}{|W_{H}|}\frac{1}{|\Gamma^{\circ}|}\sum_{\gamma\in\Gamma^{\circ}}
\gamma\big(\sum_{\tau\in W_{\Phi}}\tau(\sum_{w\in W_{\Phi}}\sgn(w)[\delta-w\delta])\big)
\\&=&\frac{1}{|\Gamma^{\circ}|}\sum_{\gamma\in\Gamma^{\circ}}\gamma\big(\sum_{w\in
W_{\Phi}}\sgn(w)[\delta-w\delta]\big)\\&=& F_{\Phi,\Gamma^{\circ}}(t).\end{eqnarray*}
Hence the last statement follows.
\end{proof}

The proof of the following proposition is a bit complicated. The basic idea is as follows.
Given a maximal torus $S$ of a connected compact Lie group $G$ and $W_{G}=N_{G}(S)/C_{G}(S)$, then
the set $G^{\natural}$ of conjugacy classes in $G$ can be identified with $S/W_{G}$. For any
connected closed subgroup $H$ with a maximal torus $T$ contained in $S$, the support of the
Sato-Tate measure $\st_{H}$ of $H$ is $p(T)$ where $p: S\longrightarrow S/W_{G}$ be the natural
projection. This indicates that: not only $\supp(\st_{H})$ has dimension equal to $\dim T$, but
also the integral of $\st_{H}$ against a continuous function on $S/W_{G}$ with a given bound and
support near a given set of dimension less than $\dim T$ can be made arbitrarily small if
the distance of the support of the function to that given set is suffiently small. With this fact,
by choosing test functions on $S/W_{G}$ appropriately we are able to distinguish different Sato-Tate
measures by their supports.

\begin{proposition}\label{P:group-root system}
Given a list $\{H_{1},H_2,\dots,H_{s}\}$ of connected closed subgroups of a compact Lie group $G$
and non-zero real numbers $c_1,\cdots,c_{s}$, \[\sum_{1\leq i\leq s} c_{i}\mathscr{D}_{H_i}=0\] if
and only if for any closed connected torus $T$ in $G$,
\[\sum_{1\leq j\leq t} c_{i_{j}}F_{\Phi_{i_{j}},\Gamma^{\circ}}=0,\] where
$\{H_{i_{j}}|i_1<i_2<\cdots<i_{t}\}$ are all subgroups among $\{H_{i}|1\leq i\leq s\}$ with maximal
tori conjugate to $T$ in $G$, $\Phi_{i_{j}}$ is the root system of $H_{i_{j}}$ regarded as a root
system on $T$ and $\Gamma^{\circ}=N_{G}(T)/C_{G}(T)$.
\end{proposition}

\begin{proof}
By the relation between dimension data and Sato-Tate measures, the equality
\[]\sum_{1\leq i\leq s}c_{i}\mathscr{D}_{H_i}=0\]
is equivalent to $\sum_{1\leq i\leq s} c_{i}\st_{H_{i}}=0$.

Choosing a maximal torus $T_{i}$ of $G_{i}$, for any $1\leq i\leq s$, we may and do assume that
$\{T_{i}: 1\leq i\leq s\}$ are all contained in a maximal torus $S$ of $G$. The set of conjugacy
classes of $G$ contained in $G_0$, $G_0/\Ad(G)\subset G^{\natural}$ can be identified with
$S/W_{G}$. Let $p: S\longrightarrow S/W_{G}$ be the natural projection. Under the above identification,
the support of $\st_{H_{i}}$ is $p(T_{i})$. Among $\{T_{i}: 1\leq i\leq s\}$, we may assume that
$T=T_1=T_2=\cdots=T_{t}$, and any other $T_{i}$ ($i\geq t+1$) is not conjugate to $T$ and has
dimension $\geq\dim T$.

The Riemannian metric $m$ on $G$ induces a Riemannian metric on $S$. It gives $S$ the structure of
a metric space. Since $W_{G}$ acts on $S$ by isometries, $S/W_{G}$ inherits a metric structure. For
any $i\geq t+1$, since $T_{i}$ is not conjugate to $T$ and has dimension $\geq\dim T$,
$gT_{i}g^{-1}\not\subset T$ for any $g\in G$. Moreover each $T_{i}$ is a closed torus, thus
$p(T)\cap p(T_{i})\subset p(T_{i})$ is a union of the images under $p$ of finitely many closed tori
of $S$ of strictly lower dimension than $\dim T_{i}$.

Given any $\epsilon>0$, let $\mathcal{U}_{\epsilon}$ be the subset of $S/W_{G}$ consisting of points
with distance to $p(T)$ within $\epsilon$. For any continuous function $f$ on $S/W_{G}$ with absolute
value bounded by a positive number $K/2$, there exists a continuous function $f_{\epsilon}$ on $S/W_{G}$
with support contained in $\mathcal{U}_{\epsilon}$ and absolute value bounded by $K$, and having
restriction to $p(T)$ equal to $f|_{T}$. Since $\sum_{1\leq i\leq s} c_{i}\st_{H_{i}}=0$,
\begin{equation}\label{A}\sum_{1\leq i\leq s} c_{i}\st_{H_{i}}(f_{\epsilon})=0.\end{equation} For any
$i\geq t+1$, since $p(T)\cap p(T_{i})\subset p(T_{i})$ is a union of the images under $p$ of finitely
many closed tori of $S$ of strictly lower dimension than $\dim T_{i}$ and $\st_{H_{i}}$ has support
equal to $p(T_{i})$, we get \[\lim_{\epsilon\rightarrow 0}\st_{H_{i}}(f_{\epsilon})=0.\] Taking the limit
as $\epsilon$ approaches $0$ in Equation (\ref{A}), we get \[\sum_{1\leq i\leq t} c_{i}\st_{H_{i}}(f)=0.\]
Hence \[\sum_{1\leq i\leq t} c_{i}\st_{H_i}=0.\] Therefore an inductive argument on $s$ finishes
the proof.
\end{proof}


\begin{remark}\label{R:DDPLDP-character}
In Proposition \ref{P:group-root system}, if $G$ is a connected compact simple Lie group with
a root system $\Psi$, then $\Gamma^{\circ}=W_{\Psi}$. Moreover, in this case the conjugacy class of
a full rank subgroup connected closed subgroup is determined by its root system, which can be regarded
as a sub-root system of $\Psi$. In this way, finding linear relations among dimension data of
connected full rank subgroups is equivalent to finding linear relations among the characters
$\{F_{\Phi,W_{\Psi}}|\ \Phi\subset\Psi\}.$
\end{remark}

\section{Comparison of different conjugacy conditions}\label{S:conjugacy}

In this section we discuss relations among several finite groups $\Gamma^{\circ}$, $\Gamma$, $W_{\Psi'_{T}}$,
$W_{\Psi_{T}}$ and $\Aut(\Psi_{T})$ defined in previous sections. Moreover, we discuss the connection between
various relations of two subgroups and conjugacy relations of their root-systems with respect to these
finite groups.

Given a compact Lie group $G$ with a biinvariant Riemannian metric $m$ and a closed connected torus $T$,
denote by $\Lambda_{T}=\Hom(T,\U(1))$. We have defined two finite groups $\Gamma^{\circ}$, $\Gamma$ by
\[\Gamma^{\circ}=N_{G}(T)/C_{G}(T)\] and \[\Gamma=\Aut(\Lambda_{T},m|_{\Lambda_{T}}).\] Moreover, we get a root
system $\Psi_{T}$ on $T$. Thus we have the finite groups $W_{\Psi_{T}}$ and $\Aut(\Psi_{T})$.
By Definition \ref{D:maxroot} the action of $W_{\Psi_{T}}$ stabilizes $\Lambda_{T}$ and preserves
$m|_{\Lambda_{T}}$. Thus $W_{\Psi_{T}}$ can be regarded as a subgroup of $\Gamma=\Aut(\Lambda_{T},m|_{\Lambda_{T}})$.
By the definition of $\Psi_{T}$ it is stable under the action of $\Gamma$. Thus there is a group homomorphism
\[\pi:\Gamma\longrightarrow\Aut(\Psi_{T}).\] In the case that $\rank\Psi_{T}=\dim T$, $\pi$ is an injective map.
In general it is not injective, however its restriction to $W_{\Psi_{T}}$ is injective.

\begin{proposition}\label{P:finite groups}
In general the inclusions $\Gamma^{\circ}\subset\Gamma$, $W_{\Psi_{T}}\subset\Gamma$ and
$\pi(\Gamma)\subset\Aut(\Psi_{T})$ are proper inclusions and neither $\Gamma^{\circ}$ nor $W_{\Psi_{T}}$
contain the other one.
\end{proposition}

\begin{proof}
Given $n\geq 5$, denote by $G=\Aut(\mathfrak{su(n)})$ and $T$ a maximal torus of $G$. Then we have
$\Psi_{T}=\A_{n-1}$, $\rank\Psi_{T}=\dim T$, $\Gamma^{\circ}=\Gamma=\Aut(\Psi_{T})=
S_{n}\times\{-1\}$ and $W_{\Psi_{T}}=S_{n}$. In this case $\Gamma\neq W_{\Psi_{T}}$ and
$\Gamma^{\circ}\not\subset W_{\Psi_{T}}$.

In Example \ref{E:LP}, by choosing $W$ appropriately we have $\Gamma\neq\Gamma^{\circ}$ and
$W_{\Psi_{T}}\not\subset\Gamma^{\circ}$.


Let $G=\big(\SU(2)^{3}\big)/\langle(-I,-I,I)\rangle$ and $T$ a maximal torus of $G$. In this case
$\Psi_{T}=3\A_1$ and $\rank\Psi_{T}=\dim T$. Thus $\pi$ is injective. Moreover, we have
\[\Gamma=\Aut(T,m|_{T})=\{\pm{1}\}^{3}\rtimes\langle\sigma_{12}\rangle\] and
\[\Aut(\Psi_{T})=\{\pm{1}\}^{3}\rtimes S_3,\] where $\sigma_{12}$ is the transposition on the first and the
second positions. Therefore $\pi(\Gamma)\neq\Aut(\Psi_{T})$.
\end{proof}



\begin{definition}\label{D:local conjugacy}
Given two compact Lie groups $H$ and $G$, two homomorphisms $\phi_1,\phi_2: H\longrightarrow G$ are
said element-conjugate if $\phi_1(x)\sim\phi_2(x)$ for any $x\in H$. Similarly, we call two closed
subgroups $H_1,H_2$ of $G$ element-conjugate if there exists an isomorphism $\phi:H_1\longrightarrow H_2$
such that $x\sim\phi(x)$ for any $x\in H$.
\end{definition}

Element-conjugate homomorphisms are defined and studied in \cite{Larsen2}. It is proved in \cite{Larsen2}
and \cite{Larsen3} that for $G$ equal to $\SU(n)$, $\SO(2n+1)$, $\Sp(n)$ or $\G_2$, element-conjugate
homomorphisms to $G$ are actually conjugate. In the converse direction, for $G$ equal to $\SO(2n)$ ($n\geq 4$)
or a connected simple group of type $\E_6$, $\E_7$, $\E_8$ or $\F_4$, there exist element-conjugate homomorphisms
to $G$ which are not conjugate. In \cite{Wang}, S. Wang considered homomorphisms from connected groups and
used the name locally conjugate instead of element-conjugate. He gave some examples of element-conjugate
homomorphisms from connected groups to $\SO(2n)$ which are not conjuagte. In the 1950s Dynkin classified
semisimple subalgebras of complex simple Lie algebras up to {\it linear conjugacy} (cf. \cite{Dynkin}).
Moreover, Minchenko distinguished the conjugacy classes among linear conjugacy classes of semisimple
subalgebras (cf. \cite{Minchenko}). By \cite{Minchenko} Theorem 1 and Proposition \ref{P:local-root system}
below, two connected closed subgroups are element-conjugate if and only if their subalgebras are linear
conjugate. From this, connected closed subgroups of connected compact simple Lie group which are locally
conjugate but not globally conjugate can be classified as well.

\begin{proposition}\label{P:local-root system}
Given two compact Lie groups $H$, $G$ with $H$ connected, and a maximal torus $T$ of $H$, two homomorphisms
$\phi_1,\phi_2: H\longrightarrow G$ are locally conjugate if and only if there exists $g\in G$ such that
$\phi_2|_{T}=(\Ad(g)\circ\phi_1)|_{T}$ and the root systems $\Phi(H_1,\phi_1(T))=\Phi(H_2,\phi_1(T))$,
where we denote by $H_{1}=(\Ad(g)\circ\phi_1)(H)$ and $H_{2}=\phi_{2}(H)$.
\end{proposition}
\begin{proof}
The ``if'' part is clear. We prove the ``only if'' part. Choosing a maximal torus $S$ of $G$, write
$W(G,S)=N_{G}(S)/C_{G}(S)$. Since $\phi_1,\phi_2$ are element-conjugate, we have $\ker\phi_1=\ker\phi_2$.
Without loss of generality we may assume that $\phi_1,\phi_2$ are injective and
$\phi_1(T),\phi_2(T)\subset S$.

For any $x\in T$, since $\phi_1(x)\sim\phi_2(x)$, we have $\phi_2=n(\phi_1(x))n^{-1}$ for some
$w=nS\in W(G,S)$. As $W(G,S)$ is finite, there exists $w\in W(G,S)$ such that $\phi_2=n(\phi_1(x))n^{-1}$
holds for a set of $x\in T$ generating $T$. Hence $\phi_2=n(\phi_1(x))n^{-1}$ for any $x\in T$.
This means that $(\Ad(n)\circ\phi_1)|_{T}=\phi_2|_{T}$. For simplicity, we assume that
$\phi_1|_{T}=\phi_2|_{T}$. We still denote by $T$ the image in $S$ of $T$ under $\phi_1$ and $\phi_2$.
In this way $\phi_1|_{T}$ and $\phi_2|_{T}$ are both the identity map. Write
\[\frh=\frt\oplus(\bigoplus_{\alpha\in\Phi}\frh_{\alpha})\] for the root space decomposition of
$\frh=(\Lie H)\otimes_{\bbR}\bbC$ with respect to the $T$ action and
\[\frg=\frg^{T}\oplus(\bigoplus_{\lambda\in\Lambda_{T}}\frg_{\lambda})\] the generalized root space
decomposition of $\frg=(\Lie G)\otimes_{\bbR}\bbC$ with respect to the $T$ action, where
$\Phi=\Phi(H,T)$ is root system of $H$ and $\Lambda_{T}=\Hom(T,\U(1))$ is the integral weight lattice.
As $\phi_1|_{T}$ and $\phi_2|_{T}$ are both the identity map,
\[\phi_1(\frh_{\alpha}),\phi_2(\frh_{\alpha})\subset\frg_{\alpha}\] for any $\alpha\in\Phi$.
That just means $\Phi(H_1,\phi_1(T))=\Phi(H_2,\phi_1(T))$.
\end{proof}


Given a compact Lie group $G$ with a biinvariant Riemannian metric $m$ and a connected closed torus $T$,
denote by $H_1$, $H_2$ two connected closed subgroups of $G$, we could consider different conjugacy
relations between their root systems. Precisely, write $\Phi_{i}=\Phi(H_{i},T)$ for the root system of
$H_{i}$, $i=1,2$. We may consider if $\Phi_1,\Phi_2\subset\Psi_{T}\subset\Lambda_{T}$ are conjugate under
$\Gamma^{\circ}$, $\Gamma$, $W_{\Psi}$ or $\Aut(\Psi_{T})$. Given two root systems
$\Phi_1,\Phi_2\subset\Psi_{T}$, since $\Gamma^{\circ}\subset\Gamma$ and $\pi(\Gamma)\subset\Aut(\Psi_{T})$,
we have \[\Phi_1\sim_{\Gamma^{\circ}}\Phi_2\Longrightarrow\Phi_1\sim_{\Gamma}\Phi_2\Longrightarrow\Phi_1
\sim_{\Aut(\Psi_{T})}\Phi_2.\]

We have the following connections between the conjugacy relations between the root systems and the
relations between the subgroups.
\begin{proposition}\label{P:root system-group}
Given a compact Lie group $G$ with a biinvariant Riemannian metric $m$ and a connected closed torus $T$,
for two closed subgroups $H_1,H_2$ with $T$ a common maximal torus of them, denote by $\Phi_{i}=
\Phi(H_{i},T)$, $i=1,2$. We have:
\begin{itemize}
\item[(1)]{if $\Phi_1\sim_{\Aut(\Psi_{T})}\Phi_2$, then the Lie algebras of $H_1$ and $H_2$ are
isomorphic.}
\item[(2)]{If $\Phi_1\sim_{\Gamma}\Phi_2$, then the Lie groups $H_1$ and $H_2$ are isomorphic.}
\item[(3)]{$\Phi_1\sim_{\Gamma^{\circ}}\Phi_2$ if and only if $H_1$ and $H_2$ are locally
conjugate.}
\end{itemize}
\end{proposition}

\begin{proof}
Write \[\frg=\frg^{T}\oplus(\bigoplus_{\lambda\in\Lambda_{T}}\frg_{\lambda})\] for the generalized root space
decomposition of $\frg=(\Lie G)\otimes_{\bbR}\bbC$ with respect to the $T$ action and
\[\frh_{i}=\frt\oplus(\bigoplus_{\alpha\in\Phi_{i}}\frh_{i,\alpha})\] the root space decomposition of
$\frh_{i}=(\Lie H_{i})\otimes_{\bbR}\bbC$ with respect to the $T$ action, $i=1,2$. We have
$\Phi_1,\Phi_2\subset\Psi_{T}\subset\Lambda_{T}$.

Suppse $\Phi_1\sim_{\Aut(\Psi_{T})}\Phi_2$. This means that there exists $\gamma\in\Aut(\Psi_{T})$ such that
$\Phi_2=\gamma\Phi_1$. By the classification of complex reductive Lie algebras (cf. \cite{Knapp}),
$\gamma$ can be lifted to an isomorphism $f:\frh_1\longrightarrow\frh_2$ such that
$f(\frh_{1,\alpha})=\frh_{2,\gamma\alpha}$ for any $\alpha\in\Phi_1$. This proves $(1)$.

Suppose $\Phi_1\sim_{\Gamma}\Phi_2$. Recall that $\Gamma=\Hom(T,m|_{T})$. For the automorphism $f$ in the
above proof, furthermore we may assume that $f|_{\Lie T}$ is induced by an automorphism of $T$. Thus
$f$ can be lifted to an isomorphism $f: H_1\longrightarrow H_2$. This proves $(2)$.

Recall that $\Gamma^{\circ}=N_{G}(T)/C_{G}(T)$. If $\Phi_1\sim_{\Gamma^{\circ}}\Phi_2$, then $\Phi_1=\Phi'_2$
for some $H'_2=gH_2g^{-1}$, $g\in N_{G}(T)$ and $\Phi'_2=\Phi(H'_2,T)$. Therefore $(3)$ follows from
Proposition \ref{P:local-root system}.
\end{proof}

\begin{remark}\label{R:local-maximal torus}
Given a compact Lie group $G$ and a connected closed torus $T$, as indicated in the proof of Proposition
\ref{P:local-root system}, a trivial observation is: the existence of locally-conjugate but not globally
conjugate connected closed subgroups with $T$ a maximal torus of them is due to the fact the generalized root
spaces $\{\frg_{\lambda}:\lambda\in\Lambda_{T}\}$ may have dimension larger than one and the different choices of
eigenvectors really matter for the conjugacy classes of the subgroups. Hence there is no local-global issue
if each $\frg_{\alpha}$ having dimension one. In particular, two full rank connected subgroups are conjugate
if and only if they are locally conjugate.
\end{remark}

\section{Formulation of the problems in terms of root systems}\label{S:formulation of questions}



We formulate the following two questions, analoguous to the dimension datum problem and the linear dependence
problem respectively.

\begin{question}\label{Q:equal-character}
Given a root system $\Psi$, for which pairs $(\Phi_1,\Phi_2)$ of reduced sub-root systems of $\Psi$, we
have $F_{\Phi_1,\Aut(\Psi)}=F_{\Phi_2,\Aut(\Psi)}$?
\end{question}

\begin{question}\label{Q:dependent-character}
Given a root system $\Psi$ and a finite group $W$ such that $W_{\Psi}\subset W\subset\Aut(\Psi)$,
for which reduced sub-root systems $\Phi_1,\Phi_2,\cdots,\Phi_{s}$ of $\Psi$, the characters
\[F_{\Phi_1,W},F_{\Phi_2,W},\cdots,F_{\Phi_{s},W}\] are linearly dependent?
\end{question}



\begin{remark}\label{R:linear}
We could take $(\Psi,W)=(\Psi_{T},\pi(\Gamma))$ or  $(\Psi,W)=(\Psi'_{T},\pi(\Gamma^{\circ}))$ in Question
\ref{Q:dependent-character}. The property $W_{\Psi_{T}}\subset\pi(\Gamma)\subset\Aut(\Psi_{T})$ follows
from the defintions of $\Gamma$ and of $\Psi_{T}$. The property $W_{\Psi'_{T}}\subset\pi(\Gamma^{\circ})\subset
\Aut(\Psi'_{T})$ is proved in Proposition \ref{P:Gamma0-Psi'}. Taking the pair $(\Psi'_{T},\Gamma^{\circ})$,
the solution to Question \ref{Q:dependent-character} gives a solution to Question \ref{Q:linear dependence}.
\end{remark}

\section{Classification of sub-root systems}\label{S:sub-root systems}

In this section, given an irreducible root system $\Psi_0$, we discuss the classification of sub-root
systems of $\Psi_0$ up to $W_{\Psi_0}$ conjugation. In the case that $\Psi_0$ is a classical irreducible root
system, classification of sub-root systems of it seems to have been known to experts la ong time ago. A
detailed exposition of this classification is given in \cite{Oshima}. In the case that $\Psi_0$ is an
exceptional irreducible root system, classification of sub-root systems of it was first achieved by Oshima
in \cite{Oshima} except that the classification of sub-root systems of $\F_4$ or $\G_2$ was known
before. Our discussion follows \cite{Larsen-Pink} and \cite{Oshima}.

Before discussing the classification, we fix notations. Let $V=\bbR^{n}$ be an $n$-dimensional
Euclidean vector space with an orthonormal basis $\{e_1,e_2,...,e_{n}\}$. Denote by the root systems
\begin{eqnarray*}&&\A_{n-1}=\big\{\pm{(e_{i}-e_{j})}|\ 1\leq i<j\leq n\big\},\\&&
\B_{n}=\big\{\pm{e_{i}}\pm{e_{j}}|\ 1\leq i<j\leq n\}\cup\{\pm{e_{i}|1\leq i\leq n}\big\},\\&&
\C_{n}=\big\{\pm{e_{i}}\pm{e_{j}}|\ 1\leq i<j\leq n\big\}\bigcup\big\{\pm{2e_{i}|\ 1\leq i\leq n}\big\},\\&&
\BC_{n}=\big\{\pm{e_{i}}\pm{e_{j}}|\ 1\leq i<j\leq n\big\}\bigcup\big\{\pm{e_{i}},\pm{2e_{i}|\ 1\leq i\leq n}
\big\},\\&& \D_{n}=\big\{\pm{e_{i}}\pm{e_{j}}|\ 1\leq i<j\leq n\big\}. \end{eqnarray*}
Thus \[\A_{n-1}\subset\D_{n}\subset\B_{n},\C_{n}\subset\BC_{n}.\]

Note that $\A_0=\D_1=\emptyset$ and we regard $\A_1,\B_1,\C_1$ different from each other though
they are isomorphic as abstract root systems. Each of the following pairs $(\B_2,\C_2)$, $(\D_2, 2\A_1)$,
$(\A_3,\D_3)$ are also regarded different.

{\it Type $\A$.} Given $\Psi_0=\A_{n-1}$ ($n\geq 2$), 
let $\Phi\subset\Psi_0$ be a sub-root system. Define a relation on $\{1,2,...,n\}$ by
\[i\sim j\Leftrightarrow e_{i}-e_{j}\in\Phi.\] Thus $\sim$ is an equivalence relation. It divides
$\{1,2,...,n\}$ into $l$ subsets of cardinalities $n_1,n_2,...,n_{l}$ where $n_1\geq n_2\geq\cdots\geq n_{l}$
and $n_1+n_2+\cdots+n_{l}=n$. In this way we can show \[\Phi\sim\bigcup_{1\leq k\leq l}
\big\{\pm{(e_{i}-e_{j})}|\ n_1+n_2+\cdots+n_{k-1}+1\leq i<j\leq n_1+n_2+\cdots+n_{k}\big\}.\] We denote by
\[\Phi_{n_1,n_2,...,n_{l}}=\bigcup_{1\leq k\leq l}\big\{\pm{(e_{i}-e_{j})}|\ n_1+n_2+\cdots+n_{k-1}+
1\leq i<j\leq n_1+n_2+\cdots+n_{k}\big\}.\] It is isomorphic to $\bigcup_{1\leq k\leq l}\A_{n_{k}-1}$.

{\it Type $\B$.} Given $\Psi_0=\B_{n}$ ($n\geq 1$),
let $\Phi\subset\Psi_0$ be a sub-root system. Define a relation on $\{1,2,...,r\}$ by
\[i\sim j\Leftrightarrow e_{i}-e_{j}\in\Phi\textrm{ or }e_{i}+e_{j}\in\Phi.\] Thus $\sim$ is an
equivalent relation. It divides $\{1,2,...,n\}$ into $l$ subsets of cardinarities $n_1,n_2,...,n_{l}$
where $n_1+n_2+\cdots+n_{l}=n$. Write $n'_{k}=n_1+n_2+\cdots+n_{k}$ for $1\leq k\leq l$. Without loss
of generality we may assume that $\{i|\ n'_{k-1}+1\leq i\leq n'_{k}\}$ is an equivalence set for any
$1\leq k\leq l$ and \[\big\{\pm{(e_{i}-e_{j})}|\ n'_{k-1}+1\leq i<j\leq n'_{k}\big\}\subset\Phi.\] If
$e_{i}\in\Phi$ for some $n'_{k-1}+1\leq i\leq n_{k}$, we have
\[\B_{n_{k}}\cong\big\{\pm{e_{i}}\pm{e_{j}}|\ n'_{k-1}+1\leq i<j\leq n'_{k}\big\}\bigcup
\big\{\pm{e_{i}}|\ n'_{k-1}+1\leq i\leq n'_{k}\big\}\subset\Phi.\] If $e_{i}\not\in\Phi$ for any
$n'_{k-1}+1\leq i\leq n_{k}$ and $e_{i}+e_{j}\in\Phi$ for some $n'_{k-1}+1\leq i<j\leq n'_{k}$,
we have \[\D_{n_{k}}\cong\big\{\pm{e_{i}}\pm{e_{j}}|\ n'_{k-1}+1\leq i<j\leq n'_{k}\big\}\subset\Phi.\]
In summary we have \[\Phi\cong\big(\bigcup_{1\leq i\leq u}\B_{r_{i}}\big)\bigcup\big(\bigcup_{1\leq j\leq v}
\D_{s_{j}}\big)\bigcup\big(\bigcup_{1\leq k\leq w}\A_{t_{k}-1}\big)\] where
$\sum_{1\leq i\leq u} r_{i}+\sum_{1\leq j\leq v} s_{j}+\sum_{1\leq k\leq w} t_{k}=n$,
$r_{i}\geq 1$, $s_{j}\geq 2$ and $t_{k}\geq 1$. Moreover, the conjugacy class of $\Phi$ is uniquely
determined by these indices $\{r_{i},s_{j},t_{k}\}$. We denote by $\Phi_{\{r_{i}\},\{s_{j}\},\{t_{k}\}}$
a sub-root system with the indices $\{r_{i},s_{j},t_{k}\}$.

{\it Type $\C$.} Given $\Psi_0=\C_{n}$, let $\Phi\subset\Psi_0$ be a sub-root system. Similar as in the
$\B_{n}$ case, we can show \[\Phi\cong(\bigcup_{1\leq i\leq u}\C_{r_{i}})\bigcup(\bigcup_{1\leq j\leq v}
\D_{s_{j}})\bigcup(\bigcup_{1\leq k\leq w}\A_{t_{k}-1})\] where
$\sum_{1\leq i\leq u} r_{i}+\sum_{1\leq j\leq v} s_{j}+\sum_{1\leq k\leq w} t_{k}=n$,
$r_{i}\geq 1$, $s_{j}\geq 2$ and $t_{k}\geq 1$. Moreover, the conjugacy class of $\Phi$ is uniquely
determined by these indices $\{r_{i},s_{j},t_{k}\}$. We denote by $\Phi_{\{r_{i}\},\{s_{j}\},\{t_{k}\}}$
a sub-root system with the indices $\{r_{i},s_{j},t_{k}\}$.

{\it Type $\BC$.} Given $\Psi_0=\BC_{n}$, let $\Phi\subset\Psi_0$ be a sub-root system. Similar as in
the $\B_{n}$ case, we can show \[\Phi\cong(\bigcup_{1\leq i\leq u}\BC_{r_{i}})\bigcup(\bigcup_{1\leq j\leq v}
\B_{s_{j}})\bigcup(\bigcup_{1\leq k\leq w}\C_{t_{k}})\bigcup(\bigcup_{1\leq e\leq x}\D_{p_{e}})\bigcup
(\bigcup_{1\leq f\leq y}\A_{q_{f}-1})\] where $\sum_{1\leq i\leq u} r_{i}+\sum_{1\leq j\leq v} s_{j}+
\sum_{1\leq k\leq w} t_{k}+\sum_{1\leq e\leq x} p_{e}+\sum_{1\leq f\leq y}q_{f}=n$,
$r_{i}\geq 1$, $s_{j}\geq 1$, $t_{k}\geq 1$, $p_{e}\geq 2$ and $q_{f}\geq 1$. Moreover, the conjugacy
class of $\Phi$ is uniquely determined by the indices $\{r_{i},s_{j},t_{k},p_{e},q_{f}\}$. We denote by
$\Phi_{\{r_{i}\},\{s_{j}\},\{t_{k}\},\{p_{e}\},\{q_{f}\}}$ a sub-root system with the indices
$\{r_{i},s_{j},t_{k},p_{e},q_{f}\}$.

{\it Type $\D$.} Given $\Psi_0=\D_{n}$ ($n\geq 5$), in this case $\Aut(\Psi_0)=W_{n}=\{\pm{1}\}^{n}\rtimes S_{n}$
and $W_{\Psi_0}$ is a subgroup of index $2$ in $W_{n}$. Let $\Phi\subset\Psi_0$ be a sub-root system. Similar as
in $\B_{n}$ case, we have show \[\Phi\cong(\bigcup_{1\leq j\leq v}\D_{s_{j}})\bigcup(\bigcup_{1\leq k\leq w}
\A_{t_{k}-1})\] where $\sum_{1\leq j\leq v} s_{j}+\sum_{1\leq k\leq w} t_{k}=n$, $s_{j}\geq 2$ and
$t_{k}\geq 1$. Moreover, the $\Aut(\Psi_0)$ conjugacy class of $\Phi$ is uniquely determined by the
indices $\{s_{j},t_{k}\}$. Given a set of indices $\{s_{j},t_{k}:1\leq j\leq v,1\leq k\leq w\}$, there are
at most two conjugacy classes of sub-root systems up to $W_{\Psi_0}$ conjugation. In the case that there is a
unique $W_{\Psi_0}$ conjugacy class of sub-root systems with this set of indices, we denote by $\Phi_{\{s_{j}\},
\{t_{k}\}}$ one of them. Otherwise, we can show that each $s_{j}=0$ and each $t_{k}$ is even. Write
\[\Phi_{\{t_{k}\}}=\bigcup_{1\leq k\leq l}\big\{\pm{(e_{i}-e_{j})}|\ t_1+t_2+\cdots+t_{k-1}+1\leq i<j\leq
t_1+t_2+\cdots+t_{k}\big\}\] and $\Phi'_{\{t_{k}\}}=s_{e_1}\Phi_{\{t_{k}\}}$. They represent these two
conjugacy classes.


{\it Type $\D_4$.} Given $\Psi_0=\D_{4}=\{\pm{e_{i}}\pm{e_{j}}|1\leq i<j\leq 4\}$, we denote by
\begin{eqnarray*}&&\A_1=\langle e_1-e_2\rangle,\\&&\A_2=\langle e_1-e_2, e_2-e_3\rangle,\\&&
\A_3=\langle e_1-e_2,e_2-e_3,e_3-e_4\rangle,\\&&\A'_3=\langle e_1-e_2,e_2-e_3,e_3+e_4\rangle,\\&&
2\A_1=\langle e_1-e_2,e_3-e_4\rangle,\\&& 2\A'_1=\langle e_1-e_2,e_3+e_4\rangle, \\&& \D_2=
\langle e_1-e_2,e_1+e_2\rangle,\\&& 3\A_1=\langle e_1-e_2,e_1+e_2,e_3-e_4\rangle,\\&& 2\D_2
=\langle e_1-e_1,e_1+e_2,e_3-e_4,e_3+e_4\rangle,\\&&\D_3=\langle e_1-e_2,e_2-e_3,e_2+e_3\rangle,
\\&&\D_4=\langle e_1-e_1,e_2-e_3,e_3-e_4,e_3+e_4\rangle.\end{eqnarray*} We can show that any sub-root
system of $\D_4$ is conjugate to one of them up to $W_{\Psi_0}$ conjugation. On the other hand, we have
\begin{eqnarray*}&&\A_3\sim_{\Aut(\D_4)}\A'_3\sim_{\Aut(\D_4)}\D_3, \\&& 2\A_1\sim_{\Aut(\D_4)}2\A'_1
\sim_{\Aut(\D_4)}\D_2.\end{eqnarray*} These are all $\Aut(\D_4)$ conjugacy relations among them.

{\it Type $\E_6$, $\E_7$ and $\E_8$.} Given $\Psi_0=\E_6$, $\E_7$ or $\E_8$, the classification of
sub-root systems of $\Psi_0$ is hard to describe. The readers can refer \cite{Oshima} for the details.

{\it Type $\F_4$.} Given $\Psi_0=\F_4$, we denote by $\{\alpha_1,\alpha_2,\alpha_3,\alpha_4\}$ a simple
system. The long roots and short roots in $\Psi_0$ consist in the sub-root systems $\D_4^{L}$ and $\D_4^{S}$,
respectively. There is a unique conjugacy class of sub-root systems of $\F_4$ isomorphic to $\B_4$ (or
$\C_4$). We denote by \[\B_4=\langle\alpha_2,\alpha_1,\alpha_2+2\alpha_3,\alpha_4\rangle\] and
\[\C_4=\langle\alpha_3,\alpha_4,\alpha_2+\alpha_3,\alpha_1\rangle,\] which are representatives of
them. The sub-root systems $\D_4^{L}$ and $\D_4^{S}$ are stable under $W_{\F_4}$. Hence there exist
homomorphisms \[W_{\F_4}\longrightarrow\Aut(\D_4^{L})\] and \[W_{\F_4}\longrightarrow\Aut(\D_4^{S}).\]

\begin{lemma}
Each of these two homomorphisms is an isomorphism.
\end{lemma}

\begin{proof}
We show $W_{\F_4}\longrightarrow\Aut(\D_4^{L})$ is an isomorphism. The proof for $W_{\F_4}
\longrightarrow\Aut(\D_4^{S})$ is an isomorphism is similar. Suppose $w\in W_{\F_4}$ is an element with
trivial restriction on $\D_4^{L}$. Choose an element $\alpha$ in $\D_4^{S}$. Since $\alpha$
and $\D_4^{L}$ generate a sub-root system isomorphic to $\B_{4}$, we have $\alpha=\frac{\beta_1+\beta_2}{2}$
for some $\beta_1,\beta_2\in D_4^{L}$. Thus \[w(\alpha)=\frac{w\beta_1+w\beta_2}{2}=
\frac{\beta_1+\beta_2}{2}=\alpha.\] Hence $w=1$. Therefore $W_{F_4}\longrightarrow\Aut(D_4^{L})$ is
injective. On the other hand, by calculating stabilizers we get
\[|W_{\F_4}|=24|W_{\B_3}|=24\times 3!\times 2^3=3\times 4!\times 2^4=|\Aut(\D_4)|.\] Hence
$W_{\F_4}\longrightarrow\Aut(\D_4^{L})$ is an isomorphism.
\end{proof}

By the above lemma and the classification in $\D_4$ case, we get \begin{eqnarray*}&&\D_2^{L}\sim
2\A_1^{L}\sim 2\A'^{L}_{1},\\&&\D_3^{L}\sim\A_3^{L}\sim\A'^{L}_{3},\\&&\D_2^{S}\sim 2\A_1^{S}\sim
2\A'^{S}_{1},\\&&\D_3^{S}\sim\A_3^{S}\sim\A'^{S}_{3}.\end{eqnarray*}
The following sub-root systems of $\F_4$ are contained in $\B_4$:
$$\D_4^{L},\ \A_3^{L}+\A_1^{S},\ 4\A_1^{L},\ \B_2+2\A_1^{L},\ 2\A_1^{L}+2\A_1^{S},$$
$$\B_4,\ \B_3+\A_1^{S},\ 2\B_2,\ \B_2+2\A_1^{S},\ 4\A_1^{S},\ \A_3^{L},$$
$$2\A_1^{L}+\A_1^{S},\ \B_3,\ \B_2+\A_1^{S},\ 3\A_1^{S},\ 3\A_1^{L},$$
$$\B_2+\A_1^{L},\ \A_1^{L}+2\A_1^{S},\ \A_2^{L}+\A_1^{S},\  2\A_1^{L},$$
$$\B_2,\ 2\A_1^{S},\ \A_1^{L}+\A_1^{S},\ \A_2^{L},\ \A_1^{S},\ \A_1^{L}.$$
By duality, we get sub-root systems of $\F_4$ contained in some $\C_4$. In particular,
the following are those contained in $\C_4$ and not contained in any $\B_4$:
$$\D_4^{S},\ \A_1^{L}+\A_3^{S},\ \C_4,\ \C_3+\A_1^{L},$$
$$\A_3^{S},\ \C_3,\ \A_1^{L}+\A_2^{S},\ \A_2^{S}.$$

Denote by $\Phi\subset\F_4$ a sub-root system not contained in any $\B_4$ or $\C_4$. If $\Phi$ is simple,
then $\Phi=\F_4$. If $\Phi$ is not simple, it can not contain a factor $\A_1^{L}$, $\A_1^{S}$ or $\B_2$ since
otherwise $\Phi$ is contained in a sub-root system isomorphic to $\B_4$ or $\C_4$. Hence
$\Phi=\A_2^{L}+\A_2^{S}$. Moreover, there exists a unique conjugacy class of sub-root systems in each of
the above types.

{\it Type $\G_2$.} Given $\Psi_0=\G_2$, denote by $\{\alpha,\beta\}$ a simple system of $\Psi_0$.
We denote by $$\A_1^{L}=\langle\alpha\rangle,\ \A_1^{S}=\langle\beta\rangle,\ \A_2^{L}
=\langle\alpha,\alpha+3\beta\rangle,\ \A_2^{S}=\langle\alpha+\beta,\beta\rangle$$ and
$$\A_1^{L}+\A_1^{S}=\langle\alpha,\alpha+2\beta\rangle,\ \G_2=\langle\alpha,\beta\rangle.$$
Therefore each sub-root system of $\G_2$ is conjugate to one of them.

\section{Formulas of the leading terms}\label{S:leading terms}

Given an irreducible root system $\Psi_0$ with a positive system $\Psi_0^{+}$, we normalize the inner
product on $\Psi_0$ by letting the short roots having length $1$.

\begin{definition}
Given a reduced sub-root system $\Phi$ of $\Psi_0$, write
\[\delta_{\Phi}=\frac{1}{2}\sum_{\alpha\in\Phi\cap\Psi_0^{+}}\alpha.\] Let $\delta'_{\Phi}$ be the
unique dominant weight with respect to $\Psi_0^{+}$ in the orbit $W_{\Psi_0}\delta_{\Phi}$.
Denote by $e_{\Psi_0}(\Phi)=|2\delta'_{\Phi}|^{2}$ and $e_{\Psi_0}=e_{\Psi_0}(\Psi_0)$.
\end{definition}

In the case that $\Psi_0$ is clear from the contest, we simply write $e(\Phi)$ for $e_{\Psi_0}(\Phi)$.
Oshima classified sub-root systems $\Phi$ of any irreducible root system $\Psi_0$ up to $W_{\Psi}$
conjugation. In this section, we give the formulas of $2\delta'_{\Phi}$ and $e(\Phi)$ for the reduced
sub-root systems.

First we have formulas for $e_{\Psi_0}=|2\delta'_{\Psi_0}|^{2}$ as in Table $1$.

\begin{table}[ht] \label{Ta:1} 
\caption{Formulas of $e_{\Psi_0}$} 
\centering 
\begin{tabular}{|c |c |c |c |c |c |} \hline 

$\Psi_0$ & $\A_{n-1}$ & $\B_{n} $& $\C_{n}$ & $\D_{n}$ &\\ [0.3ex] \hline

$e_{\Psi_0}$ & $\frac{(n-1)n(n+1)}{6}$ & $\frac{(2n-1)n(2n+1)}{3}$ & $\frac{n(n+1)(2n+1)}{3}$ &
$\frac{n(n-1)(2n-1)}{3}$ &  \\ [0.3ex] \hline 

$\Psi_0$ & $\E_6$ & $\E_7$ & $\E_8$ & $\F_4$ & $\G_2$\\ [0.3ex] \hline

$e_{\Psi_0}$  & $156$ & $399$ & $1240$ & $156$ & $28$\\ [0.3ex] \hline 
\end{tabular}
\end{table}

Given a reduced sub-root system $\Phi\subset\Psi_0$, denote by \[\Phi=\bigsqcup_{1\leq i\leq s}\Phi_{i}\]
the decomposition of $\Phi$ into a disjoint union of irreducible sub-root systems and $\sqrt{k_{i}}$ the
ratio of the length of short roots of $\Phi_{i}$ and $\Psi$. Then we have \[e_{\Psi_0}(\Phi)=\sum_{1\leq i\leq s}
k_{i}e_{\Phi_{i}}.\] With this formula we can calculate $e_{\Psi_0}(\Phi)$ quickly from Table $1$.
On the other hand, from the expression of $2\delta'_{\Phi}$ into a linear combination of fundamental
weights, we can also calculate $e_{\Psi_0}(\Phi)=|2\delta'_{\Phi}|^{2}$. This helps on checking whether a
formula for $2\delta'_{\Phi}$ is correct or not. Given a classical irreducible root system $\Psi$,
the calculation of $2\delta'_{\Phi}$ for sub-root systems $\Phi$ of $\Psi$ is easy. We omit it here.
For an exceptional irreducible root system $\Psi$, in Oshima's classification (cf. \cite{Oshima}),
given an abstract root system $\Phi$, sometimes there exist two conjugacy classes of sub-root systems of
$\Psi$ isomorphic to $\Phi$. In the following we give representatives of sub-root systems for all
the cases when the above ambiguity occurs.

Given $\Psi_0=\E_7$, denote by $\beta=2\alpha_1+2\alpha_2+3\alpha_3+4\alpha_4+3\alpha_5+2\alpha_6+
\alpha_7$ and $\beta'=\alpha_2+\alpha_3+2\alpha_4+\alpha_5$. We set
\begin{itemize}
\item[1.]{$A_5:\langle\alpha_2,\alpha_4,\alpha_5,\alpha_6,\alpha_7\rangle$,}

\item[2.]{$(A_5)': \langle\alpha_3,\alpha_4,\alpha_5,\alpha_6,\alpha_7\rangle$,}

\item[3.]{$3A_1: \langle\alpha_2,\alpha_5,\alpha_7\rangle$,}

\item[4.]{$(3A_1)':\langle\alpha_3,\alpha_5,\alpha_7\rangle$,}

\item[5.]{$4A_1: \langle\alpha_2,\alpha_3,\alpha_5,\alpha_7\rangle$,}

\item[6.]{$(4A_1)': \langle\alpha_2,\alpha_3,\alpha_5,\beta'\rangle$,}

\item[7.]{$A_3+A_1: \langle\alpha_2,\alpha_5,\alpha_6,\alpha_7\rangle$,}

\item[8.]{$(A_3+A_1)': \langle\alpha_3,\alpha_5,\alpha_6,\alpha_7\rangle$,}

\item[9.]{$2A_1+A_3: \langle\alpha_2,\alpha_3,\alpha_5,\alpha_6,\alpha_7\rangle$,}

\item[10.]{$(2A_1+A_3)': \langle\beta,\alpha_3,\alpha_5,\alpha_6,\alpha_7\rangle$,}

\item[11.]{$A_1+A_5: \langle\beta,\alpha_2,\alpha_4, \alpha_5,\alpha_6,\alpha_7\rangle$,}

\item[12.]{$(A_1+A_5)': \langle\beta,\alpha_3,\alpha_4, \alpha_5,\alpha_6,\alpha_7\rangle$.}
\end{itemize}

Given $\Psi_0=\E_8$, let
$\beta=2\alpha_1+3\alpha_2+4\alpha_3+6\alpha_4+5\alpha_5+4\alpha_6+3\alpha_7+2\alpha_8$
$\beta'=\alpha_1+2\alpha_2+2\alpha_3+4\alpha_4+4\alpha_5+3\alpha_6+2\alpha_7+\alpha_8$,
and $\beta''=2\alpha_1+2\alpha_2+4\alpha_3+5\alpha_4+4\alpha_5+3\alpha_6+2\alpha_7+\alpha_8$.
We set
\begin{itemize}
\item[13.]{$A_7: \langle\beta,\alpha_1,\alpha_3,\alpha_4,\alpha_5,\alpha_6,\alpha_7\rangle$,}

\item[14.]{$(A_7)':\langle\alpha_1,\alpha_3,\alpha_4,\alpha_5,\alpha_6,\alpha_7,\alpha_8\rangle$,}

\item[15.]{$4A_1:\langle\alpha_2,\alpha_3,\alpha_5,\alpha_2+\alpha_3+2\alpha_4+\alpha_5\rangle$,}

\item[16.]{$(4A_1)':\langle\alpha_2,\alpha_3,\alpha_5, \alpha_8\rangle$,}

\item[17.]{$2A_1+A_3: \langle\beta',\alpha_1,\alpha_6,\alpha_7,\alpha_8\rangle$,}

\item[18.]{$(2A_1+A_3)':\langle\alpha_2,\alpha_4,\alpha_3,\alpha_6,\alpha_8\rangle$,}

\item[19.]{$2A_3:\langle\beta',\alpha_3,\alpha_1,\alpha_6,\alpha_7,\alpha_8\rangle$,}

\item[20.]{$(2A_3)':\langle\alpha_2,\alpha_3,\alpha_4,\alpha_6,\alpha_7,\alpha_8\rangle$,}

\item[21.]{$A_1+A_5:\langle\beta'',\alpha_4,\alpha_5,\alpha_6,\alpha_7,\alpha_8\rangle$,}

\item[22.]{$(A_1+A_5)':\langle\alpha_1,\alpha_4,\alpha_5,\alpha_6,\alpha_7,\alpha_8\rangle$}
\end{itemize}

Given $\Psi_0=\D_4$, we set \begin{itemize}
\item[23.]{$2A_1:\langle\alpha_1,\alpha_3\rangle$,}

\item[24.]{$2A'_1: \langle\alpha_1,\alpha_4\rangle$,}

\item[25.]{$D_2: \langle\alpha_3,\alpha_4\rangle$,}

\item[26.]{$A_3:\langle\alpha_1,\alpha_2,\alpha_3\rangle$,}

\item[27.]{$A'_3: \langle\alpha_1,\alpha_2,\alpha_4\rangle$,}

\item[28.]{$D_3: \langle\alpha_2,\alpha_3,\alpha_4\rangle$.}
\end{itemize}












\begin{table}\label{Ta:D4}
\caption{Formulas of $2\delta'_{\Phi}$} 
\centering
\begin{tabular}{|c|c|c|c|c|c|c|c|c|} \hline
$\Psi_0$ & $\Phi$ & $2\delta'_{\Phi}$ & $e(\Phi)$\\ \hline
$\D_4$ & $A_1$ & $\omega_2$ & 1\\ \hline
$\D_4$ & $2A_1$\footnotemark & $2\omega_4$ & 2 \\ \hline
$\D_4$ & $2A'_1$\footnotemark & $2\omega_3$ & 2 \\ \hline
$\D_4$ & $D_2$\footnotemark & $2\omega_1$ & 2 \\ \hline
$\D_4$ & $A_2$   & $2\omega_2$ &  4 \\ \hline
$\D_4$ & $3A_1$ & $\omega_1+\omega_3+\omega_4$ & 3 \\ \hline
$\D_4$ & $A_3$\footnotemark & $2(\omega_2+\omega_4)$ & 10 \\ \hline
$\D_4$ & $A'_3$\footnotemark & $2(\omega_2+\omega_3)$ & 10 \\ \hline
$\D_4$ & $D_3$\footnotemark & $2(\omega_2+\omega_1)$ & 10 \\ \hline
$\D_4$ & $4A_1$ & $2\omega_2$ & 4 \\ \hline
$\D_4$ & $D_4$   & $2(\omega_1+\omega_2+\omega_3+\omega_4)$ &  28 \\ \hline
\end{tabular}
\end{table}



\begin{table}\label{Ta:E6}
\caption{Formulas of $2\delta'_{\Phi}$}
\centering
\begin{tabular}{|c|c|c|c|} \hline
$\Psi_0$ & $\Phi$ & $2\delta'_{\Phi}$ & $e(\Phi)$
 \\ \hline
$\E_6$ & $A_1$      &  $\omega_2$ & 1 \\ \hline
$\E_6$ & $A_2$      & $2\omega_2$ &  4 \\ \hline
$\E_6$ & $A_3$      & $\omega_1+2\omega_2+\omega_6$ & 10 \\ \hline
$\E_6$ & $A_4$      & $2\omega_1+2\omega_2+2\omega_6$ &  20 \\ \hline
$\E_6$ & $A_5$      & $2\omega_1+\omega_2+\omega_3+\omega_5+2\omega_6$ & 35 \\ \hline
$\E_6$ & $D_4$  & $2\omega_2+2\omega_4$ &  28 \\ \hline
$\E_6$ & $D_5$  & $2\omega_1+2\omega_2+2\omega_4+2\omega_6$ &  60 \\ \hline
$\E_6$ & $E_6$  & $2(\omega_1+\omega_2+\omega_3+\omega_4+\omega_5+\omega_6)$ & 156 \\ \hline
$\E_6$ & $2A_1$     & $\omega_1+\omega_6$ & 2 \\ \hline
$\E_6$ & $3A_1$     & $\omega_4$ & 3 \\ \hline
$\E_6$ & $4A_1$     & $2\omega_2$ & 4 \\ \hline
$\E_6$ & $A_2+A_1$  & $\omega_1+\omega_2+\omega_6$ & 5 \\ \hline
$\E_6$ & $A_2+2A_1$ & $\omega_3+\omega_5$ & 6 \\ \hline
$\E_6$ & $2A_2$     & $2\omega_1+2\omega_6$ & 8 \\ \hline
$\E_6$ & $2A_2+A_1$ & $\omega_1+\omega_4+\omega_6$ & 9 \\ \hline
$\E_6$ & $3A_2$     & $2\omega_4$ & 12 \\ \hline
$\E_6$ & $A_3+A_1$  & $\omega_2+\omega_3+\omega_5$ & 11 \\ \hline
$\E_6$ & $A_3+2A_1$ & $2\omega_4$ & 12 \\ \hline
$\E_6$ & $A_4+A_1$ & $2\omega_1+2\omega_4+2\omega_6$ & 21 \\ \hline
$\E_6$ & $A_5+A_1$  & $\omega_1+\omega_2+\omega_3+\omega_5+\omega_6$ & 36 \\ \hline
\end{tabular}
\end{table}




\begin{table}[ht] \label{Ta:E7}
\caption{Formulas of $2\delta'_{\Phi}$}
\centering
\begin{tabular}{|c |c |c |c |c | c |c |} \hline
$\Psi_0$ & $\Phi$ & $2\delta'_{\Phi}$ & $e(\Phi)$\\ \hline
$\E_7$ & $A_{1}$ & $\omega_1$ & 1 \\ \hline

$\E_7$ & $A_2$ & $2\omega_1$ & 4 \\ \hline

$\E_7$ & $A_3$ & $2\omega_1+\omega_6$ & 10 \\ \hline

$\E_7$ & $A_4$  &  $2\omega_1+2\omega_6$ &  20\\ \hline

$\E_7$ & $A_5$\footnotemark & $2\omega_1+2\omega_6+2\omega_7$
& 35 \\ \hline

$\E_7$ & $(A_5)'$\footnotemark & $\omega_1+\omega_4+2\omega_6$
& 35 \\  \hline

$\E_7$ & $A_6$  & $2\omega_4+2\omega_6$ & 56 \\ \hline

$\E_7$ & $A_7$ & $2\omega_1+2\omega_4+2\omega_6$ & 84 \\ \hline

$\E_7$ & $D_4$ & $2\omega_1+2\omega_3$ & 28 \\ \hline

$\E_7$ & $D_5$ &  $\omega_1+2\omega_3+2\omega_6$ & 60 \\ \hline

$\E_7$ & $D_6$ & $2\omega_1+\omega_2+\omega_3+\omega_5+2\omega_6+2\omega_7$ & 110\\ \hline

$\E_7$ & $E_6$ & $\omega_1+\omega_3+\omega_4+\omega_6$ & 156 \\ \hline

$\E_7$ & $E_7$ & $2(\omega_1+\omega_2+\omega_3+\omega_4+\omega_5+\omega_6+\omega_7)$ & 399 \\ \hline

$\E_7$ &  $2A_1$  &  $\omega_6$ & 2\\ \hline

$\E_7$ & $3A_1$\footnotemark&  $2\omega_7$ & 3\\ \hline

$\E_7$ & $(3A_1)'$\footnotemark   & $\omega_3$ & 3\\ \hline

$\E_7$ & $4A_1$\footnotemark & $\omega_2+\omega_7$ & 4\\ \hline

$\E_7$ & $(4A_1)'$\footnotemark  & $2\omega_1$ & 4\\ \hline

$\E_7$ & $5A_1$  &  $\omega_1+\omega_6$ & 5\\ \hline

$\E_7$&  $6A_1$  &  $\omega_4$ & 6\\ \hline

$\E_7$ & $7A_1$ & $2\omega_2$ & 7\\ \hline

$\E_7$ & $A_2+A_1$ & $\omega_1+\omega_6$ & 5\\ \hline

$\E_7$ & $A_2+2A_1$   &  $\omega_4$ & 6\\ \hline

$\E_7$ & $A_2+3A_1$  &  $2\omega_2$ & 7\\ \hline

$\E_7$ & $2A_2$ & $2\omega_6$ & 8\\ \hline

$\E_7$ &  $2A_2+A_1$  & $\omega_3+\omega_6$ & 9\\ \hline

$\E_7$ &  $3A_2$  & $2\omega_3$ & 12\\ \hline

$\E_7$ & $A_3+A_1$\footnotemark & $2\omega_1+2\omega_7$ & 11\\ \hline

$\E_7$ & $(A_3+A_1)'$\footnotemark & $\omega_1+\omega_4$ & 11\\ \hline

$\E_7$ & $A_3+2A_1$\footnotemark & $\omega_1+\omega_5+\omega_7$ & 12\\ \hline

$\E_7$ & $(A_3+2A_1)'$\footnotemark & $2\omega_3$  & 12\\ \hline

$\E_7$ & $A_3+3A_1$ &  $\omega_2+\omega_3+\omega_7$ & 13\\ \hline

$\E_7$ & $A_3+A_2$ & $\omega_4+\omega_6$ & 14\\ \hline

$\E_7$ & $A_3+A_2+A_1$ & $2\omega_5$ & 15\\ \hline

$\E_7$ & $2A_3$  & $2\omega_1+2\omega_6$ & 20\\ \hline

$\E_7$ & $2A_3+A_1$  & $\omega_1+\omega_4+\omega_6$ & 21\\ \hline

$\E_7$ & $A_4+A_1$ & $\omega_1+\omega_4+\omega_6$ & 21\\ \hline

$\E_7$ & $A_4+A_2$ & $2\omega_4$ & 24\\ \hline

$\E_7$ & $A_5+A_1$\footnotemark &$\omega_1+\omega_4+\omega_6+2\omega_7$ & 36\\ \hline

$\E_7$ & $(A_5+A_1)'$\footnotemark & $2\omega_3+2\omega_6$ & 36\\ \hline

$\E_7$ & $A_5+A_2$  &  $2\omega_4+2\omega_7$ & 39\\ \hline

$\E_7$ & $D_4+A_1$ &  $2\omega_1+\omega_2+\omega_3+\omega_7$ & 29\\ \hline

$\E_7$ & $D_4+2A_1$  & $2\omega_1+\omega_4+\omega_6$ & 30\\ \hline




\end{tabular}
\end{table}



\begin{table}[ht] \label{Ta:E7-II}
\caption{Formulas of $2\delta'_{\Phi}$}
\centering
\begin{tabular}{|c |c |c |c |c | c |c |}\hline
$\Psi_0$ & $\Phi$ & $2\delta'_{\Phi}$ & $e(\Phi)$\\ \hline












$\E_7$ & $D_4+3A_1$ & $2\omega_1+2\omega_5$ & 31\\ \hline

$\E_7$ & $D_5+A_1$ & $2\omega_1+\omega_2+\omega_3+\omega_5+\omega_6$ & 61\\ \hline

$\E_7$ & $D_6+A_1$  & $2\omega_1+2\omega_4+2\omega_6+2\omega_7$ & 111\\ \hline
\end{tabular}
\end{table}




\begin{table}[ht] \label{Ta:E8}
\caption{Formulas of $2\delta'_{\Phi}$}
\centering \begin{tabular}{|c |c |c |c |c |}\hline
$\Psi_0$ & $\Phi$ & $2\delta'_{\Phi}$ & $e(\Phi)$\\ \hline
$\E_8$ & $A_{1}$ & $\omega_8$ & 1\\ \hline

$\E_8$ & $A_2$ & $2\omega_8$ & 4\\ \hline

$\E_8$ & $A_3$  & $\omega_1+2\omega_8$ & 10 \\ \hline

$\E_8$ & $A_4$ & $2\omega_1+2\omega_8$ & 14 \\ \hline

$\E_8$ & $A_5$ &  $2\omega_1+\omega_6+\omega_8$ & 35\\ \hline

$\E_8$ & $A_6$ &$2\omega_1+2\omega_6$ & 56  \\ \hline

$\E_8$ & $A_7$\footnotemark & $2\omega_1+2\omega_6+2\omega_8$ & 84 \\ \hline

$\E_8$ & $(A_7)'$\footnotemark & $\omega_1+\omega_4+\omega_6+\omega_7$ & 84 \\ \hline

$\E_8$ & $A_8$ & $2\omega_4+2\omega_7$ & 120 \\ \hline

$\E_8$ & $D_4$ & $2\omega_7+2\omega_8$ & 28 \\ \hline

$\E_8$ & $D_5$ & $2\omega_1+2\omega_7+2\omega_8$ & 60 \\ \hline

$\E_8$ & $D_6$ & $2\omega_1+\omega_2+\omega_3+\omega_7+2\omega_8$ & 110 \\ \hline

$\E_8$ & $D_7$ & $2\omega_1+\omega_2+\omega_3+\omega_5+\omega_6+\omega_8$ & 182 \\ \hline

$\E_8$ & $D_8$ & $2\omega_1+2\omega_4+2\omega_6+2\omega_8$ & 280\\ \hline

$\E_8$ & $E_6$ & $\omega_1+\omega_6+\omega_7+\omega_8$ & 156 \\ \hline

$\E_8$ & $E_7$ & $2\omega_1+\omega_2+\omega_3+\omega_5+2\omega_6+2\omega_7+2\omega_8$ & 399 \\ \hline

$\E_8$ &  $E_8$ & $2(\omega_1+\omega_2+\omega_3+\omega_4+\omega_5+\omega_6+\omega_7+\omega_8)$
& 1240\\ \hline

$\E_8$ & $2A_1$ &  $\omega_1$ & 2 \\ \hline

$\E_8$ & $3A_1$ & $\omega_7$ & 3 \\ \hline

$\E_8$ & $4A_1$\footnotemark & $2\omega_8$  & 4\\ \hline

$\E_8$& $(4A_1)'$\footnotemark & $\omega_2$  & 4\\ \hline

$\E_8$ & $5A_1$ &  $\omega_1+\omega_8$ & 5 \\ \hline

$\E_8$ & $6A_1$  & $\omega_6$ & 6 \\ \hline

$\E_8$ & $7A_1$ & $\omega_3$ & 7 \\ \hline

$\E_8$ & $8A_1$ & $2\omega_1$ & 8 \\ \hline

$\E_8$ & $A_2+A_1$ & $\omega_1+\omega_8$ & 5 \\ \hline

$\E_8$ & $A_2+2A_1$ & $\omega_6$ & 6 \\ \hline

$\E_8$ & $A_2+3A_1$  &  $\omega_3$ & 7 \\ \hline

$\E_8$ & $A_2+4A_1$ & $2\omega_1$ & 8 \\ \hline

$\E_8$ & $2A_2$  & $2\omega_1$ & 8 \\ \hline

$\E_8$ & $2A_2+A_1$  & $\omega_1+\omega_7$ & 9 \\ \hline

$\E_8$ & $2A_2+2A_1$ & $\omega_5$& 10\\ \hline

$\E_8$ & $3A_2$  & $2\omega_7$ & 12 \\ \hline

$\E_8$ & $3A_2+A_1$ & $\omega_2+\omega_7$ & 13 \\ \hline

$\E_8$ & $4A_2$ & $2\omega_2$ & 16 \\ \hline

$\E_8$ & $A_3+A_1$  & $\omega_6+\omega_8$ & 11 \\ \hline
\end{tabular}
\end{table}



\begin{table}[ht] \label{Ta:E8-II}
\caption{Formulas of $2\delta'_{\Phi}$}
\centering \begin{tabular}{|c |c |c |c |c |}\hline
$\Psi_0$ & $\Phi$ & $2\delta'_{\Phi}$ & $e(\Phi)$\\ \hline
$\E_8$ & $A_3+2A_1$\footnotemark & $2\omega_7$ & 12 \\ \hline

$\E_8$ & $(A_3+2A_1)'$\footnotemark & $\omega_3+\omega_8$ & 12 \\ \hline

$\E_8$ & $A_3+3A_1$ & $\omega_2+\omega_7$ & 13\\ \hline

$\E_8$ & $A_3+4A_1$  & $\omega_1+\omega_6$ & 14\\ \hline

$\E_8$ & $A_3+A_2$ & $\omega_1+\omega_6$ & 14\\ \hline

$\E_8$ & $A_3+A_2+A_1$ & $\omega_4$ & 15\\ \hline

$\E_8$ & $A_3+A_2+2A_1$  & $2\omega_2$ & 16\\ \hline

$\E_8$ & $2A_3$\footnotemark  & $2\omega_1+2\omega_8$ & 20\\ \hline

$\E_8$ & $(2A_3)'$\footnotemark  & $\omega_1+\omega_5$ & 20 \\ \hline

$\E_8$ & $2A_3+A_1$ & $\omega_1+\omega_6+\omega_8$ & 21\\ \hline

$\E_8$ & $2A_3+2A_1$ & $\omega_4+\omega_8$ & 22\\ \hline

$\E_8$ & $A_4+A_1$  &    $\omega_1+\omega_6+\omega_8$ & 21 \\ \hline

$\E_8$ & $A_4+2A_1$ &  $\omega_4+\omega_8$  & 22\\ \hline

$\E_8$ & $A_4+A_2$ &   $2\omega_6$ & 24 \\ \hline

$\E_8$ & $A_4+A_2+A_1$ & $\omega_3+\omega_6$ & 25 \\ \hline

$\E_8$ & $A_4+A_3$  &  $\omega_4+\omega_7$ & 30 \\ \hline

$\E_8$ & $2A_4$ &  $2\omega_5$ & 40\\ \hline

$\E_8$ & $A_5+A_1$\footnotemark & $2\omega_1+2\omega_7$ & 36 \\ \hline

$\E_8$ & $(A_5+A_1)'$\footnotemark & $\omega_1+\omega_4+\omega_8$ & 36 \\ \hline

$\E_8$ & $A_5+2A_1$  & $\omega_1+\omega_5+\omega_7$ & 37 \\ \hline

$\E_8$ & $A_5+A_2$ &  $\omega_4+\omega_6$  & 39 \\ \hline

$\E_8$ & $A_5+A_2+A_1$  & $2\omega_5$ & 40 \\ \hline

$\E_8$ & $A_6+A_1$ & $\omega_1+\omega_4+\omega_6$ & 57\\ \hline

$\E_8$ & $A_7+A_1$  & $\omega_1+\omega_4+\omega_6+2\omega_8$& 85\\ \hline

$\E_8$ & $D_4+A_1$ & $\omega_2+\omega_7+2\omega_8$  & 29 \\ \hline

$\E_8$ & $D_4+2A_1$ & $\omega_1+\omega_6+2\omega_8$   &  30 \\ \hline

$\E_8$ & $D_4+3A_1$  & $\omega_4+2\omega_8$  & 31 \\ \hline

$\E_8$ & $D_4+4A_1$  & $2\omega_2+2\omega_8$ & 32 \\ \hline

$\E_8$ & $D_4+A_2$ &  $2\omega_2+2\omega_8$ & 32 \\ \hline

$\E_8$ & $D_4+A_3$ & $\omega_2+\omega_3+\omega_7$   & 38 \\ \hline

$\E_8$ & $2D_4$  &  $2\omega_1+2\omega_6$& 56\\ \hline

$\E_8$ & $D_5+A_1$  & $\omega_1+\omega_5+\omega_7+2\omega_8$&  61 \\ \hline

$\E_8$ & $D_5+2A_1$ & $\omega_2+\omega_3+\omega_7+2\omega_8$ & 62\\\hline

$\E_8$ & $D_5+A_2$ & $2\omega_5+2\omega_8$ & 64 \\ \hline

$\E_8$ & $D_5+A_3$ & $\omega_1+\omega_4+\omega_6+\omega_8$ & 70 \\ \hline

$\E_8$ & $D_6+A_1$ & $2\omega_1+\omega_4+\omega_6+2\omega_8$ & 111 \\ \hline

$\E_8$ & $D_6+2A_1$ & $2\omega_1+2\omega_5+2\omega_8$ & 112 \\ \hline

$\E_8$ & $E_6+A_1$ & $\omega_1+\omega_4+\omega_6+2\omega_7+2\omega_8$ & 157 \\ \hline

$\E_8$ & $E_6+A_2$ & $\omega_4+\omega_7+\omega_8$ & 160 \\ \hline

$\E_8$ & $E_7+A_1$ & $2\omega_1+2\omega_4+2\omega_6+2\omega_7+2\omega_8$ &400 \\ \hline
\end{tabular}
\end{table}



\begin{table}\label{Ta:F4}
\caption{Formulas of $2\delta'_{\Phi}$}
\centering
\begin{tabular}{|c|c|c|c|c|c|c|c|c|} \hline
$\Psi_0$ & $\Phi$ & $2\delta'_{\Phi}$ & $e(\Phi)$\\ \hline
$\F_4$ & $A_1^L$ & $\omega_1$ &  2 \\ \hline
$\F_4$ & $A_1^S$ & $\omega_4$ &  1  \\ \hline
$\F_4$ & $A_2^L$ & $2\omega_1$ & 8 \\ \hline
$\F_4$ & $A_2^S$ & $2\omega_4$ & 4 \\ \hline
$\F_4$ & $A_3^L$ & $2\omega_1+2\omega_4$ & 20 \\ \hline
$\F_4$ & $A_3^S$ & $\omega_1+2\omega_4$ &  10 \\ \hline
$\F_4$ & $D_4^L$ & $2\omega_1+2\omega_2$ & 56 \\ \hline
$\F_4$ & $D_4^S$ & $2\omega_3+2\omega_4$ & 28 \\ \hline
$\F_4$ & $B_2$ & $\omega_1+2\omega_4$ &    10 \\ \hline
$\F_4$ & $B_3$ & $2\omega_1+\omega_2+\omega_4$ & 35\\ \hline
$\F_4$ & $C_3$ & $2\omega_3+2\omega_4$ & 28\\ \hline
$\F_4$ & $B_4$ & $2\omega_1+2\omega_2+2\omega_4$ &  84\\ \hline
$\F_4$ & $C_4$ & $2\omega_1+2\omega_3+2\omega_4$ &  60\\ \hline
$\F_4$ & $F_4$ & $2\omega_1+2\omega_2+2\omega_3+2\omega_4$ &  156\\ \hline
$\F_4$ & $2A_1^L$ & $2\omega_4$ &  4\\ \hline
$\F_4$ & $2A_1^S$ & $\omega_1$ &  2\\ \hline
$\F_4$ & $A_1^S+A_1^L$ & $\omega_3$ & 3\\ \hline
$\F_4$ & $3A_1^L$ & $\omega_2$ & 6\\ \hline
$\F_4$ & $3A_1^S$ & $\omega_3$ & 3\\ \hline
$\F_4$ & $A_1^S+2A_1^L$ & $\omega_1+\omega_4$ & 5\\ \hline
$\F_4$ & $2A_1^S+A_1^L$ & $2\omega_4$ &  4\\ \hline
$\F_4$ & $4A_1^L$ & $2\omega_1$ &  8\\ \hline
$\F_4$ & $4A_1^S$ & $2\omega_4$ &  4\\ \hline
$\F_4$ & $2A_1^S+2A_1^L$ & $\omega_2$ & 6\\ \hline
$\F_4$ & $A_2^L+A_1^S$ & $\omega_1+\omega_3$ & 9\\ \hline
$\F_4$ & $A_2^S+A_1^L$ &  $\omega_2$ & 6\\ \hline
$\F_4$ & $A_2^S+A_2^L$ & $2\omega_3$ & 12\\ \hline
$\F_4$ & $B_2+A_1^L$ & $2\omega_3$ & 12\\ \hline
$\F_4$ & $B_2+A_1^S$ & $\omega_2+\omega_4$ & 11\\ \hline
$\F_4$ & $B_2+2A_1^L$ &  $\omega_1+\omega_4$ & 14\\ \hline
$\F_4$ & $B_2+2A_1^S$ & $2\omega_3$ &  12\\ \hline
$\F_4$ & $2B_2$ & $2\omega_1+2\omega_4$ & 20\\ \hline
$\F_4$ & $A_3^S+A_1^L$ & $2\omega_3$ &  12\\ \hline
$\F_4$ & $A_3^L+A_1^S$ & $\omega_1+\omega_2+\omega_4$ & 21\\ \hline
$\F_4$ & $C_3+A_1^L$ & $\omega_1+\omega_2+2\omega_4$ & 30\\ \hline
$\F_4$ & $B_3+A_1^S$ &  $2\omega_1+2\omega_3$ & 36\\ \hline
\end{tabular}
\end{table}



\begin{table}\label{Ta:G2}
\caption{Formulas of $2\delta'_{\Phi}$} 
\centering
\begin{tabular}{|c|c|c|c|c|c|c|c|c|} \hline
$\Psi_0$ & $\Phi$ & $2\delta'_{\Phi}$ & $e(\Phi)$\\ \hline
$\G_2$ & $A_1^L$ & $\omega_1$ & 3\\ \hline
$\G_2$ & $A_1^S$ & $\omega_2$ & 1 \\ \hline
$\G_2$ & $A_2^L$ & $2\omega_1$ & 12 \\ \hline
$\G_2$ & $A_2^S$ & $2\omega_2$ & 4 \\ \hline
$\G_2$ & $G_2$   & $2(\omega_1+\omega_2)$ &  28 \\ \hline
$\G_2$ & $A_1^S+A_1^L$ & $2\omega_2$ & 4 \\ \hline
\end{tabular}
\end{table}




\section{Equalities among dimension data}\label{S:dimension-equal}

In this section, we solve Question \ref{Q:equal-character}. The strategy is as follows. First we reduce it
to the case that the root system $\Psi$ is an irreducible root system. In the case that $\Psi$ if of type $\BC$,
there is an algebra isomorphism $E$ from the ring of characters to the polynomial ring $\bbQ[x_0,x_1,\dots]$
defined in \cite{Larsen-Pink}. Moreover, the polynomials in the image are generated by certain polynomials
$a_{n}$, $b_{n}$, $c_{n}$, $d_{n}$, corresponding to the sub-root systems $\A_{n-1}$, $\B_{n}$, $\C_{n}$ and
$\D_{n}$, respectively. We are able to get all multiplicative relations among these polynomials
$\{a_{n},b_{n},c_{n},d_{n}|\ n\geq 1\}$. In the case that $\Psi$ is of type $\B$, $\C$ or $\D$, the solution in
type $\BC$ case solves the question in this case as well, with a little more consideration if $\Psi=\D_{n}$.
In the case that $\Psi$ is of type $\A$, we show that the characters of non-conjugate sub-root systems have
different leading terms. In the case that $\Psi$ is an exceptional irreducible root system, the formulas of the
leading terms $\{2\delta'_{\Phi}|\ \Phi\subset\Psi\}$ have been given in Section \ref{S:leading terms}. We
give a case by case discussion of sub-root systems with equal leading terms in their characters.

\begin{theorem}\label{T:equal character-nonsimple}
Given a root system $\Psi$ and two reduced sub-root systems $\Phi_1$ and $\Phi_2$,
$F_{\Phi_1,\Aut(\Psi)}=F_{\Phi_2,\Aut(\Psi)}$ if and only if there exists $\gamma\in\Aut(\Psi)$ such
that \[F_{\Phi_{i}^{(1)},\Aut(\Psi_{i})}=F_{\Phi_{i}^{(2)},\Aut(\Psi_{i})}\] for any  $1\leq i\leq m$,
where $\Psi=\bigsqcup_{1\leq i\leq m}\Psi_{i}$ with each $\Psi_{i}$ an irreducible root system,
\[\gamma\Phi_1=\bigsqcup_{1\leq i\leq m}\Phi_{i}^{(1)}\] and
\[\Phi_2=\bigsqcup_{1\leq i\leq m}\Phi_{i}^{(2)}\] with
$\Phi_{i}^{(1)},\Phi_{i}^{(2)}\subset\Psi_{i}$.
\end{theorem}

\begin{proof}
Since $\Aut(\Psi)$ permutes simple factors of $\Psi$ and it permutes two simple factors if and only
if they are isomorphic abstract irreducible root systems, we may assume that $\Psi=m\Psi_0$ and
$\Aut(\Psi)=\Aut(\Psi_0)^{m}\rtimes S_{m}$ where $\Psi_0$ is an irreducible root system and
$m\Psi_0$ denotes the direct sum of $m$ copies of $\Psi_0$.

Denote by $\Lambda=\bbZ\Psi_0$ the root lattice, by $\bbQ[\Lambda]$ the character ring and by
$U=\bbQ[\Lambda]^{\Aut(\Psi_0)}$ the invariant characters. Thus $U$ is
a $\bbQ$ vector space with a basis $\Lambda$. Write \[S(U)=\bigoplus_{n\geq 0} S^{n}(U)\] for the
symmetric tensor algebra over $U$. It is a polynomial algebra with symmetric tensor product as
multiplication. Hence it is a unique factorization domain.

Write $\Phi_{j}=\bigoplus_{1\leq i\leq m}\Phi_{i}^{(j)},j=1,2$. In $S(U)$, we have \[F_{\Phi_{j},
\Aut(\Psi)}=F_{\Phi_1^{(j)},\Aut(\Psi_0)}\cdot F_{\Phi_2^{(j)},\Aut(\Psi_0)}\cdots F_{\Phi_{m-1}^{(j)},
\Aut(\Psi_0)}\cdot F_{\Phi_{m}^{(j)},\Aut(\Psi_0)}.\] As each $F_{\Phi_{i}^{(j)},\Aut(\Psi_0)}$ being of
degree one and having constant term $1$, $F_{\Phi_1,\Aut(\Psi)}=F_{\Phi_2,\Aut(\Psi)}$ if and only if
$\{F_{\Phi_{i}^{(1)},\Aut(\Psi_0)}|\ 1\leq i\leq m\}$ and $\{F_{\Phi_{i}^{(2)},\Aut(\Psi_0)}|\ 1\leq i
\leq m\}$ differ by a permutation. Therefore the conclusion follows.
\end{proof}

\begin{remark}\label{R:equal character-nonsimple}
In the case of $\Psi_0=\BC_1=\{\pm{e_1},\pm{2e_1}\}$, \[\Lambda=\bbZ\Psi_0=\{ne_1:n\in\bbZ\}.\] Write $x_{n}=
\frac{[ne_1]+[-ne_1]}{2}$ for any $n\in\bbZ_{\geq 0}$. In this case $U=\span_{\bbQ}\{x_{n}|\ n\geq 0\}$
and \[S(U)=\bbQ[x_0,x_1,\dots,x_{n},\dots]\] is the usual polynomial algebra. The identification is given
by the map $E$ defined in \cite{Larsen-Pink}.
\end{remark}


Theorem \ref{T:equal character-nonsimple} reduces Question \ref{Q:equal-character} to the case that $\Psi$
is an irreducible root system. We introduce some notation now. Given a root system $\Psi$, recall that we
have defined the integral weight lattice $\Lambda_{\Psi}$ (which is a subset of the rational vector space
spanned by roots of $\Psi$) in the ``notation and conventions`` part.

\begin{definition}
Given an integral weight $\lambda$, let \[\chi_{\lambda,\Psi}^{\ast}=\frac{1}{|W_{\Psi}|}
\sum_{\gamma\in W_{\Psi}}[\gamma\lambda].\]
\end{definition}

In the case that the root system $\Psi$ is clear from the context, we simply write $\chi_{\lambda}^{\ast}$ for
$\chi_{\lambda,\Psi}^{\ast}$. As discussed in Section \ref{S:characters}, the characters
$\{\chi_{\lambda}^{\ast}|\ \lambda\in\Lambda_{\Psi}^{+}\}$ is a basis of of the vector space
$\bbQ[\Lambda_{\Psi}]^{W_{\Psi_0}}$, where $\Lambda_{\Psi}^{+}$ is the set of dominant weights in
$\Lambda_{\Psi}$ (with respect to a positive system determined by a chosen order).

Combining Theorems \ref{T:character-classical} and \ref{character-exceptional}, we get the following theorem.
It answers Question \ref{Q:equal-character} completely in the case that $\Psi$ is an irreducible root system.

\begin{theorem}\label{T:equal character-simple}
Given an irreducible root system $\Psi$, if there exist two non-conjugate reduced sub-root systems
$\Phi_1,\Phi_2\subset\Psi$ with $F_{\Phi_1,\Aut(\Psi)}=F_{\Phi_2,\Aut(\Psi)}$,
then $\Psi\cong\C_{n}$, $\BC_{n}$ or $\F_4$.

In the case that $\Psi=\C_{n}$ or $\BC_{n}$, $F_{\Phi_1,\Aut(\Psi)}=F_{\Phi_2,\Aut(\Psi)}$ if and only if
\[b_{m}(\Phi_1)-b_{m}(\Phi_2)=a_{2m}(\Phi_1)-a_{2m}(\Phi_2)=0\] and
\[a_{2m-1}(\Phi_1)-a_{2m-1}(\Phi_2)=c_{m-1}(\Phi_2)-c_{m-1}(\Phi_1)=d_{m}(\Phi_2)-d_{m}(\Phi_1)\]
for any $m\geq 1$. Here $a_{m}(\Phi_{i})$, $b_{m}(\Phi_{i})$, $c_{m}(\Phi_{i})$, $d_{m}(\Phi_{i})$
is the number of simple factors of $\Phi_{i}\subset\Psi\subset\BC_{n}$ isomorphic to
$\A_{m-1}$, $\B_{m}$, $\C_{m}$ or $\D_{m}$, respectively.

In the case that $\Psi=\F_4$, $F_{\Phi_1,\Aut(\Psi)}=F_{\Phi_2,\Aut(\Psi)}$ if and only if
\[\Phi_1\sim\Phi_2,\] \[\{\Phi_1,\Phi_2\}\sim\{A_2^{S},A_1^{L}+2A_1^{S}\}\] or
\[\{\Phi_1,\Phi_2\}\sim\{A_1^{L}+A_2^{S}, 2A_1^{L}+2A_1^{S}\}.\]
\end{theorem}

\subsection{Classical irreducible root systems}\label{SS:Character-classical}

As in Section 3 of \cite{Larsen-Pink}, let
\begin{eqnarray*}&& \bbZ^{n}:=\bbZ\BC_{n}=\Lambda_{\BC_{n}}=\span_{\bbZ}\{e_1,e_2,...,e_{n}\},
\\&& W_{n}:=\Aut(\BC_{n})=W_{\BC_{n}}=\{\pm{1}\}^{n}\rtimes S_{n},\\&& \bbZ_{n}:=\bbQ[\bbZ^{n}],
\\&& Y_{n}:=\bbZ_{n}^{W_{n}}. \end{eqnarray*} For $m\leq n$, the injection
\[\bbZ^{m}\hookrightarrow\bbZ^{n}:(a_1,...,a_{m})\mapsto(a_1,...,,a_{m},0,...,0)\] extends to
an injection $i_{m,n}:\bbZ_{m}\hookrightarrow\bbZ_{n}$. Define
$\phi_{m,n}:\bbZ_{m}\rightarrow\bbZ_{n}$ by
\[\phi_{m,n}(z)=\frac{1}{|W_{n}|}\sum_{w\in W_{n}}w(i_{m,n}(z)).\]
Thus $\phi_{m,n}\phi_{k,m}=\phi_{k,n}$ for any $k\leq m\leq n$ and the image of
$\phi_{m,n}$ lies in $Y_n$. Hence $\{Y_{m}:\phi_{m,n}\}$ forms a direct system and
we define \[Y=\lim_{\longrightarrow_{n}} Y_{n}.\] Define the map
$j_{n}: \bbZ_{n}\rightarrow Y$ by composing $\phi_{n,p}$ with the injection
$Y_{p}\hookrightarrow Y$. The isomorphism $\bbZ^{m}\oplus\bbZ^{n}\longrightarrow\bbZ^{m+n}$
gives a canonical isomorphism $M: \bbZ_{m}\otimes_{\bbQ}\bbZ_{n}\longrightarrow\bbZ_{m+n}$.
Given two elements of $Y$ represented by $y\in Y_{m}$ and $y'\in Y_{n}$ we define
\[yy'=j_{m+n}(M(y\otimes y')).\] This product is independent of the choice of $m$ and $n$ and
makes $Y$ a commutative associative algebra. The monomials
$[e_1]^{a_1}\cdots[e_{n}]^{a_{n}}$ ($a_1,a_2,\cdots,a_{n}\in\bbZ$) form a
$\bbQ$ basis of $\bbZ_{n}$, where $[e_i]^{a_i}=[a_ie_i]\in\bbZ_1$ is a linear character.
Hence $Y$ has a $\bbQ$ basis \[e(a_1,a_2,...,a_{n})=j_{n}([e_1]^{a_1}\cdots[e_{n}]^{a_{n}})\]
indexed by $n\geq 0$ and $a_1\geq a_2\geq\cdots\geq a_{n}\geq 0$. Mapping $e(a_1,a_2,...,a_{n})$
to $x_{a_1}x_{a_2}\cdots x_{a_{n}}$, we get a $\bbQ$ linear map
\[E: Y\longrightarrow \bbQ[x_0,x_1,...,x_{n},...].\]

\begin{lemma}(\cite{Larsen-Pink} Page 390)\label{Isomorphism E}
The above map $E$ is an algebra isomorphism.
\end{lemma}

\begin{proof}
The bijectivity is clear. It is an isomorphism follows from the equality
\[e(a_1,a_2,...,a_{m})e(b_1,b_2,...,b_{n})=e(c_1,c_2,...,c_{m+n}),\] where $(c_1,c_2,...,c_{m+n})$ is
the re-permutation of $\{a_1,...,a_{m},b_1,...,b_{n}\}$ in deceasing order. This equlity can be proved
by a combinatorial calculation.
\end{proof}

\begin{definition}\label{D:abcd}
Write $a_{n}$, $b_{n}$, $c_{n}$, $d_{n}$ for the image of $j_{n}(F_{\Phi,W_{n}})$ under $E$ for
$\Phi=\A_{n-1}$, $\B_{n}$, $\C_{n}$ or $\D_{n}$, respectively.
\end{definition}

Observe that $a_{n}$, $b_{n}$, $c_{n}$, $d_{n}$ are homogeneous polynomials of degree $n$ with integer
coefficients and a term $x_0^{n}$.

\begin{remark}\label{R:x0=1}
Here, $E$ maps linear characters to homogeneous polynomials. Letting $x_0=1$, then the definition here
becomes that in \cite{Larsen-Pink}. In this terminology, each of $a_{n}$, $b_{n}$, $c_{n}$, $d_{n}$ has
integer coefficients and constant term $1$.
\end{remark}

Note that our $b_{n},c_{n},d_{n}$ are $b'_{n},c'_{n},d'_{n}$ in \cite{Larsen-Pink}. For any
$a_1\geq a_2\geq \cdots\geq a_{n}\geq 0$ and $\lambda=a_1e_1+a_2e_2+\cdots+a_{n}e_{n}$, one sees
that \[\chi^{\ast}_{\lambda,W_{n}}=[e_1]^{a_1}\cdots[e_{n}]^{a_{n}}\] in $Y$. Thus
$E(\chi^{\ast}_{\lambda,W_{n}})=x_{a_1}x_{a_2}\cdots x_{a_{n}}$.

For small $n$, we have $a_{1}=d_1=x_0$, $b_1=x_0-x_1$, $c_1=x_0-x_2$,
\[a_2=x_0^{2}-x_1^{2},\] \[b_2=x_0^{2}-x_0x_1-x_0x_3+x_1x_3-x_1^{2}+2x_1x_2-x_2^{2},\]
\[c_2=x_0^{2}-x_0x_2-x_0x_4+x_2x_4-x_1^{2}+2x_1x_3-x_3^{2},\] \[d_2=x_0^{2}-2x_1^{2}+x_0x_2\]
and $a_3=x_0^{3}-2x_0x_1^{2}+2x_1^{2}x_2-x_0x_2^{2}$.

For the convenience in writing notations, we define $a_0=b_0=c_0=d_0=1$ and $c_{-1}=x_0^{-1}$.

\begin{proposition}(\cite{Larsen-Pink}, Page 390)\label{basic properties}
\begin{itemize}
\item[(1)]{We have $c_{n},d_{n+1}\in\bbQ[x_1,x_2,...,x_{2n}]-\bbQ[x_1,x_2,...,x_{2n-1}]$,\\
$b_{n}\in\bbQ[x_1,x_2,...,x_{2n-1}]-\bbQ[x_1,x_2,...,x_{2n-2}]$.}
\item[(2)]{Each of $b_{n},c_{n},d_{n+1}$ is a prime in $\bbQ[x_1,x_2,...]$
and any two of them are different.}
\item[(3)]{Each of the subsets $\{b_1,...,b_{n},c_1,...,c_{n}\}$,
$\{b_1,...,b_{n},d_2,...,d_{n+1}\}$,\\$\{c_1,...,c_{n},d_2,...,d_{n+1}\}$
is algebraically independent.}
\end{itemize}
\end{proposition}

Given $f\in\bbQ[x_0,x_1,...]$, let \[\sigma(f)(x_0,x_1,...,x_{2n},x_{2n+1},...)=
f(x_0,-x_1,...,x_{2n},-x_{2n+1},...).\] Then $\sigma$ is an involutive automorphism of
$\bbQ[x_0,x_1,...]$.

\begin{proposition}\label{P:abcd-symmetry}
We have $\sigma(a_{n})=a_{n}$, $\sigma(c_{n})=c_{n}$, $\sigma(d_{n})=d_{n}$ for any $n$ and
but $\sigma(b_{n})\neq b_{n}$ when $n\geq 1$.
\end{proposition}

\begin{proof}
This follows from the formula
\[F_{\Phi,W_{n}}=\sum_{w\in W_{\Phi}}\epsilon(w)\chi^{\ast}_{\delta_{\Phi}-w\delta_{\Phi},W_{n}}\]
and the expression of $\delta_{\Phi}$ for $\Phi=\A_{n-1},\B_{n},\C_{n},\D_{n}$.
\end{proof}

\begin{definition}\label{D:bn}
Define $b'_{n}=\sigma(b_{n})$.
\end{definition}

\begin{proposition}\label{P:A=BB and A=CD}
For any $n\geq 1$, we have $a_{2n}=b_{n}b'_{n}$ and $a_{2n+1}=c_{n}d_{n+1}$.
\end{proposition}

\begin{proof}
Given $n\geq 1$, define the matrices $A_{n}=(x_{|i-j|})_{n\times n}$,
\begin{eqnarray*}&&
B_{n}=(x_{|i-j|}-x_{i+j-1})_{n\times n},\\&&
C_{n}=(x_{|i-j|}-x_{i+j})_{n\times n},\\&& B'_{n}=(x_{|i-j|}+x_{i+j-1})_{n\times n},\\&&
D_{n}=(x_{|i-j|}+x_{i+j-2})_{n\times n},\\&& D'_{n}=(a_{i,j})_{n\times n}, \end{eqnarray*}
where $a_{i,j}=x_{|i-j|}+x_{i+j-2}$ if $i,j\geq 2$, $a_{1,j}=a_{j,1}=\sqrt{2}x_{j-1}$,
$a_{1,1}=1$. Then we have the following equalities:
\begin{equation}\label{Eq:an}
a_{n}=\det A_{n},
\end{equation}
\begin{equation}\label{Eq:bn}
b_{n}=\det B_{n},
\end{equation}
\begin{equation}\label{Eq:bn2}
b'_{n}=\det B'_{n},
\end{equation}
\begin{equation}\label{Eq:cn}
c_{n}=\det C_{n},
\end{equation}
\begin{equation}\label{Eq:dn}
d_{n}=\frac{1}{2}\det D_{n}=\det D'_{n}.
\end{equation}

We prove the equality (\ref{Eq:an}). The others can be proved similarly. Recall that
\[F_{\A_{n-1},W_{n}}=\sum_{w\in S_{n}}\epsilon(w)\chi^{\ast}_{\delta-w\delta,W_{n}},\] where
\begin{eqnarray*}\delta &=&\frac{(n-1)e_1+(n-3)e_2+\cdots+(1-n)e_{n}}{2}\\&=&
(ne_1+(n-1)e_2+\cdots+e_{n})-\frac{n+1}{2}(e_1+e_2+\cdots+e_{n}).\end{eqnarray*}
Let $\delta'=ne_1+(n-1)e_2+\cdots+e_{n}$. Then
\[F_{\A_{n-1},W_{n}}=\sum_{w\in S_{n}}\epsilon(w)\chi^{\ast}_{\delta'-w\delta',W_{n}}.\] Expanding
$\det A_{n}=\det (x_{|i-j|})_{n\times n}$ into the sum of terms according to permutations, a term
corresponding to a permutation $w\in S_{n}$ is equal to the polynomial
$E(j_{n}(\chi^{\ast}_{\delta'-w\delta',W_{n}}))$. Summing up all terms, we get $a_{n}=\det A_{n}$.

Let \[L_{n}=\left(\begin{array}{ccc}&&1\\&\ddots&\\1&&\end{array}\right),\]
\[J_{2n}=\left(\begin{array}{cc}\frac{1}{\sqrt{2}}I_{n}&-\frac{1}{\sqrt{2}}L_{n}\\
\frac{1}{\sqrt{2}}L_{n}&\frac{1}{\sqrt{2}}I_{n} \end{array}\right)\] and
\[J_{2n+1}=\left(\begin{array}{ccc}\frac{1}{\sqrt{2}}I_{n}&\vdots&-\frac{1}{\sqrt{2}}L_{n}
\\\cdots&1&\vdots\\\frac{1}{\sqrt{2}}L_{n}&\vdots&\frac{1}{\sqrt{2}}I_{n} \end{array}\right).\]

By matrix calculation, we have \begin{eqnarray*} &&J_{2n}A_{2n}J_{2n}^{-1}=\left(\begin{array}
{cc}B_{n}&\\&B'_{n}\end{array}\right),\\&& J_{2n+1}A_{2n+1}J_{2n+1}^{-1}=\left(\begin{array}
{cc}C_{n}&\\&D'_{n+1}\end{array} \right).\end{eqnarray*}
Taking determinants, we get $a_{2n}=b_{n}b'_{n}$ and $a_{2n+1}=c_{n}d_{n+1}$.
\end{proof}

\begin{corollary} \label{multiplicative relation, polynomials}
Any multiplicative relation among $\{a_{n+1},b_{n},c_{n},d_{n+1}|n\geq 1\}$ is generated
by $\{a_{2n+1}=c_{n}d_{n+1}: n\geq 1\}$.
\end{corollary}

\begin{proof}This follows from Proposition \ref{basic properties} and Proposition
\ref{P:A=BB and A=CD}.
\end{proof}

\begin{theorem}\label{T:character-classical}
Given a classical irreducible root system $\Psi$, if there exists non-conjugate sub-root systems
$\Phi_1$ and $\Phi_2$ of $\Psi$ such that $F_{\Phi_1,\Aut(\Psi)}=F_{\Phi_2,\Aut(\Psi)}$, then
$\Psi\cong\C_{n}$ or $\BC_{n}$.

In the case that $\Psi=\C_{n}$ or $\BC_{n}$, $F_{\Phi_1,\Aut_{\Psi}}=F_{\Phi_2,\Aut(\Psi)}$ if and
only if \[\forall m\leq n, b_{m}(\Phi_1)-b_{m}(\Phi_2)=a_{2m}(\Phi_1)-a_{2m}(\Phi_2)=0\]
\[ \textrm{and }a_{2m+1}(\Phi_1)-a_{2m+1}(\Phi_2)=c_{m}(\Phi_2)-c_{m}(\Phi_1)=d_{m+1}(\Phi_2)-
d_{m+1}(\Phi_1).\] Here $a_{m}(\Phi_{i}),b_{m}(\Phi_{i}),c_{m}(\Phi_{i}),d_{m}(\Phi_{i})$ is the
number of simple factors of $\Phi_{i}\subset\BC_{n}$ isomorphic to $\A_{m-1},\B_{m},\C_{m},\D_{m}$
respectively.
\end{theorem}

Note that, we embed $\A_{m-1},\B_{m},\C_{m},\D_{m}$ into $\BC_{n}$ ($n\geq m$) in the standard way.
This means that the sub-root systems $\A_1$, $\B_1$, $\C_1$ are non-isomorphic to each other, as well as
each of the pairs $(\B_2,\C_2)$, $(\D_2,\A_1\bigsqcup\A_1)$ and $(\A_3,\D_3)$.

\begin{proof}[Proof of Theorem \ref{T:character-classical}]
In the case of $\Psi=\C_{n}$ or $\BC_{n}$, since $j_{n}: Y_{n}\longrightarrow Y$ is an injection,
$F_{\Phi_1,W_{n}}=F_{\Phi_2,W_{n}}$ if and only if $E(F_{\Phi_1,W_{n}})=E(F_{\Phi_2,W_{n}})$, i.e.
\[\prod_{1\leq i\leq n}(a_i^{r_{i}^{(1)}}b_{i}^{s_{i}^{(1)}}c_{i}^{u_{i}^{(1)}}d_{i}^{v_{i}^{(1)}})
=\prod_{1\leq i\leq n}(a_{i}^{r_{i}^{(2)}}b_{i}^{s_{i}^{(2)}}c_{i}^{u_{i}^{(2)}}d_{i}^{v_{i}^{(2)}}).\]
Here we write $r_{m}^{(j)}=a_{m}(\Phi_{j})$, $s_{m}^{(j)}=b_{m}(\Phi_{j})$, $u_{m}^{(j)}=c_{m}(\Phi_{j})$,
$v_{m}^{(j)}=d_{m}(\Phi_{j})$. Therefore the conclusion follows from Corollary
\ref{multiplicative relation, polynomials}.

In the case that $\Psi=\B_{n}$ ($n\geq 1$) or $\D_{n}$ ($n\geq 5$), $\Aut(\Psi_0)=W_{n}$. Since any $\C_{k}$ is
not contained in $\B_{n}$ or $\D_{n}$, the conclusion follows from the conclusion for $\BC_{n}$ case.

In the case that $\Psi=\D_4$, only the characters of the non-conjugate sub-root systems $A_2,4A_1$ have the
equal leading term, which is $\chi_{2\omega_2}$. We have \[F_{A_2,W_{\Psi}}=
1-2\chi^{\ast}_{\omega_2}+2\chi^{\ast}_{\omega_1+\omega_3+\omega_4}-\chi^{\ast}_{2\omega_2}\] and
\[F_{4A_1,W_{\Psi}}=1-4\chi^{\ast}_{\omega_2}+2(\chi^{\ast}_{2\omega_1}+\chi^{\ast}_{2\omega_3}+
\chi^{\ast}_{2\omega_4})-4\chi^{\ast}_{\omega_1+\omega_3+\omega_4}+\chi^{\ast}_{2\omega_2}.\] Thus
$F_{A_2,W_{\Psi}}\neq F_{4A_1,W_{\Psi}}$. Therefore the conclusion follows.

In the case that $\Psi=\A_{n-1}$, $F_{\Phi_1,\Aut_{\Psi}}=F_{\Phi_2,\Aut_{\Psi}}$ implies that
$2\delta_{\Phi_1}\sim_{\Aut(\Psi)}2\delta_{\Phi_2}$. And the latter implies that
$\Phi_1\sim_{\Aut(\Psi)}\Phi_2$. Therefore the conclusion follows.
\end{proof}


\subsection{Exceptional irreducible root systems}\label{SS:Character-exceptional}

In Section \ref{S:leading terms}, we give the formulas of $2\delta'_{\Phi}$ and their modulus squares
for reduced sub-root systems of the exceptional simple root systems $\Psi=\E_6$, $\E_7$, $\E_8$, $\F_4$
or $\G_2$. Using these formulas, in this section we classify non-conjugate sub-root systems
$\Phi_1,\Phi_2$ of any exceptional irreducible root system $\Psi$ such that
\[F_{\Phi_1,W_{\Psi}}=F_{\Phi_2,W_{\Psi}}.\]



\begin{theorem} \label{character-exceptional}
Given an irreducible root system $\Psi=\E_6$, $\E_7$, $\E_8$, $\F_4$ or $\G_2$, if there exists non-conjugate
reduced sub-root systems $\Phi_1$ and $\Phi_2$ of $\Psi$ such that $F_{\Phi_1,\Aut(\Psi)}=F_{\Phi_2,\Aut(\Psi)}$,
then $\Psi\cong F_4$.

In the case that $\Psi=\F_4$, $F_{\Phi_1,\Psi}=F_{\Phi_2,\Psi}$ and $\Phi_1\not\sim\Phi_2$ if and only if
\[\{\Phi_1,\Phi_2\}\sim\{A_2^{S},A_1^{L}+2A_1^{S}\}\textrm{ or }\{A_1^{L}+A_2^{S}, 2A_1^{L}+2A_1^{S}\}.\]
\end{theorem}

\begin{proof}
In the case that $\Psi=\E_6$, among the dominant integral weights appearing in $\{2\delta'_{\Phi}|\ \Phi\subset\E_6\}$,
those weights that appearing more than once include $\{2\omega_2,2\omega_4\}$ and the sub-root
systems $\Phi$ with $2\delta_{\Phi}$ conjugate to them are
\begin{itemize}
\item[(1)]{$2\omega_2$: $A_2$, $4A_1^{S}$, appears 2 times.}
\item[(2)]{$2\omega_4$: $3A_2$, $A_3+2A_1$, appears 2 times.}
\end{itemize} The coefficients of $\chi^{\ast}_{\omega_2}$ in
$F_{A_2,\Aut(\E_6)},F_{4A_1,\Aut(\E_6)}$ are different and the coefficients of
$\chi^{\ast}_{\omega_2}$ in $F_{3A_2,\Aut(\E_6)},F_{A_3+2A_1,\Aut(\E_6)}$ are also different.
Therefore the conclusion in the $\E_6$ case follows.

\smallskip

In the case that $\Psi=\E_7$, among the dominant integral weights appearing in
$\{2\delta'_{\Phi}|\ \Phi\subset\E_7\}$, those weights that appearing more than once include
\[\{2\omega_1, \omega_1+\omega_6, \omega_4, 2\omega_2, 2\omega_3, 2\omega_1+2\omega_6,
\omega_1+\omega_4+\omega_6\}\] and the sub-root systems $\Phi$ with $2\delta_{\Phi}$ conjugate
to them are \begin{itemize}
\item[(1)]{$2\omega_1$: $A_2$, $4A'_1$, appears two times.}
\item[(2)]{$\omega_1+\omega_6$: $A_2+A_1$, $5A_1$, appears two times.}
\item[(3)]{$\omega_4$: $A_2+2A_1$, $6A_1$, appears two times.}
\item[(4)]{$2\omega_2$: $A_2+3A_1$, $7A_1$, appears two times.}
\item[(5)]{$2\omega_3$: $3A_2$, $A_3+2A_1$, appears two times.}
\item[(6)]{$2\omega_1+2\omega_6$: $A_4$, $2A_3$, appears two times.}
\item[(7)]{$\omega_1+\omega_4+\omega_6$: $A_4+A_1$, $2A_3+A_1$, appears two times.}
\end{itemize}

Given a weight $\lambda$, there are at most two conjugacy classes of sub-root systems $\Phi_1,\Phi_2$ of
$\E_7$ such that $2\delta'_{\Phi}=\lambda$. For each of such $\lambda$, the numbers of simple roots in
$\Phi_1,\Phi_2$ are non-equal. Thus the coefficients of $\chi^{\ast}_{\omega_1}$ in
$F_{\Phi_1,W_{\E_7}}$ and $F_{\Phi_2,W_{\E_7}}$ are different. Therefore the conclusion in the $\E_7$ case
follows.

\smallskip

In the case that $\Psi=\E_8$, among the dominant integral weights appearing in $\{2\delta'_{\Phi}|\ \Phi\subset\E_8\}$,
those weights that appearing more than once include \begin{eqnarray*}&& \{2\omega_8, \omega_1+\omega_8,
\omega_6, \omega_3, 2\omega_1, 2\omega_7, \omega_2+\omega_7, \omega_1+\omega_6, 2\omega_2,
2\omega_1+2\omega_8, \omega_1+\omega_6+\omega_8, \\&& \omega_4+\omega_8, 2\omega_2+2\omega_8,
2\omega_5, 2\omega_1+2\omega_6\}\end{eqnarray*} and the sub-root systems $\Phi$ with
$2\delta_{\Phi}$ conjugate to them are
\begin{itemize}
\item[(1)]{$2\omega_8$: $A_2$, $4A'_1$, appears two times.}
\item[(2)]{$\omega_1+\omega_8$: $A_2+A_1$, $5A_1$, appears two times.}
\item[(3)]{$\omega_6$: $A_2+2A_1$, $6A_1$, appears two times.}
\item[(4)]{$\omega_3$: $A_2+3A_1$, $7A_1$, appears two times.}
\item[(5)]{$2\omega_1$: $2A_2$, $A_2+4A_1$, $8A_1$, appears three times.}
\item[(6)]{$2\omega_7$: $A_3+2A_1$, $3A_2$, appears two times.}
\item[(7)]{$\omega_2+\omega_7$: $A_3+3A_1$, $3A_2+A_1$, appears two times.}
\item[(8)]{$\omega_1+\omega_6$: $A_3+4A_1$, $A_3+A_2$, appears two times.}
\item[(9)]{$2\omega_2$: $A_3+A_2+2A_1$, $4A_2$, appears two times.}
\item[(10)]{$2\omega_1+2\omega_8$: $A_4$, $2A_3$, appears two times.}
\item[(11)]{$\omega_1+\omega_6+\omega_8$: $A_4+A_1$, $2A_3+A_1$, appears two times.}
\item[(12)]{$\omega_4+\omega_8$: $A_4+2A_1$, $2A_3+2A_1$, appears two times.}
\item[(13)]{$2\omega_2+2\omega_8$: $D_4+A_2$, $D_4+4A_1$, appears two times.}
\item[(14)]{$2\omega_5$: $A_5+A_2+A_1$, $2A_4$, appears two times.}
\item[(15)]{$2\omega_1+2\omega_6$: $A_6$, $2D_4$, appears two times.}
\end{itemize}

One sees: for any two non-conjugate sub-root systems $\Phi_1,\Phi_2\subset\E_8$ with
$2\delta'_{\Phi_1}=2\delta'_{\Phi_2}$, the numbers of simple roots of $\Phi_1,\Phi_2$ are non-equal,
so the coefficients of $\chi^{\ast}_{\omega_1}$ in $F_{\Phi_1,W_{\E_8}},F_{\Phi_2,W_{\E_8}}$ are
different. Thus $F_{\Phi_1,W_{\E_8}}\neq F_{\Phi_2,W_{\E_8}}$. Therefore the conclusion in the $\E_8$
case follows.

\smallskip

In the case that $\Psi=\F_4$, among the dominant integral weights appearing in $\{2\delta'_{\Phi}:\Phi\subset\F_4\}$,
those weights that appearing more than once include \[\{\omega_1,\omega_3,2\omega_4,\omega_2,
2\omega_1,\omega_1+2\omega_4,2\omega_3,2\omega_1+2\omega_4,2\omega_3+2\omega_4\}\]
and the sub-root systems $\Phi$ with $2\delta_{\Phi}$ conjugate to them are
\begin{itemize}
\item[(1)]{$\omega_1$: $A_1^{L}$, $2A_1^{S}$, appears 2 times.}
\item[(2)]{$\omega_3$: $A_1^{L}+A_1^{S}$, $3A_1^{S}$, appears 2 times.}
\item[(3)]{$2\omega_4$: $A_2^{S}$, $2A_1^{L}$, $A_1^{L}+2A_1^{S}$, $4A_1^{S}$,
appears 4 times.}
\item[(4)]{$\omega_2$: $3A_1^{L}$, $2A_1^{L}+2A_1^{S}$, $A_1^{L}+A_2^{S}$,
appears 3 times.}
\item[(5)]{$2\omega_1$: $A_2^{L}$, $4A_1^{L}$, appears 2 times.}
\item[(6)]{$\omega_1+2\omega_4$: $A_3^{S}$, $B_2$, appears 2 times.}
\item[(7)]{$2\omega_3$: $A_2^{L}+A_2^{S}$, $A_1^{L}+B_2$, $2A_1^{S}+B_2$,
$A_1^{L}+A_3^{S}$, appears 4 times.}
\item[(8)]{$2\omega_1+2\omega_4$: $A_3^{L}$, $2B_2$, appears 2 times.}
\item[(9)]{$2\omega_3+2\omega_4$: $D_4^{S}$, $C_3$, appears 2 times.}
\end{itemize}

The non-conjugate pairs of sub-root system $\Phi_1,\Phi_2\subset\F_4$ with conjugate
leading terms $2\delta'_{\Phi_{i}}$ and the same number of short simple roots are
$(A_2^{S},A_1^{L}+2A_1^{S})$, $(2A_1^{L}+2A_1^{S},A_1^{L}+A_2^{S})$,
$(A_2^{L},4A_1^{L})$, $(2A_1^{S}+B_2,A_1^{L}+A_3^{S})$. The coefficients of
shortest terms in $F_{A_2^{L},F_4},F_{4A_1^{L},F_4}$ are non-equal,
and $F_{2A_1^{S}+B_2}$ is not equal to $F_{A_1^{L}+A_3^{S}}$ since
\begin{eqnarray*}F_{2A_1^{S}+B_2,W_{\F_4}}&=&1-3\chi^{\ast}_{\omega_4}+
2\chi^{\ast}_{\omega_1}+\chi^{\ast}_{\omega_3}-
\chi^{\ast}_{2\omega_4}+2\chi^{\ast}_{\omega_1+\omega_4}-4\chi^{\ast}_{\omega_2}+
2\chi^{\ast}_{\omega_3+\omega_4}-\\&&\chi^{\ast}_{2\omega_1}
+2\chi^{\ast}_{\omega_1+\omega_3}-\chi^{\ast}_{3\omega_4}+2\chi^{\ast}_{\omega_1+2\omega_4}
-3\chi^{\ast}_{\omega_2+\omega_4}+\chi^{\ast}_{2\omega_3}\end{eqnarray*} and
\begin{eqnarray*}F_{A_1^{L}+A_3^{S},W_{\F_4}}&=&1-3\chi^{\ast}_{\omega_4}+7\chi^{\ast}_{\omega_3}-
3\chi^{\ast}_{2\omega_4}-6\chi^{\ast}_{\omega_1+\omega_4}+6\chi^{\ast}_{\omega_3+\omega_4}
+3\chi^{\ast}_{2\omega_1}-\\&&6\chi^{\ast}_{\omega_1+\omega_3}-\chi^{\ast}_{3\omega_4}+
3\chi^{\ast}_{\omega_2+\omega_4}-\chi^{\ast}_{2\omega_3}. \end{eqnarray*}
Calculation shows that \[F_{A_2^{S},W_{\F_4}}=F_{A_1^{L}+2A_1^{S},W_{\F_4}}=1-2\chi^{\ast}_{\omega_4}+
2\chi^{\ast}_{\omega_3}-\chi^{\ast}_{2\omega_4}\] and \[F_{A_1^{L}+A_2^{S},W_{\F_4}}=
F_{2A_1^{L}+2A_1^{S},W_{\F_4}}=1-2\chi^{\ast}_{\omega_4}-\chi^{\ast}_{\omega_1}+4\chi^{\ast}_{\omega_3}-
\chi^{\ast}_{2\omega_4}-2\chi^{\ast}_{\omega_1+\omega_4}+\chi^{\ast}_{\omega_2}.\]
Therefore the conclusion in the $\F_4$ case follows.

\smallskip

In the case that $\Psi=\G_2$, the only non-conjugate pair $(\Phi_1,\Phi_2)$ of sub-root systems such that
$2\delta'_{\Phi_1}=2\delta'_{\Phi_2}$ is $(A_2^{S},A_1^{L}+A_1^{S})$. The numbers of short simple
roots of $A_2^{S}$ and $A_1^{L}+A_1^{S}$ are different, so the coefficients of
$\chi^{\ast}_{\omega_2}$ in $F_{A_2^{S},W_{G_2}}, F_{A_1^{L}+A_1^{S},W_{G_2}}$ are non-equal.
Thus  $F_{\Phi_1,W_{\G_2}}\neq F_{\Phi_2,W_{\G_1}}$. Therefore the conclusion in the $\G_2$ case follows.
\end{proof}

\section{Linear relations among dimension data}\label{S:dimension-dependent}

In this section, we solve Question \ref{Q:dependent-character}.

First, once we know all linear relations among $\{F_{\Phi,W_{\Psi}}|\ \Phi\subset\Psi\}$, we also
know all linear relations among $\{F_{\Phi,W}|\ \Phi\subset\Psi\}$ for any finite group $W$ between
$W_{\Psi}$ and $\Aut(\Psi)$. So we just need to consider the linear relations among
$\{F_{\Phi,W_{\Psi}}|\ \Phi\subset\Psi\}$.

If $\Psi$ is not an irreducible root system, let \[\Psi=\bigsqcup_{1\leq i\leq s}\Psi_{i}\] be
the decomposition of $\Psi$ into a direct sum of simple root systems. For a reduced sub-root
system $\Phi$ of $\Psi$, $\Phi$ can be written as $\Phi=\bigsqcup_{1\leq i\leq s}\Phi_{i}$,
where $\Phi_{i}\subset\Psi_{i}$ for any $1\leq i\leq s$. Thus we have
\[F_{\Phi,W_{\Psi}}=F_{\Phi_1,W_{\Psi_1}}\otimes\cdots\otimes F_{\Phi_{s},W_{\Psi_{s}}}.\]
From this, we see that linear relations among $\{F_{\Phi,W_{\Psi}}|\ \Phi\subset\Psi\}$ arise
from linear relations among $\{F_{\Phi_{i},W_{\Psi_{i}}}|\ \Phi_{i}\subset\Psi_{i}, 1\leq i\leq s\}$.
Hence it is sufficient to consider $\{F_{\Phi,W_{\Psi}}|\ \Phi\subset\Psi\}$ for reduced sub-root systems
$\{\Phi|\ \Phi\subset\Psi\}$ of an irreducible root system $\Psi$.

\begin{remark}\label{R:linear2}
The passing from $\{F_{\Phi,W}|\Phi\subset\Psi\}$ to $\{F_{\Phi,W_{\Psi}}|\Phi\subset\Psi\}$
is like that: any linear relation among the former is also a linear relation among the latter;
and any linear relation among the latter gives a linear relation among the former after the
$W$-averaging process, that is to replace a character $F_{\Phi,W_{\Psi}}\in\bbQ[\Lambda]$ by
\[F_{\Phi,W}=\frac{1}{|W|}\sum_{\gamma\in w}\gamma F_{\Phi,W_{\Psi}}.\] But the relation between
the two sets of linear relations might not be described explicitly, there are at least two reasons
for this. The first reason is the number of distinct characters in
$\{\gamma F_{\Phi,W_{\Psi}}:\gamma\in W/W_{\Psi}\}$ may vary; the second and the more serious reason
is different linear relations among $\{F_{\Phi,W_{\Psi}}|\Phi\subset\Psi\}$ may give the same
linear relation among $\{F_{\Phi,W}|\Phi\subset\Psi\}$ after the $W$-averaging process.

The passing from $\{F_{\Phi,W_{\Psi}}|\Phi\subset\Psi\}$ to $\{F_{\Phi_{i},W_{\Psi_{i}}}|\Phi_{i}
\subset\Psi_{i}\}$ is reasonably well, since by Linear Algebra all linear relations among the former
can be explicitly expressed in terms of linear relations among the latter.
\end{remark}


\begin{remark}\label{R:equal2}
Another way of getting all linear relations among $\{F_{\Phi,W}|\Phi\subset\Psi\}$ is to express
each $F_{\Phi,W}$ in terms of $\{F_{\Phi_{i},W_{\Psi_{i}}}\}$ by the tensor operation and the
$W$-averaging operation. Starting from all linear linear relations among
$\{F_{\Phi_{i},W_{\Psi_{i}}}:\Phi_{i}\subset\Psi_{i}\}$, in this way we can get all linear
relations among $\{F_{\Phi,W}|\Phi\subset\Psi\}$.

For two reduced sub-root systems $\Phi_1,\Phi_2\subset\Psi$ with $\Phi_2\not\sim_{W}\Phi_1$, the above way is
useful for checking if $F_{\Phi_1,W}=F_{\Phi_2,W}$ or not. We express $F_{\Phi_1,W}-F_{\Phi_2,W}$ in terms
of $\{F_{\Phi_{i},W_{\Psi_{i}}}\}$ by tensor operation, addition and subtraction. Then the linear relations
among $\{F_{\Phi_{i,j},W_{\Psi_{i}}}:\Phi_{i,j}\subset\Psi_{i}\}$ tells us if such an expression is $0$ or
not. For example, in the case that $\Psi=m\Psi_0$ and $W=(W_{\Psi_0})^{m}\rtimes A_{m}$, if
$\Phi_1,\dots,\Phi_{m}\subset\Psi_0$ are reduced sub-root systems such that
$F_{\Phi_1,W_{\Psi_0}},\dots,F_{\Phi_m,W_{\Psi_0}}$ are linearly dependent, then for
\[\Phi=\Phi_1\sqcup\Phi_2\sqcup\cdots\sqcup\Phi_{m}\] and
\[\Phi=\Phi_{2}\sqcup\Phi_{1}\sqcup\Phi_3\sqcup\Phi_4\sqcup\cdots\sqcup\Phi_{m},\] we have
$F_{\Phi_1,W}=F_{\Phi_2,W}$.
\end{remark}

\subsection{Algebraic relations among $\{a_{n},b_{n},b'_{n},c_{n},d_{n}|\ n\geq 1\}$}\label{SS:abcd}


The following proposition can be proven in a similar way as the proof of Proposition \ref{P:A=BB and A=CD}.
The method is: after showing equalities (\ref{Eq:an})-(\ref{Eq:dn}), we are led to prove some identities for
determinants.

\begin{proposition}\label{P:A-BB'CD}
For any $n\geq 0$, we have $a_{2n}=b_{n}b'_{n}$, $a_{2n+1}=c_{n}d_{n+1}$, $2a_{2n}=c_{n}d_{n}+c_{n-1}d_{n+1}$
and $2a_{2n+1}=b_{n}b'_{n+1}+b'_{n}b_{n+1}$.
\end{proposition}

\begin{proof}
The equalties \begin{equation}\label{Eq:a1}  a_{2n}=b_{n}b'_{n} \end{equation} and
\begin{equation}\label{Eq:a2} a_{2n+1}=c_{n}d_{n+1} \end{equation} are proven in Proposition
\ref{P:A=BB and A=CD}.  We prove \begin{equation}\label{Eq:a3} 2a_{2n}=c_{n}d_{n}+c_{n-1}d_{n+1}\end{equation}
and \begin{equation}\label{Eq:a4} 2a_{2n+1}=b_{n}b'_{n+1}+b'_{n}b_{n+1} \end{equation} here.
In the proof of Proposition \ref{P:A=BB and A=CD}, we have introduced the matrices
$A_{n}$, $B_{n}$, $C_{n}$, $D_{n}$, $D'_{n}$ and expressed their determinants in terms of the polynomials
$a_{n}$, $b_{n}$, $c_{n}$, $d_{n}$.

Let \[A'_{2n}=\left(\begin{array}{ccccc}x_0&x_1&\ldots&x_{2n-2}&x_{2n-1}\\x_1&x_0&\ldots&x_{2n-3}&x_{2n-2}\\
\vdots&\vdots&&\vdots&\vdots\\x_{2n-2}&x_{2n-3}&\ldots&x_0&x_1\\x_{2n-1}+x_{1}&x_{2n-2}+x_{2}&\ldots
&x_1+x_{2n-1}&x_0+x_{2n}\\\end{array}\right)\] and
\[A''_{2n}=\left(\begin{array}{ccccc}x_0&x_1&\ldots&x_{2n-2}&x_{2n-1}\\x_1&x_0&\ldots&x_{2n-3}&x_{2n-2}\\
\vdots&\vdots&&\vdots&\vdots\\x_{2n-2}&x_{2n-3}&\ldots&x_0&x_1\\x_{2n-1}-x_{1}&x_{2n-2}-x_{2}&\ldots
&x_1-x_{2n-1}&x_0-x_{2n}\\\end{array}\right).\]
Then \[\det A'_{2n}+\det A''_{2n}=2\det A_{2n}=2a_{2n}.\]

Define \[J'_{2n}=\left(\begin{array}{cccc}\frac{1}{\sqrt{2}}I_{n-1}&0_{(n-1)\times 1}&-\frac{1}{\sqrt{2}}L_{n-1}
&0_{(n-1)\times 1}\\0_{1\times(n-1)}&1&0_{1\times(n-1)}&0\\\frac{1}{\sqrt{2}}L_{n-1}&0_{(n-1)\times 1}&\frac{1}
{\sqrt{2}}I_{n-1}&0_{(n-1)\times 1}\\0_{1\times(n-1)}&0&0_{1\times(n-1)}&1\end{array}\right)\] and
\[J''_{2n}=\left(\begin{array}{cccc}\frac{1}{\sqrt{2}}I_{n-1}&0_{(n-1)\times 1}&\frac{1}{\sqrt{2}}L_{n-1}
&0_{(n-1)\times 1}\\0_{1\times(n-1)}&1&0_{1\times(n-1)}&0\\-\frac{1}{\sqrt{2}}L_{n-1}&0_{(n-1)\times 1}&\frac{1}
{\sqrt{2}}I_{n-1}&0_{(n-1)\times 1}\\0_{1\times(n-1)}&0&0_{1\times(n-1)}&1\end{array}\right).\]
By matrix calculation, we get
\[J'_{2n}A'_{2n}(J'_{2n})^{-1}=\left(\begin{array}{cc}X_{1}&\ast_1\\0_{(n+1)\times(n-1)}&X_2\end{array}
\right)\] with $\det X_1=c_{n-1}$ and $\det X_2=d_{n+1}$. Similarly we get
\[J''_{2n}A''_{2n}(J''_{2n})^{-1}=\left(\begin{array}{cc}Y_{2}&\ast_2\\0_{n\times n}&Y_1\end{array}
\right)\] with $\det Y_1=c_{n}$ and $\det Y_2=d_{n}$. Taking determinants, we get
\[2a_{2n}=c_{n}d_{n}+c_{n-1}d_{n+1}.\]

Let \[A'_{2n+1}=\left(\begin{array}{ccccc}x_0&x_1&\ldots&x_{2n-2}&x_{2n}\\x_1&x_0&\ldots&x_{2n-2}&x_{2n-1}\\
\vdots&\vdots&&\vdots&\vdots\\x_{2n-1}&x_{2n-2}&\ldots&x_0&x_1\\x_{2n}+x_{1}&x_{2n-1}+x_{2}&\ldots
&x_1+x_{2n}&x_0+x_{2n+1}\\\end{array}\right)\] and
\[A''_{2n+1}=\left(\begin{array}{ccccc}x_0&x_1&\ldots&x_{2n-1}&x_{2n}\\x_1&x_0&\ldots&x_{2n-2}&x_{2n-1}\\
\vdots&\vdots&&\vdots&\vdots\\x_{2n-1}&x_{2n-2}&\ldots&x_0&x_1\\x_{2n}-x_{1}&x_{2n-1}-x_{2}&\ldots
&x_1-x_{2n}&x_0-x_{2n+1}\\\end{array}\right).\]
Then we have \[\det A'_{2n+1}+\det A''_{2n+1}=2\det A_{2n+1}=2a_{2n+1}.\]

Let \[J'_{2n+1}=\left(\begin{array}{cccc}\frac{1}{\sqrt{2}}I_{n-1}&-\frac{1}{\sqrt{2}}L_{n-1}
&0_{(n-1)\times 1}\\\frac{1}{\sqrt{2}}L_{n-1}&\frac{1}{\sqrt{2}}I_{n-1}&0_{(n-1)\times 1}\\
0_{1\times(n-1)}&0_{1\times(n-1)}&1\end{array}\right)\] and
\[J''_{2n+1}=\left(\begin{array}{cccc}\frac{1}{\sqrt{2}}I_{n-1}&\frac{1}{\sqrt{2}}L_{n-1}
&0_{(n-1)\times 1}\\-\frac{1}{\sqrt{2}}L_{n-1}&\frac{1}{\sqrt{2}}I_{n-1}&0_{(n-1)\times 1}\\
0_{1\times(n-1)}&0_{1\times(n-1)}&1\end{array}\right).\]
By matrix calculation, we get
\[J'_{2n+1}A'_{2n+1}(J'_{2n+1})^{-1}=\left(\begin{array}{cc}X_{1}&\ast_1\\0_{(n+1)\times(n-1)}&X_2\end{array}
\right)\] with $\det X_1=b_{n}$ and $\det X_2=b'_{n+1}$. Similarly we get
\[J''_{2n+1}A''_{2n+1}(J''_{2n+1})^{-1}=\left(\begin{array}{cc}Y_{2}&\ast_2\\0_{n\times n}&Y_1\end{array}
\right)\] with $\det Y_1=b_{n+1}$ and $\det Y_2=b'_{n}$. Taking determinants, we get
\[2a_{2n+1}=b_{n}b'_{n+1}+b_{n+1}b'_{n}.\] \end{proof}

\begin{remark}\label{R:A-BB'CD}
Here, we remark that, the equations (\ref{Eq:an})-(\ref{Eq:dn}) are discovered in \cite{An-Yu-Yu}
and the equalities \[\det A_{2n}=\det B_{n}\det B'_{n},\] \[\det A_{2n+1}=\det C_{n}\det D'_{n+1},\]
\[2\det A_{2n}=\det C_{n}\det D'_{n}+\det C_{n-1}\det D'_{n+1}\] and
\[2\det A_{2n+1}=\det B_{n}\det B'_{n+1}+\det B_{n+1}\det B'_{n}\]
are proven in the book \cite{Vein-Dale}, Page 88.
\end{remark}



\begin{proposition}\label{P:algebraic relations}
The equalities $a_{2n}=b_{n}b'_{n}$, $2a_{2n}=c_{n}d_{n}+c_{n-1}d_{n+1}$, $a_{2n+1}=c_{n}d_{n+1}$,
$2a_{2n+1}=b_{n}b'_{n+1}+b'_{n}b_{n+1}$ generate all algebraic relations among
\[\{a_{n},b_{n},b'_{n},c_{n},d_{n}|n\geq 1\}.\]
\end{proposition}

\begin{proof}
By Proposition \ref{basic properties}, we know that $\{b_{n},c_{n}|n\geq 1\}$ are algebraically
independent. It is clear that, from these equalities we can express $a_{n},b'_{n},d_{n}$ in
terms of rational functions of $\{b_{m},c_{m}|1\leq m\leq n\}$, so these equalities
generate all algebraic relations among $\{a_{n},b_{n},b'_{n},c_{n},d_{n}|n\geq 1\}$.
\end{proof}

\begin{remark}
More important than Proposition \ref{P:algebraic relations} itself is: we can use the four
identities to derive many other algebraic relations. Moreover, by expressing $a_{n},b'_{n},d_{n}$
in terms of rational functions of $\{b_{m},c_{m}|1\leq m\leq n\}$, we are able to check
whether any given polynomial function of $\{a_{n},b_{n},b'_{n},c_{n},d_{n}|n\geq 1\}$ is
identical to 0 or not.
\end{remark}

\textbf{Type $\BC$.}
From the equations $b_{n+1}b'_{n}+b_{n}b_{n+1}=2c_{n}d_{n+1}$,
$2b_{n}b'_{n}=c_{n}d_{n}+c_{n-1}d_{n+1}$, $2b_{n+1}b'_{n+1}=c_{n+1}d_{n+1}+c_{n}d_{n+2}$,
we get
\begin{equation}\label{Eq:BC} b_{n+1}^{2}(c_{n}d_{n}+c_{n-1}d_{n+1})+b_{n}^{2}(c_{n+1}d_{n+1}+c_{n}d_{n+2})-
4b_{n+1}b_{n}c_{n}d_{n+1}=0.\end{equation} When $n=0$, this equation is
$2b_1^{2}+(c_1d_1+d_2)-4b_1d_1=0$. Since $\{d_{n},c_{n}|n\geq 1\}$ are algebraically
independent, so these equations generate all algebraic relations among
$\{b_{n},c_{n},d_{n}|\ n\geq 1\}$. Together with the identities
$a_{2n-1}-c_{n-1}d_{n}=0$ and $2a_{2n}-c_{n}d_{n}-c_{n-1}d_{n+1}=0$, they generate
all algebraic relations among $\{a_{n},b_{n},c_{n},d_{n}|\ n\geq 1\}$.

\textbf{Type $\B$.}
From $2a_{2n+1}=b_{n+1}b'_{n}+b_{n}b_{n+1}$, $a_{2n+2}=b_{n+1}b'_{n+1}$,
$a_{2n}=b_{n}b'_{n}$, we get \begin{equation}\label{Eq:B1}a_{2n+2}b_{n}^{2}+a_{2n}b_{n+1}^{2}=
2a_{2n+1}b_{n}b_{n+1}.\end{equation} From $2a_{2n}=c_{n}d_{n}+c_{n-1}d_{n+1}$, $a_{2n+1}=c_{n}d_{n+1}$,
$a_{2n-1}=c_{n-1}d_{n}$, we get \begin{equation}\label{Eq:B2} a_{2n+1}d_{n}^{2}+a_{2n-1}d_{n+1}^{2}=
2a_{2n}d_{n}d_{n+1}.\end{equation} The equalities
\[\{a_{2n+2}b_{n}^{2}+a_{2n}b_{n+1}^{2}-2a_{2n+1}b_{n}b_{n+1}=0|\ n\geq 0\}\] and
\[\{a_{2n+1}d_{n}^{2}+a_{2n-1}d_{n+1}^{2}-2a_{2n}d_{n}d_{n+1}=0|\ n\geq 1\}\] generate all
algebraic relations among $\{a_{n},b_{n},d_{n}|\ n\geq 1\}$. Here we regard
$a_{-1}=d_1^{-1}=x_0^{-1}$, $a_0=1$.

\textbf{Type $\C$.} We have \begin{equation}\label{Eq:C1} a_{2n+1}=c_{n}d_{n+1}\end{equation} and
\begin{equation}\label{Eq:C2} 2a_{2n+2}=c_{n+1}d_{n+1}+c_{n}d_{n+2}\end{equation}
As $\{d_{n},c_{n}|n\geq 1\}$ are algebraically independent, the equalities
\[\{a_{2n+1}-c_{n}d_{n+1}=0|\ n\geq 1\}\] and
\[\{2a_{2n+2}-c_{n+1}d_{n+1}-c_{n}d_{n+2}=0|\ n\geq 0\}\]
generate all algebraic relations among $\{a_{n},c_{n},d_{n}|\ n\geq 1\}$.

\textbf{Type $\D$.} The equations \[\{a_{2n+1}d_{n}^{2}+a_{2n-1}d_{n+1}^{2}-2a_{2n}d_{n}d_{n+1}=0|\ n\geq 1\}\]
generate all algebraic relations among $\{a_{n},d_{n}|\ n\geq 1\}$.

\begin{remark}\label{R:ideal}
Let \[R=\bbQ[\{y_{1,n},y_{2,n},y_{3,n},y_{4,n},y_{5,n}|\ n\geq 1\}]\] be the free polynomial algebra over the
rational field with infinitely many indeterminates \[\{y_{1,n},y_{2,n},y_{3,n},y_{4,n},y_{5,n}|\ n\geq 1\}.\]
Let \[\Eva: R\longrightarrow\bbQ[\{x_{n}|\ n\geq 1\}]\] be the homomorphism defined by
\[\Eva(x_{1,n})=a_{n},\  \Eva(x_{2,n})=b_{n},\  \Eva(x_{3,n})=b'_{n}\] and
\[\Eva(x_{4,n})=c_{n},\  \Eva(x_{5,n})=d_{n}.\] Let $I$ be the kernel of $I$.
Let  \begin{equation}\label{Eq:I1}  f_{1,n}=y_{1,2n}-y_{2,n}y_{3,n},\end{equation}
\begin{equation}\label{Eq:I2} f_{2,n}=y_{1,2n+1}-y_{4,n}y_{5,n+1},\end{equation}
\begin{equation}\label{Eq:I3} f_{3,n}=2y_{1,2n}-y_{4,n}y_{5,n}-y_{4,n-1}y_{5,n+1}\end{equation} and
\begin{equation}\label{Eq:I4} f_{4,n}=2y_{1,2n+1}-y_{2,n}y_{3,n+1}-y_{3,n}y_{2,n+1}.\end{equation}
Then $f_{1,n},f_{2,n},f_{3,n},f_{4,n}\in I$ for any $n\geq 1$. Moreover, let $R_1$ be the lozalization of
$R$ with respect to the multiplicative system generated by $\{y_{2,n},y_{4,n}|\ n\geq 1\}$. Then we
have a homomorphism $\Eva_1: R_1\longrightarrow\bbQ(\{x_{n}|\ n\geq 1\}])$. One can show that,
the elements $\{f_{1,n},f_{2,n},f_{3,n},f_{4,n}|\ n\geq 1\}$ genetate the kernel of $\Eva_1$. Similarly,
they generate the corresponding kernels if we consider the localization with respect to the multiplicative
system generated by $\{y_{3,n},y_{4,n}|\ n\geq 1\}$ (or $\{y_{2,n},y_{4,n}|\ n\geq 1\}$, etc) and the similar
homomorphism. Does $\{f_{1,n},f_{2,n},f_{3,n},f_{4,n}|\ n\geq 1\}$ generate $I$? It seems to the author that
the answer to this is no. In that case, can one find a system of gneerators of $I$? One could consider a subset
of $\{a_{n},b_{n},b'_{n},c_{n},d_{n}|\ n\geq 1\}$ (e.g, $\{a_{n},b_{n},c_{n},d_{n}|\ n\geq 1\}$), the similar
homomorphism as the above $\Eva$ and ask about the generators of the kernel.
\end{remark}

\subsection{Classical irreducible root systems}\label{SS:classical root system}

\begin{proposition}\label{P:gamma-character}
Given a root system $\Psi$ and an automorphism $\gamma\in\Aut(\Psi)$, if $\gamma\Phi=\Phi$ for
a reduced sub-root system $\Phi$ of $\Psi$, then $\gamma F_{\Phi,W_{\Psi}}=F_{\Phi,W_{\Psi}}$.
\end{proposition}
\begin{proof}
Replacing $\gamma$ by some $w\gamma$ ($w\in W_{\Phi}$) if necessary, we may assume that
$\gamma\Phi^{+}=\Phi^{+}$. Then $\gamma\delta_{\Phi}=\delta_{\Phi}$ and $\gamma$ maps simple
roots of $\Phi$ to simple roots. By the latter, we get that
$\gamma W_{\Phi}\gamma^{-1}=W_{\Phi}$.

Hence \begin{eqnarray*}\gamma F_{\Phi,W_{\Psi}}&=&\gamma(\sum_{w\in W_{\Phi}}
\chi^{\ast}_{\delta_{\Phi}-w\delta_{\Phi},W_{\Psi}})\\&=&\sum_{w\in W_{\Phi}}
\chi^{\ast}_{\gamma\delta_{\Phi}-(\gamma w\gamma^{-1})\gamma\delta_{\Phi},W_{\Psi}}
\\&=&\sum_{w\in W_{\Phi}}\chi^{\ast}_{\delta_{\Phi}-(\gamma w\gamma^{-1})\delta_{\Phi},W_{\Psi}}
\\&=&\sum_{w\in W_{\Phi}}\chi^{\ast}_{\delta_{\Phi}-w\delta_{\Phi},W_{\Psi}}\\&=
& F_{\Phi,W_{\Psi}}.\end{eqnarray*}
\end{proof}

Given a classical irreducible root system $\Psi$ of rank $n$, in the case that $\Psi=\A_{n}$ ($n\geq 1$),
we have the following simple statement.

\begin{proposition}\label{P:linear-A}
For any sub-root system $\Phi\subset\A_{n}$, $F_{\Phi,W_{\A_{n}}}=F_{\Phi,\Aut(\A_{n})}$.

For any two sub-root systems $\Phi_1,\Phi_2\subset\A_{n}$, $2\delta'_{\Phi_1}=
2\delta'_{\Phi_2}$ if and only if $\Phi_1\sim\Phi_2$.

For any pairwise non-conjugate sub-root systems $\Phi_1,\dots,\Phi_{s}\subset\A_{n}$,
the characters $\{F_{\Phi_{i},W_{\A_{n}}}|1\leq i\leq s\}$ are linearly independent.
\end{proposition}

\begin{proof}
By Proposition \ref{P:gamma-character}, we have $F_{\Phi,W_{\A_{n}}}=F_{\Phi,\Aut(\A_{n})}$ since there
exists $\gamma\in\Aut(\A_{n})-W_{\A_{n}}$ such that $\gamma\Phi=\Phi$ (the index of $W_{\A_{n}}$ in
$\Aut(\A_{n})$ is $2$).

For a sub-root system \[\Phi\cong\bigsqcup_{1\leq i\leq s}\A_{n_{i}-1}\subset\A_{n}\] with $n_{i}\geq 1$
and $\sum_{1\leq i\leq s}n_{i}=n+1$, \[2\delta_{\Phi}\sim(\underbrace{n_1-1,n_1-3,...,1-n_1},....,
\underbrace{n_{s}-1,n_{s}-3,...,1-n_{s}}).\] From this one sees that the weights $2\delta_{\Phi}$ are
non-conjugate for any two non-conjugate sub-root systems. Therefore the remaining conclusions in the
proposition follow.
\end{proof}

In the case that $\Psi=\BC_{n}$, $\B_{n}$ or $\C_{n}$, we have $W_{\Psi}=W_{\BC_{n}}=W_{n}$. Therefore, the
question of getting linear relations among $\{F_{\Phi,W_{\Psi}}|\ \Phi\subset\Psi\}$ reduces to the question
of getting algebraic relations of homogeneous degree $n$ among the polynomials
$\{a_{m},b_{m},c_{m},d_{m}|\ 1\leq m\leq n\}$, $\{a_{m},b_{m},d_{m}|\ 1\leq m\leq n\}$,
$\{a_{m},c_{m},d_{m}|\ 1\leq m\leq n\}$, respectively. These algebraic relations are discussed in the last
subsection.

In the case of $\Psi=\BC_{n}$, we make an interesting observation. We have showed
the equations $b_{n+1}^{2}(c_{n}d_{n}+c_{n-1}d_{n+1})+b_{n}^{2}(c_{n+1}d_{n+1}+c_{n}d_{n+2})-
4b_{n+1}b_{n}c_{n}d_{n+1}=0$ for any $n\geq 0$. When $n=0$ and $n=1$, they are
$2b_1^{2}+(c_1d_1+d_2)-4b_1d_1=0$ and \[b_{2}^{2}(c_{1}d_{1}+d_{2})+b_{1}^{2}(c_{2}d_{2}+c_{1}d_{3})-
4b_{2}b_{1}c_{1}d_{2}=0.\] Eliminating $d_1$, we get
\begin{eqnarray*}0&=&(4b_1-c_1)(b_{2}^{2}(c_{1}d_{1}+d_{2})+b_{1}^{2}(c_{2}d_{2}+c_{1}d_{3})-
4b_{2}b_{1}c_{1}d_{2})\\&&+(b_2^{2}c_1)(2b_1^{2}+(c_1d_1+d_2)-4b_1d_1)
\\&=&b_1(2b_1b_2^{2}c_1+4b_1^2c_2d_2+4b_1^2c_1d_3+4b_2c_1^2d_2+
4b_2^{2}d_2\\&&-b_1c_1c_2d_2-b_1c_1^2d_3-16b_1b_2c_1d_2).\end{eqnarray*} Since
$b_1=x_0-x_1$ is irreducible, we get \[2b_1b_2^{2}c_1+4b_1^2c_2d_2+
4b_1^2c_1d_3+4b_2c_1^2d_2+4b_2^{2}d_2-b_1c_1c_2d_2-b_1c_1^2d_3-16b_1b_2c_1d_2=0.\]
This gives 8 rank-6 semisimple subgroups of $\SU(15)$ with linearly dependent dimension
data. This equation is also given in \cite{An-Yu-Yu}, Example 5.6 and implicit in
\cite{Larsen-Pink}, Page 393.

In the case of $\Psi=\D_{n}$ ($n\geq 4$), algebraic relations of homogeneous degree $n$ among
$\{a_{m},c_{m},d_{m}|1\leq m\leq 1\}$ correspond to the linear relations among
the characters \[\{F_{\Phi,W_{n}}|\ \Phi\subset\D_{n}\}.\]
Note that \[W_{D_{n}}=\Gamma_{n}\rtimes S_{n}\] is a subgroup of $W_{n}$ of index 2,
where \[\Gamma_{n}=\{(a_1,a_2,\dots,a_{n})\subset\{\pm{1}\}^{n}:a_1a_2\cdots a_{n}=1\}.\]

\begin{proposition}\label{P:Character-D}
Given $\Psi=\D_{n}$ ($n\geq 4$) and a sub-root system $\Phi\subset\Psi$, the following conditions
are equivalent to each other:
\begin{itemize}
\item[(1)]{$F_{\Phi,W_{\Psi}}\neq F_{\Phi,W_{n}}$.}
\item[(2)]{$2\delta'_{\Phi}$ is not $W_{n}$ invariant.}
\item[(3)]{\[\Phi\cong\bigsqcup_{1\leq i\leq s}\A_{n_{i}-1}\] with $2\leq n_1\leq n_2\cdots\leq n_{s}\leq n$,
each $n_{i}$ is even, and $n_1+n_2+\cdots+n_{s}=n$.}
\item[(4)]{$\Phi\not\sim_{W_{\D_{n}}}\gamma\Phi$ for some (equivalently, for any) $\gamma\in W_{n}-W_{\D_{n}}$.}
\end{itemize}
\end{proposition}

\begin{proof}
$(2)\Rightarrow(1)$ is clear. We prove $(1)\Rightarrow(3)$, $(3)\Rightarrow(2)$ and $(3)\Leftrightarrow(4)$
in the below.

Any sub-root system $\Phi$ of $\D_{n}$ (cf. \cite{Larsen-Pink} and \cite{Oshima}) is conjugate to one of
\[(\bigsqcup_{1\leq i\leq s}\D_{n_{i}})\bigsqcup(\bigsqcup_{1\leq j\leq t}\A_{m_{j}}),\]
where $n_{i}\geq 1$, $m_{j}\geq 2$, $\sum_{1\leq i\leq s}n_{i}+\sum_{1\leq j\leq t}m_{j}=n$.
Thus $(3)\Leftrightarrow(4)$.

If $(3)$ does not hold, then $s\geq 1$ or $s=0$ and some $m_{i}$ is odd. In the case that $s\geq 1$, we may assume
that $\D_{n_1}=\langle e_1-e_2,\dots,e_{n_{1}-1}-e_{n_1},e_{n_{1}-1}+e_{n_1}\rangle\subset\Phi$. Thus
$s_1\Phi=\Phi$. By Proposition \ref{P:gamma-character} we get $s_1 F_{\Phi,W_{\Psi}}=F_{\Phi,W_{\Psi}}$. Since
$s_1,W_{\D_{n}}$ generate $\Aut(\D_{n})$ ($n\geq 5$), we have \[F_{\Phi,W_{\Psi}}=F_{\Phi,W_{n}}.\] In the case
that $s=0$ and some $m_{j}$ is odd, we may and do assume that $m_1$ is odd and $A_{m_1}=\langle e_1-e_2,\dots,
e_{m_{1}-1}-e_{m_1}\rangle$. Then $s_1s_2\cdots s_{m_1}\Phi=\Phi$. By Proposition \ref{P:gamma-character},
\[s_1s_1\cdots s_{m_1} F_{\Phi,W_{\Psi}}=F_{\Phi,W_{\Psi}}.\] Since $s_1s_2\cdots s_{m_1},W_{\D_{n}}$
generate $\Aut(\D_{n})$ ($n\geq 5$), we get $F_{\Phi,W_{\Psi}}=F_{\Phi,W_{n}}$.
This proves $(1)\Rightarrow(3)$.

If $(3)$ holds, by calculation we get
\[2\delta_{\Phi}=(n_{1}-1,n_{1}-3,\dots,1-n_1,\dots,n_{s}-1,\dots,3-n_{s},1-n_{s}).\] Thus
$2\delta'_{\Phi}$ is not $W_{n}$ invariant. This proves $(3)\Rightarrow(2)$.
\end{proof}

By Proposition \ref{P:Character-D}, for any sub-root system $\Phi$ of $\D_n$, either
$W_{\D_{n}}\Phi=W_{n}\Phi$ and $F_{\Phi,W_{\D_{n}}}=F_{\Phi,W_{n}}$, or $W_{\D_{n}}\Phi\neq W_{n}\Phi$
and $F_{\Phi,W_{\D_{n}}}$, $F_{\gamma\Phi,W_{\D_{n}}}$ are linearly independent. Here $\gamma$ is any
element in $W_{n}-W_{\D_{n}}$. In this way we get all linear relations among
$\{F_{\Phi,W_{\D_{n}}}: \Phi\subset\D_{n}\}$ from the linear relations among
$\{F_{\Phi,W_{n}}: \Phi\subset\D_{n}\}$.

In the case of $\Psi=\D_4$, its automorphism group is larger than $W_{n}$. We discuss more on this case.
Firstly the $W_{\Psi}$-conjugacy classes of sub-root systems are \[\emptyset,\ A_1,\ A_2,\
D_2,\ 2A_1,\ (2A_1)',\ A_3\] and \[(A_3)',\ D_3,\ D_2+A_1,\ D_4,\ 2D_2.\] Here
\[D_2=\langle e_1-e_2,e_1+e_2\rangle,\] \[2A_1=\langle e_1-e_2,e_3-e_4\rangle,\]
\[(2A_1)'=\langle e_1-e_2,e_3+e_4\rangle,\] \[A_3=\langle e_1-e_2,e_2-e_3,e_3-e_4\rangle,\]
\[(A_3)'=\langle e_1-e_2,e_2-e_3,e_3+e_4\rangle\] and \[D_3=\langle e_2-e_3,e_3-e_4,e_3+e_4\rangle.\]
The sub-root systems $D_2$, $2A_1$, $(2A_1)'$ are conjugate to each other under $\Aut(\D_4)$, and the
sub-root systems $A_3$, $(A_3)'$, $D_3$ are conjugate to each other under $\Aut(\D_4)$.

Only the characters of sub-root systems $A_2$ and $4A_1$ have equal leading term, which is $2\omega_2$.
We have \begin{eqnarray*} &&F_{A_2,W_{\Psi}}=1-2\chi^{\ast}_{\omega_2}+
2\chi^{\ast}_{\omega_1+\omega_3+\omega_4}-\chi^{\ast}_{2\omega_2},\\&& F_{4A_1,W_{\Psi}}
=1-4\chi^{\ast}_{\omega_2}+2(\chi^{\ast}_{2\omega_1}+\chi^{\ast}_{2\omega_3}+
\chi^{\ast}_{2\omega_4})-4\chi^{\ast}_{\omega_1+\omega_3+\omega_4}+\chi^{\ast}_{2\omega_2}.
\end{eqnarray*} A little more calculation shows that \[F_{A_2,W_{\Psi}}+F_{4A_1,W_{\Psi}}-
2F_{D_2+A_1,W_{\Psi}}=0.\] This equality also follows from the equality $a_3+d_2^{2}=d_2(c_1+d_2)=2a_2d_2$
by noting that $F_{\Phi,W_{\D_4}}=F_{\Phi,W_4}$ if $\Phi=A_2$, $4A_1$ or $D_2+A_1$. This is the only
linear relation among $\{F_{\Phi,W_{\D_4}}|\ \Phi\subset\D_4\}$.


\begin{proposition}\label{P:linear-BCD}
Given a compact connected simple Lie group $G$ of type $\B_{n}$ ($n\geq 4$),
$\C_{n}$ ($n\geq 3$) or $\D_{n}$ ($n\geq 4$), there exist non-isomorphic closed connected subgroups
of $G$ with linearly dependent dimension data.
\end{proposition}

\begin{proof}
Taking $n=1$ in Equation (\ref{Eq:B2}) and using $a_1=d_1=x_0$, we get \[a_3d_1+d_2^2-2a_2d_2=0.\]
This gives us a linear relation of three non-isomorphic subgroups of $G$ in the case that $G$ is
of type $\B_{n}$ or $\D_{n}$ with $n\geq 4$. In the $\D_n$ case, by Proposition
\ref{P:Character-D} we know that the $W_{\D_{n}}$ trace and the $W_{n}$ trace of characters are equal for
the sub-root systems corresponding to the polynomials $a_3d_1$, $d_2^2$ and $a_2d_2$.

By Equations (\ref{Eq:C1}) and (\ref{Eq:C2}), we get \[2a_{2n+2}c_{n}c_{n+1}-c_{n+1}^{2}a_{2n+1}-
c_{n}^{2}a_{2n+3}=0.\] Taking $n=0$, we get $2a_2c_1-a_1c_1^{2}-a_3=0$, which gives us a linear relation
of three non-isomorphic subgroups of $G$ in the case that $G$ is of type $\C_{n}$ with $n\geq 3$.
\end{proof}

\subsection{A generating function}\label{SS:generating function}

Given an irreducible root system $\Psi$ with a positive system $\Psi^{+}$, a root $\alpha\in\Psi$
is called a short root if for any other root $\beta\in\Psi$, either $(\alpha,\beta)=0$
or $|\beta|\geq|\alpha|$. We normalize the inner product on $\Psi$ (or to say, on
$\Lambda_{\Psi}$) by letting all the short roots of $\Psi$ have length 1. For a reduced
sub-root system $\Phi$ of $\Psi$, recall that
\[\delta_{\Phi}=\frac{1}{2}\sum_{\alpha\in\Phi\cap\Psi^{+}}\alpha\] and
$\delta'_{\Phi}$ be the unique dominant weight in the Weyl group (of $\Psi$) orbit of
$\delta_{\Phi}$. Let $e(\Phi)=|2\delta'_{\Phi}|^{2}$ be the square of the length of
$2\delta'_{\Phi}$.


\begin{definition}\label{D:f}
For a reduced sub-root system $\Phi$ of $\Psi$, let
\[f_{\Phi,\Psi}(t)=\sum_{w\in W_{\Phi}}\epsilon(w)t^{|\delta_{\Phi}-w\delta_{\Phi}|^2}.\]

In particular let $f_{\Psi}(t)=f_{\Psi,\Psi}(t)$.
\end{definition}

As $|\delta_{\Phi}-w\delta_{\Phi}|^2=(2\delta_{\Phi},\delta_{\Phi}-w\delta_{\Phi})$,
an equivalent definition for $f_{\Phi}$ is \[f_{\Phi}(t)=
\sum_{w\in W_{\Phi}}\epsilon(w)t^{(2\delta_{\Phi},\delta_{\Phi}-w\delta_{\Phi})}.\]

If $\Phi=\bigsqcup_{1\leq i\leq s}\Phi_{s}$ is an orthogonal decomposition of $\Phi$ into irreducible
sub-root systems and $\sqrt{r_{i}}$ is the shortest length of roots in $\Phi_{i}\subset\Psi$, then
\[f_{\Phi,\Psi}=\prod_{1\leq i\leq s}f_{\Phi_{i}}(t^{r_{i}}).\] Thus the calculation of the polynomials
$f_{\Phi,\Psi}(t)$ reduces to the calculation of the polynomials $f_{\Psi}(t)$ for irreducible root
systems.

\begin{definition}\label{D:chi-lambda and E'}
Let $\Lambda_{\Psi}\subset\bbQ\Psi$ be the set of integral weights of $\Psi$. For any weight
$\lambda\in\Lambda_{\Psi}$, define
\[\chi^{\ast}_{\lambda}=\frac{1}{|W_{\Psi}|}\sum_{\gamma\in W_{\Psi}}\gamma\lambda.\]
We define a linear map \[E'=\bbQ[\Lambda_{\Psi}]^{W_{\Psi}}\longrightarrow\bbQ[t]\] by
$E'(\chi^{\ast}_{\lambda})=t^{|\lambda|^2}$.
\end{definition}

\begin{proposition}\label{P:character-1-polynomial}
Given a reduced sub-root system $\Phi$ of $\Psi$, we have $E'(F_{\Phi,W_{\Psi}})=f_{\Phi,\Psi}$.
\end{proposition}

\begin{proof}
This follows from the formulas
\[F_{\Phi,W_{\Psi}}=\sum_{w\in W_{\Phi}}\epsilon(w)\chi^{\ast}_{\delta_{\Phi}-w\delta_{\Phi}}\]
and \[f_{\Phi,\Psi}(t)=\sum_{w\in W_{\Phi}}\epsilon(w)t^{|\delta_{\Phi}-w\delta_{\Phi}|^2}.\]
\end{proof}

Let $\psi:\bbQ[x_0,x_1,\dots,x_{n},\dots]\longrightarrow\bbQ[t]$ be an algebra homomorphism
defined by \[\psi(x_{n})=t^{n^2},\forall n\geq 0.\]

\begin{proposition}\label{P:n-polynomial-1-polynomial}
Given $\Psi=\BC_{n}$ and a reduced sub-root system $\Phi$ of $\BC_{n}$, we have
\[f_{\Phi,\Psi}=\psi(E(j_{n}(F_{\Phi,W_{n}})).\]
\end{proposition}

\begin{proof}This follows from the definitions of $E$, $\psi$ and $f_{\Phi,\Psi}$. \end{proof}

Recall that $E(j_{n}(F_{\Phi,W_{n}}))\in\bbQ[x_0,x_1,\dots,x_{n},\dots]$ is the multi-variable
polynomial associated to reduced sub-root systems of $\BC_{n}$ in Section \ref{S:dimension-equal}. Proposition
\ref{P:n-polynomial-1-polynomial} connects two polynomials by a simple relation.

\smallskip

For irreducible reduced root systems of small rank, calculation shows that
$f_{\A_1}=1-t$, \[f_{\A_2}=(1-t)^{2}(1-t^2),\] \[f_{\B_2}=(1-t)(1-t^{2})(1-t^{3})(1-t^4),\]
\[f_{\G_2}=(1-t)(1-t^{3})(1-t^4)(1-t^{5})(1-t^6)(1-t^{9}),\] and
$f_{\A_3}=(1-t)^3(1-t^{2})^{2}(1-t^3)$.
By calculation, we also have
\[f_{\A_4}=1-4t+3x^2+6x^3-7x^4-2x^5-4x^{6}+\textrm{higher terms},\]
\[f_{\D_4}=1-4t+3x^2+5x^3-3x^4-6x^5-6x^{6}+\textrm{higher terms},\]
\[f_{\A_5}=1-5t+6x^2+7x^3-16x^4+0x^5+2x^{6}+\textrm{higher terms}.\]

\begin{proposition}\label{P:f-basic}
$f_{\Phi,\Psi}$ has constant term 1 and  leading term $(-1)^{|\Phi^{+}|}t^{|2\delta_{\Phi}|^{2}}$,
it satisfies \[f_{\Phi,\Psi}(t)=(-1)^{|\Phi^{+}|}t^{|2\delta_{\Phi}|^{2}}f_{\Phi,\Psi}(t^{-1}).\]
\end{proposition}

\begin{proof}
Recall that, $W_{\Phi}$ has a longest element $\omega_0$ which maps $\Phi^{+}$ to
$-\Phi^{+}$, so $\epsilon_{w_0}=(-1)^{|\Phi^{+}|}$ and $w_0(\delta_{\Phi})=-\delta_{\Phi}$.
Moreover, For any $w\in W$, we have
\[(2\delta_{\Phi},\delta_{\Phi}-w\delta_{\Phi})+
(2\delta_{\Phi},\delta_{\Phi}-w_0w^{-1}\delta_{\Phi})=|2\delta_{\Phi}|^2.\]
Therefore the proposition follows.
\end{proof}

\begin{proposition}\label{P:Weyl-generating function}
Given an irreducible reduced root system $\Psi$, we have \[f_{\Psi}=\prod_{\alpha\in\Psi^{+}}
(1-t^{(2\delta_{\Psi},\alpha)}).\]
\end{proposition}

\begin{proof}
Let $E'':\bbQ[\Lambda_{\Psi}]\longrightarrow\bbQ[t]$ be a linear map defined by
\[E''([\lambda])=t^{\langle2\delta_{\Psi},\lambda\rangle},\ \forall \lambda\in\Lambda_{\Psi}.\]
Then $E''$ is an algebra homomorphism and we have
\[f_{\Psi}(t)=E''(\sum_{w\in W_{\Psi}}\sign(w)[\delta_{\Psi}-w\delta_{\Psi}]).\]

By the Weyl denominator formula, we have \[\sum_{w\in W_{\Psi}}\sign(w)[w\delta_{\Psi}]=
\prod_{\alpha\in\Psi^{+}}([\frac{\alpha}{2}]-[-\frac{\alpha}{2}]).\]
Then we have \[\sum_{w\in W_{\Psi}}\sign(w)[\delta_{\Psi}-w\delta_{\Psi}]=
\prod_{\alpha\in\Psi^{+}}(1-[\alpha]).\] Taking the map $E''$ on both sides, we get
\[f_{\Psi}=\prod_{\alpha\in\Psi^{+}}(1-t^{(2\delta_{\Psi},\alpha)}).\]
\end{proof}

\subsection{Exceptional irreducible root systems}\label{SS:linear-exceptional}

In this subsection, for any exceptional irreducible root system $\Psi$, we consider the
linear relations among the characters $\{F_{\Phi,W_{\Psi}}|\ \Phi\subset\Psi\}$.

We start with some observations.

Firstly, if \[\sum_{1\leq i\leq s} c_{i}F_{\Phi_{i},W_{\Psi}}=0\] for some reduced sub-root systems
$\{\Phi_{i}\subset\Psi:1\leq i\leq s\}$ and some non-zero coefficients $\{c_{i}\in\bbR: 1\leq i\leq s\}$,
then for any $i$ with \[|\delta'_{\Phi_{i}}|=\max\{|\delta'_{\Phi_{j}}|:1\leq j\leq i\},\] there exits
$j\neq i$ such that $\delta'_{\Phi_{j}}=\delta'_{\Phi_{i}}$. Since otherwise,
$\chi_{2\delta_{\Phi_{i}},W_{\Psi}}^{\ast}$ has a non-zero coefficient in the character
$\sum_{1\leq i\leq s} c_{i}F_{\Phi_{i},W_{\Psi}}$.

Secondly, if $\Phi_1,\dots,\Phi_{s}$ are all contained in another reduced sub-root system $\Psi'$ of $\Psi$,
and $\sum_{1\leq i\leq s} c_{i}F_{\Phi_{i},W_{\Psi'}}=0$ for some constants $c_1,\dots,c_{s}\in\bbR$,
then $\sum_{1\leq i\leq s} c_{i}F_{\Phi_{i},W_{\Psi}}=0$. This is due to \[F_{\Phi_{i},W_{\Psi}}=
\frac{1}{|W_{\Psi}|}\sum_{\gamma\in W_{\Psi}}\gamma F_{\Phi_{i},W_{\Psi'}}\] for each
$i$, $1\leq i\leq s$.

Thirdly, if $\Phi'_1,\dots,\Phi'_{s}$ are all contained in another reduced sub-root system $\Psi'$ of $\Psi$,
and \[\sum_{1\leq i\leq s} c_{i}F_{\Phi'_{i},W_{\Psi'}}=0\] for some constants $c_1,\dots,c_{s}$. Let
\[\Phi'\subset\Psi'^{\perp}=\{\alpha\in\Psi|(\alpha,\beta)=0,\forall\beta\in\Psi'\}\] and
$\Phi_{i}=\Phi'_{i}\sqcup\Phi'$. Then \[\sum_{1\leq i\leq s} c_{i}F_{\Phi_{i},W_{\Psi}}=0.\]
This follows from the second observation.

Fourthly, if $\sum_{1\leq i\leq s} c_{i}F_{\Phi_{i},W_{\Psi}}=0$, then \[\sum_{1\leq i\leq s}
c_{i}f_{\Phi_{i},\Psi}=0.\] This follows by applying the map $E'$.

\smallskip

\textbf{Type $\E_6$.} In the proof of Theorem \ref{character-exceptional}, we have observed that only
the two weights $2\omega_2, 2\omega_4$ are of the form $2\delta'_{\Phi}$ for at least two non-conjugate
reduced sub-root systems of $\E_6$. It happens that $2\delta'_{\Phi}=2\omega_2$ for $\Phi=A_2,4A_1$, and
$2\delta'_{\Phi}=2\omega_4$ for $\Phi=3A_2,A_3+2A_1$.

Since $\D_4\subset\E_6$, from the conclusion in $\D_4$ case and the second observation, we get
\begin{equation}F_{A_2,W_{\E_6}}+F_{4A_1,W_{\E_6}}-2F_{3A_1,W_{\E_6}}=0.\end{equation} Moreover, we have
\begin{equation}F_{3A_2,W_{\E_6}}+F_{A_3+2A_1,W_{\E_6}}+F_{A_3+A_1,W_{\E_6}}-3F_{2A_2+A_1,W_{\E_6}}=0.
\label{Eq:E6-identity} \end{equation} The proof of this equality is given below. These two relations
generate all linear relations among $\{F_{\Phi,W_{\E_6}}|\Phi\subset\E_6\}$.

Let $\theta$ be a linear map on weights defined by $\theta(\omega_1)=\omega_6$, $\theta(\omega_6)=
\omega_1$, $\theta(\omega_3)=\omega_5$, $\theta(\omega_5)=\omega_3$, $\theta(\omega_2)=\omega_2$,
$\theta(\omega_4)=\omega_4$. Then $\theta$ acts as an isometry and it maps dominant integral weights
to dominant integral weights. We have $\Aut(\E_6)=W_{\E_6}\rtimes\langle\theta\rangle$ as groups acting
on the weights.

\begin{lemma}\label{L:E6-small length}
For a positive integer $k\leq 11$, if $k\neq 4,5,7,8,9,10$, then there exists a unique
$\Aut(\E_6)$-orbit of weights $\lambda$ in the root lattice such that $|\lambda|^{2}=k$.

When $k=4,5,7,8,9$, there exist two $\Aut(\E_6)$-orbits of weights in the root lattice with
$|\lambda|^{2}=k$. The representatives are \[k=4:\{\omega_1+\omega_3,2\omega_2\};\]
\[k=5:\{\omega_3+\omega_5,\omega_1+\omega_2+\omega_6\};\]
\[k=7:\{2\omega_1+\omega_5,\omega_2+\omega_4\};\]
\[k=8:\{2\omega_1+2\omega_6,\omega_1+\omega_2+\omega_3\};\]
\[k=9:\{3\omega_2,\omega_1+\omega_4+\omega_6\}.\]

When $k=10$, there exist four $\Aut(\E_6)$-orbits of weights in the root
lattice with $|\lambda|^{2}=10$. The representatives are \[\omega_1+2\omega_5,
3\omega_1+\omega_2,\omega_2+\omega_3+\omega_5,\omega_1+2\omega_2+\omega_6.\]
\end{lemma}

\begin{proof}
The inverse to the Cartan matrix of $\E_6$ is
\[\frac{1}{3}\times\left(\begin{array}{cccccc}4&3&5&6&4&2\\3&6&6&9&6&3
\\5&6&10&12&8&4\\6&9&12&18&12&6\\4&6&8&12&10&5\\2&3&4&6&5&4\\\end{array}\right).\]
Given a dominant integral weight $\lambda=\sum_{1\leq i\leq 6}a_{i}\omega_{i}$, $a_{i}\in\bbZ_{\geq 0}$,
we have \begin{eqnarray*}|\lambda|^2&=&\frac{1}{3}(2a_{1}^{2}+2a_6^2+5a_3^2+5a_5^2)+\\&&
\frac{1}{3}(2a_1a_6+8a_3a_5+5a_1a_3+5a_5a_6+4a_1a_5+4a_3a_6)+\\&&(a_2^2+3a_4^2+3a_2a_4)+
(2a_3+2a_5+a_1+a_6)(a_2+2a_4).\end{eqnarray*} We also have: $\lambda$ is in the root lattice
if and only if $3|a_1+a_5-a_3-a_6$.

If $|\lambda|^2\leq 11$, let $\lambda_1=a_1\omega_1+a_3\omega_3+a_5\omega_5+a_6\omega_6$ and
$\lambda_2=a_2\omega_2+a_4\omega_4$. Then, $|\lambda_1|^2\leq 11$ and $|\lambda_2|^2\leq 11$.

Consider the weight $\lambda_1$. Let $k=a_1+a_5$ and $l=a_3+a_6$, $\lambda$ is in the root lattice
implies that $3|k-l$. From $|\lambda_1|^2\leq 11$, we get $k,l\leq 3$. When $|k-l|=3$, we have
\[\lambda_1=3\omega_1 (6),\ 3\omega_6 (6),\ 2\omega_1+\omega_5 (7),\] \[\omega_3+2\omega_6 (7),\
\omega_1+2\omega_5 (10),\ 2\omega_3+\omega_6 (10).\] Here, the numbers in the brackets mean the
squares of modulus of the weights. Similar notation will be used in the remaining part of this proof
and the proof for Lemma \ref{L:F4-short}. When $|k-l|=0$, we have $k=l\leq 2$. Moreover, we have
\[\lambda_1=2\omega_1+\omega_3+\omega_6 (11),\ \omega_1+\omega_5+2\omega_6 (11),\
2\omega_1+2\omega_6 (8),\] \[\omega_3+\omega_5 (5),\ \omega_1+\omega_3 (4),\omega_5+\omega_6 (4),\
\omega_1+\omega_6 (2).\]

Consider the weight $\lambda_2$. From $|\lambda_2|^2\leq 11$, we get \[\lambda_2=\omega_4+\omega_2 (7),\
\omega_4 (3),3\omega_2 (9),\ 2\omega_2 (4),\ \omega_2 (1).\]

In the case that $\lambda_1\neq 0$ and $\lambda_2\neq 0$, we have
\[\lambda=\omega_1+2\omega_2+\omega_6 (10),\ 3\omega_1+\omega_2 (10), \omega_2+3\omega_6 (10),\]
\[\omega_2+\omega_3+\omega_5 (10),\ \omega_1+\omega_4+\omega_6 (9),\ \omega_1+\omega_2+\omega_3 (8),\]
\[\omega_2+\omega_5+\omega_6 (8),\ \omega_1+\omega_2+\omega_6 (5).\]

We finish the proof of the lemma.
\end{proof}

\begin{proof}[Proof of equality (\ref{Eq:E6-identity}).]
Since any weight appearing in $F_{\Phi,W_{\E_6}}$ is an integral linear combination of roots, so any
term $\chi^{\ast}_{\lambda}$ appearing in $F_{\Phi,W_{\E_6}}$ having $\lambda$ in the root lattice. A case
by case calculation enables us to show that any reduced sub-root system $\Phi\subset\E_6$ is stable under the action
of some $\gamma\in\Aut(\E_6)-W_{\E_6}$, so $\theta F_{\Phi,W_{\E_6}}=F_{\Phi,W_{\E_6}}$ (cf. Proposition
\ref{P:gamma-character}). Thus the coefficient of any term $\chi^{\ast}_{\lambda}$ in $F_{\Phi,W_{\E_6}}$
is equal to the coefficient of the term $\chi^{\ast}_{\theta\lambda}$ in $F_{\Phi,W_{\E_6}}$.
By Lemma \ref{L:E6-small length}, to prove equality (\ref{Eq:E6-identity}) it is enough to prove the
corresponding equality about $f_{\Phi,\E_6}$ and to calculate the coefficients of
terms $\chi^{\ast}_{\lambda}$ with $|\lambda|^2=4,5,7,8,9,10$.

For the functions $f_{\Phi,\E_6}$, we have
\begin{eqnarray*}&& f_{3A_2,\E_6}+f_{A_3+2A_1,\E_6}+f_{A_3+A_1,\E_6}-3f_{2A_2+A_1,\E_6}
\\&=&(1-t)^{6}(1-t^{2})^{3}+(1-t)^{5}(1-t^2)^{2}(1-t^3)\\&&+
(1-t)^{4}(1-t^2)^{2}(1-t^3)-3(1-t)^{5}(1-t^2)^{2}\\&=&(1-t)^{4}(1-t^{2})^{2}
((1-2t+2t^{3}-t^4)+(1-t-t^{3}+t^{4})+(1-t^3))\\&&-3(1-t)^{5}(1-t^2)^{2}
\\&=&(1-t)^{4}(1-t^2)^{2}(3-3t)-3(1-t)^{5}(1-t^2)^{2}\\&=&
0.\end{eqnarray*}

The terms $\chi^{\ast}_{\lambda}$ with $|\lambda|^2=4,5,7,8,9,10$ in
\[F_{3A_2,W_{\E_6}}+F_{A_3+2A_1,W_{\E_6}}+F_{A_3+A_1,W_{\E_6}}-
3F_{2A_2+A_1,W_{\E_6}}\]
are \begin{eqnarray*}&&(-12\chi^{\ast}_{\omega_1+\omega_3}-12\chi^{\ast}_{\omega_5+\omega_6}
-3\chi^{\ast}_{2\omega_2})+
(-4\chi^{\ast}_{\omega_1+\omega_3}-4\chi^{\ast}_{\omega_5+\omega_6}
-\chi^{\ast}_{2\omega_2})+\\&&
(-2\chi^{\ast}_{\omega_1+\omega_3}-2\chi^{\ast}_{\omega_5+\omega_6}
-2\chi^{\ast}_{2\omega_2})+
(-6\chi^{\ast}_{\omega_1+\omega_3}-6\chi^{\ast}_{\omega_5+\omega_6}
-2\chi^{\ast}_{2\omega_2})=0, \end{eqnarray*}
\[(36\chi^{\ast}_{\omega_1+\omega_2+\omega_6})+
(6\chi^{\ast}_{\omega_1+\omega_2+\omega_6})+
(0\chi^{\ast}_{\omega_1+\omega_2+\omega_6})-
3(14\chi^{\ast}_{\omega_1+\omega_2+\omega_6})=0,\]
\begin{eqnarray*}&&(-12\chi^{\ast}_{2\omega_1+\omega_5}-12\chi^{\ast}_{\omega_3+2\omega_6}
-12\chi^{\ast}_{\omega_2+\omega_4})+(2\chi^{\ast}_{2\omega_1+\omega_5}+
2\chi^{\ast}_{\omega_3+2\omega_6}+2\chi^{\ast}_{\omega_2+\omega_4})+\\&&
(\chi^{\ast}_{2\omega_1+\omega_5}+\chi^{\ast}_{\omega_3+2\omega_6}+
4\chi^{\ast}_{\omega_2+\omega_4})-3(-3\chi^{\ast}_{2\omega_1+\omega_5}-
3\chi^{\ast}_{\omega_3+2\omega_6}-2\chi^{\ast}_{\omega_2+\omega_4})=0,
\end{eqnarray*}
\begin{eqnarray*}&&
(3\chi^{\ast}_{2\omega_1+2\omega_6}+12\chi^{\ast}_{\omega_1+\omega_2+\omega_3}
+12\chi^{\ast}_{\omega_2+\omega_5+\omega_6})+
(-\chi^{\ast}_{2\omega_1+2\omega_6}-4\chi^{\ast}_{\omega_1+\omega_2+\omega_3}-
\\&&4\chi^{\ast}_{\omega_2+\omega_5+\omega_6})+(\chi^{\ast}_{2\omega_1+2\omega_6}
-2\chi^{\ast}_{\omega_1+\omega_2+\omega_3}-2\chi^{\ast}_{\omega_2+\omega_5+\omega_6})
-\\&&3(\chi^{\ast}_{2\omega_1+2\omega_6}+2\chi^{\ast}_{\omega_1+\omega_2+\omega_3}
+2\chi^{\ast}_{\omega_2+\omega_5+\omega_6})=0,\end{eqnarray*}
\[(2\chi^{\ast}_{3\omega_2})+(-\chi^{\ast}_{3\omega_2})+
(-3\chi^{\ast}_{\omega_1+\omega_4+\omega_6}-\chi^{\ast}_{3\omega_2})-
3(-\chi^{\ast}_{\omega_1+\omega_4+\omega_6})=0\] and
\begin{eqnarray*}&&
(-3\chi^{\ast}_{\omega_1+2\omega_5}-3\chi^{\ast}_{2\omega_3+\omega_6}-
6\chi^{\ast}_{\omega_1+2\omega_2+\omega_6})+(2\chi^{\ast}_{\omega_1+2\omega_5}+
2\chi^{\ast}_{2\omega_3+\omega_6}+4\chi^{\ast}_{\omega_1+2\omega_2+\omega_6})\\&&+
(\chi^{\ast}_{\omega_1+2\omega_5}+\chi^{\ast}_{2\omega_3+\omega_6}+
2\chi^{\ast}_{\omega_1+2\omega_2+\omega_6})=0\end{eqnarray*}
respectively.

Therefore the equality (\ref{Eq:E6-identity}) follows.
\end{proof}

\smallskip

\textbf{Type $\E_7$.} As in the proof of Theorem \ref{character-exceptional}, we have observed
that those weights that appearing more than once in $\{2\delta'_{\Phi}|\ \Phi\subset\E_7\}$ and the
reduced sub-root systems for which they appeared in are as follows,
\begin{itemize}
\item[(1)]{$2\omega_1$: $A_2$, $4A'_1$, appears two times.}
\item[(2)]{$\omega_1+\omega_6$: $A_2+A_1$, $5A_1$, appears two times.}
\item[(3)]{$\omega_4$: $A_2+2A_1$, $6A_1$, appears two times.}
\item[(4)]{$2\omega_2$: $A_2+3A_1$, $7A_1$, appears two times.}
\item[(5)]{$2\omega_3$: $3A_2$, $A_3+2A_1$, appears two times.}
\item[(6)]{$2\omega_1+2\omega_6$: $A_4$, $2A_3$, appears two times.}
\item[(7)]{$\omega_1+\omega_4+\omega_6$: $A_4+A_1$, $2A_3+A_1$, appears two times.}
\end{itemize}

By the conclusion from the $\E_6$ case and the second and the third observations in the beginning
of this subsection, we get
\begin{equation} F_{A_2,W_{\E_7}}+F_{(4A_1)',W_{\E_7}}-2F_{3A_1,W_{\E_7}}=0,\end{equation}
\begin{equation} F_{A_2+A_1,W_{\E_7}}+F_{5A_1,W_{\E_7}}-2F_{4A_1,W_{\E_7}}=0,\end{equation}
\begin{equation} F_{A_2+2A_1,W_{\E_7}}+F_{6A_1,W_{\E_7}}-2F_{5A_1,W_{\E_7}}=0,\end{equation}
\begin{equation} F_{A_2+3A_1,W_{\E_7}}+F_{7A_1,W_{\E_7}}-2F_{6A_1,W_{\E_7}}=0,\end{equation}
\begin{equation} F_{3A_2,W_{\E_7}}+F_{A_3+2A_1,W_{\E_7}}+F_{A_3+A_1,W_{\E_7}}-3F_{2A_2+A_1,W_{\E_7}}=0.
\end{equation}

We prove that these relations generate all linear relations among
$\{F_{\Phi,W_{\E_7}}|\ \Phi\subset\E_7\}$. By the first observation, to show this, we just need
to show: any non-trivial linear combination \[c_1 F_{A_4,W_{\E_7}}+c_2 F_{2A_3,W_{\E_7}}\] is
not a linear combination of the characters $\{F_{\Phi,W_{\E_7}}:\ |\delta'_{\Phi}|^2<
|\delta'_{A_4}|^2=20\}$ and any non-trivial linear combination \[c_1 F_{A_4+A_1,W_{\E_7}}+
c_2 F_{2A_3+A_1,W_{\E_7}}\] is not a linear combination of the characters
$\{F_{\Phi,W_{\E_7}}||\delta'_{\Phi}|^2<|\delta'_{A_4+A_1}|^2=21\}$.

Suppose \[c_1 F_{A_4,W_{\E_7}}+c_2 F_{2A_3,W_{\E_7}}=\sum_{1\leq i\leq s, |2\delta'_{\Phi_{i}}|
\leq 19}c_{i}F_{\Phi_{i},W_{\E_7}}.\] Then \[c_1 f_{A_4,\E_7}+c_2 f_{2A_3,\E_7}=
\sum_{1\leq i\leq s, |2\delta'_{\Phi_{i}}|\leq 19}c_{i}f_{\Phi_{i},\E_7}\] by the fourth observation.
Since there is no $\Phi\subset\E_7$ with $|2\delta'_{\Phi}|^2=19$, the RHS of the above equation has
a degree at most $18$. Comparing the coefficients of the terms $t^{20}$ and $t^{19}$, we get
$c_1+c_2=0$ and $-4c_1-6c_2=0$. Hence $c_1=c_2=0$.

After calculating the highest (=longest) terms, we get
\begin{eqnarray*}F_{A_4+A_1,W_{\E_7}}&=&-\chi^{\ast}_{\omega_1+\omega_4+\omega_6}+
2\chi^{\ast}_{\omega_2+\omega_5+\omega_6}+\chi^{\ast}_{2\omega_1+2\omega_6}
\\&&+2\chi^{\ast}_{\omega_1+\omega_2+\omega_3+\omega_7}+\textrm{lower terms},\end{eqnarray*}
\begin{eqnarray*}F_{2A_3+A_1,W_{\E_7}}&=&-\chi^{\ast}_{\omega_1+\omega_4+\omega_6}+
2\chi^{\ast}_{\omega_2+\omega_5+\omega_6}+\chi^{\ast}_{2\omega_1+2\omega_6}
\\&&+4\chi^{\ast}_{\omega_1+\omega_2+\omega_3+\omega_7}+\textrm{lower terms}.\end{eqnarray*}
Considering the coefficients of the terms $\chi^{\ast}_{\omega_1+\omega_4+\omega_6}$ and
$\chi^{\ast}_{\omega_1+\omega_2+\omega_3+\omega_7}$, we see any non-trivial linear combination
$c_1 F_{A_4+A_1,W_{\E_7}}+c_2 F_{2A_3+A_1,W_{\E_7}}$ is not a linear combination
of the characters $\{F_{\Phi,W_{\E_7}}:\ |\delta'_{\Phi}|^2<|\delta'_{A_4+A_1}|^2=21\}$.

\smallskip

\textbf{Type $\E_8$.} As in the proof of Theorem \ref{character-exceptional}, we have observed that
those weights that appearing more than once in $\{2\delta'_{\Phi}|\ \Phi\subset\E_8\}$ and the sub-root
systems for which they appeared in are as follows,
\begin{itemize}
\item[(1)]{$2\omega_8$: $A_2$, $4A'_1$, appears two times.}
\item[(2)]{$\omega_1+\omega_8$: $A_2+A_1$, $5A_1$, appears two times.}
\item[(3)]{$\omega_6$: $A_2+2A_1$, $6A_1$, appears two times.}
\item[(4)]{$\omega_3$: $A_2+3A_1$, $7A_1$, appears two times.}
\item[(5)]{$2\omega_1$: $2A_2$, $A_2+4A_1$, $8A_1$, appears three times.}
\item[(6)]{$2\omega_7$: $A_3+2A_1$, $3A_2$, appears two times.}
\item[(7)]{$\omega_2+\omega_7$: $A_3+3A_1$, $3A_2+A_1$, appears two times.}
\item[(8)]{$\omega_1+\omega_6$: $A_3+4A_1$, $A_3+A_2$, appears two times.}
\item[(9)]{$2\omega_2$: $A_3+A_2+2A_1$, $4A_2$, appears two times.}
\item[(10)]{$2\omega_1+2\omega_8$: $A_4$, $2A_3$, appears two times.}
\item[(11)]{$\omega_1+\omega_6+\omega_8$: $A_4+A_1$, $2A_3+A_1$, appears two times.}
\item[(12)]{$\omega_4+\omega_8$: $A_4+2A_1$, $2A_3+2A_1$, appears two times.}
\item[(13)]{$2\omega_2+2\omega_8$: $D_4+A_2$, $D_4+4A_1$, appears two times.}
\item[(14)]{$2\omega_5$: $A_5+A_2+A_1$, $2A_4$, appears two times.}
\item[(15)]{$2\omega_1+2\omega_6$: $A_6$, $2D_4$, appears two times.}
\end{itemize}

By the conclusion from the $\E_7$ case and the second and the third observations, we get
\begin{equation} F_{A_2,W_{\E_8}}+F_{(4A_1)',W_{\E_8}}-2F_{3A_1,W_{\E_8}}=0,\end{equation}
\begin{equation} F_{A_2+A_1,W_{\E_8}}+F_{5A_1,W_{\E_8}}-2F_{4A_1,W_{\E_8}}=0,\end{equation}
\begin{equation} F_{A_2+2A_1,W_{\E_8}}+F_{6A_1,W_{\E_8}}-2F_{5A_1,W_{\E_8}}=0,\end{equation}
\begin{equation} F_{A_2+3A_1,W_{\E_8}}+F_{7A_1,W_{\E_8}}-2F_{6A_1,W_{\E_8}}=0,\end{equation}
\begin{equation} F_{A_2+4A_1,W_{\E_8}}+F_{8A_1,W_{\E_8}}-2F_{7A_1,W_{\E_8}}=0,\end{equation}
\begin{equation} F_{2A_2,W_{\E_8}}+F_{A_2+4A_1,W_{\E_8}}-2F_{A_2+3A_1,W_{\E_8}}=0,\end{equation}
\begin{equation} F_{3A_2,W_{\E_8}}+F_{A_3+2A_1,W_{\E_8}}+F_{A_3+A_1,W_{\E_8}}-3F_{2A_2+A_1,W_{\E_8}}=0,
\end{equation}
\begin{equation} F_{3A_2+A_1,W_{\E_8}}+F_{A_3+3A_1,W_{\E_8}}+F_{A_3+2A_1,W_{\E_8}}-3F_{2A_2+2A_1,W_{\E_8}}=0,
\end{equation}
\begin{equation} F_{4A_2,W_{\E_8}}+F_{A_3+A_2+2A_1,W_{\E_8}}+F_{A_3+A_2+A_1,W_{\E_8}}-3F_{3A_2+A_1,W_{\E_8}}=0,
\end{equation}
\begin{equation} F_{D_4+A_2,W_{\E_8}}+F_{D_4+4A_1,W_{\E_8}}-2F_{D_4+3A_1,W_{\E_8}}=0.
\end{equation}
In the below we prove that these relations generate all linear relations among
$\{F_{\Phi,W_{\E_8}}|\ \Phi\subset\E_8\}$.

Since $\Psi=\E_8$ has no sub-root systems $\Phi$ with $e(\Phi)=19$, one can show similarly
as in the $\E_7$ case that any non-trivial linear combination $c_1 F_{A_4,W_{\E_8}}+
c_2 F_{2A_3,W_{\E_8}}$ is not a linear combination of the characters
$\{F_{\Phi,W_{\E_8}}:\ |\delta'_{\Phi}|^2<|\delta'_{A_4+A_1}|^2=20\}$.

By calculating the highest (=longest) terms, we get
\begin{eqnarray*}F_{A_4+A_1,W_{\E_8}}&=&-\chi^{\ast}_{\omega_1+\omega_6+\omega_8}+
2\chi^{\ast}_{\omega_1+\omega_5}+\chi^{\ast}_{2\omega_1+2\omega_8}
\\&&+2\chi^{\ast}_{\omega_2+\omega_7+\omega_8}+\textrm{lower terms},\end{eqnarray*}
\begin{eqnarray*}F_{2A_3+A_1,W_{\E_8}}&=&-\chi^{\ast}_{\omega_1+\omega_6+\omega_8}+
2\chi^{\ast}_{\omega_1+\omega_5}+\chi^{\ast}_{2\omega_1+2\omega_8}
\\&&+4\chi^{\ast}_{\omega_2+\omega_7+\omega_8}+\textrm{lower terms},\end{eqnarray*}
Considering the coefficients of the terms $\chi^{\ast}_{\omega_1+\omega_6+\omega_8}$ and
$\chi^{\ast}_{\omega_2+\omega_7+\omega_8}$, we see that any non-trivial linear combination
$c_1 F_{A_4+A_1,W_{\E_8}}+c_2 F_{2A_3+A_1,W_{\E_8}}$ is not a linear combination
of the characters $\{F_{\Phi,W_{\E_8}}:\ |\delta'_{\Phi}|^2<|\delta'_{A_4+A_1}|^2=21\}$.

We have \begin{eqnarray*}F_{A_4+2A_1,W_{\E_8}}&=&\chi^{\ast}_{\omega_4+\omega_8}-
4\chi^{\ast}_{\omega_1+\omega_6+\omega_8}-2\chi^{\ast}_{\omega_2+\omega_3}\\&&+\textrm{lower terms},
\end{eqnarray*} \begin{eqnarray*}F_{2A_3+2A_1,W_{\E_8}}&=&\chi^{\ast}_{\omega_4+\omega_8}-
4\chi^{\ast}_{\omega_1+\omega_6+\omega_8}-4\chi^{\ast}_{\omega_2+\omega_3}\\&&+\textrm{lower terms}.
\end{eqnarray*} Considering the coefficients of the terms $\chi^{\ast}_{\omega_4+\omega_8}$ and
$\chi^{\ast}_{\omega_2+\omega_3}$, we see that any non-trivial linear combination
$c_1 F_{A_4+2A_1,W_{\E_8}}+c_2 F_{2A_3+2A_1,W_{\E_8}}$ is not a linear combination
of the characters $\{F_{\Phi,W_{\E_8}}:\ |\delta'_{\Phi}|^2<|\delta'_{A_4+A_1}|^2=21\}$.

For the weight $2\omega_5$ and sub-root systems $A_5+A_2+A_1$, $2A_4$, suppose some
non-trivial linear combination $c_1 F_{A_5+A_2+A_1,W_{\E_8}}+c_2 F_{2A_4,W_{\E_8}}$ is a
linear combination of the characters $\{F_{\Phi,W_{\E_8}}:\ |\delta'_{\Phi}|^2<|2\omega_5|^2=40\}$.
The sub-root systems with $33\leq |2\delta'_{\Phi}|^{2}\leq 39$ include
$\{A_5+A_2,D_4+A_3,A_5+2A_1,A_5+A_1,(A_5+A_1)',A_5\}$, so there exists constants
$c_3,c_4,c_5,c_6,c_7,c_8$ such that
\begin{eqnarray*}
&& c_1f_{A_5+A_2+A_1,\E_8}+c_2 f_{2A_4,\E_8}+c_3 f_{A_5+A_2,\E_8}+c_4 f_{D_4+A_3,\E_8}+\\&&
 c_5 f_{A_5+2A_1,\E_8}+c_6 f_{A_5+A_1,\E_8}+c_7 f_{(A_5+A_1)',\E_8}+c_8 f_{A_5,\E_8}
\end{eqnarray*} is a polynomial of degree $\leq 32$.

By formulas in Subsection \ref{SS:generating function}, we have
\[f_{A_5+A_2+A_1}(t)=-t^{40}+8t^{39}-23t^{38}+19t^{37}+38t^{36}-90t^{35}+39t^{34}+\textrm{lower terms},\]
\[f_{2A_4}(t)=t^{40}-8t^{39}+22t^{38}-12t^{37}-53t^{36}+88t^{35}+2t^{34}+\textrm{lower terms},\]
\[f_{D_4+A_3}(t)=t^{38}-7t^{37}+16t^{36}-4t^{35}-33t^{34}+\textrm{lower terms},\]
\[f_{A_5+A_1}(t)=f_{(A_5+A_1)'}(t)=t^{36}-6t^{35}+11t^{34}+\textrm{lower terms}.\]
Moreover, \[f_{A_5+A_2}(t)=t^{39}+\textrm{lower terms},\]
\[f_{A_5+2A_1}(t)=-t^{37}+\textrm{lower terms},\]
\[f_{A_5}(t)=-t^{35}+\textrm{lower terms}.\]

Considering the coefficients of the terms $t^{40}$ and $t^{39}$, we get
\[-c_1+c_2=0,\quad 8c_1-8c_2+c_3=0.\] Thus $c_2=c_1$ and $c_3=0$. Considering the coefficients of
the terms $t^{38}$, $t^{37}$, we get \[-23c_1+22c_2+c_4=0,\quad 19c_1-12c_2-7c_4-c_5=0.\] Thus
$c_4=c_1$ and $c_5=0$. Considering the coefficients of $t^{36}$, $t^{35}$, we get
\[38c_1-53c_2+16c_4+(c_6+c_7)=0\] and \[-90c_1+88c_2-4c_4-6(c_6+c_7)-c_8=0.\] Hence
$c_6+c_7=-c_1$ and $c_8=0$. Finally, considering the coefficient of $t^{34}$,
we get \[0=39c_1+2c_2-33c_3+11(c_6+c_7)=-3c_1.\] Therefore $c_1=c_2=0$.

For the weight $2\omega_1+2\omega_6$ and sub-root systems $A_6$, $2D_4$, since
$\Psi=\E_8$ has no sub-root systems $\Phi$ with $e(\Phi)=55$, we can show
any non-trivial linear combination \[c_1 F_{A_6,W_{\E_8}}+c_2 F_{2D_4,W_{\E_8}}\] is not
a linear combination of the characters
$\{F_{\Phi,W_{\E_8}}:\ |\delta'_{\Phi}|^2<|2\omega_1+2\omega_6|^2=56\}$.

This proves the conclusion in the $\E_8$ case.

\smallskip

\textbf{Type $\F_4$.} As in the proof of Theorem \ref{character-exceptional}, we have observed that
those weights appearing more than once in $\{2\delta'_{\Phi}|\ \Phi\subset\F_4\}$ and the sub-root
systems for which they appeared in are as follows:
\begin{itemize}
\item[(1)]{$\omega_1$: $A_1^{L}$, $2A_1^{S}$, appears 2 times.}
\item[(2)]{$\omega_3$: $A_1^{L}+A_1^{S}$, $3A_1^{S}$, appears 2 times.}
\item[(3)]{$2\omega_4$: $A_2^{S}$, $2A_1^{L}$, $A_1^{L}+2A_1^{S}$, $4A_1^{S}$,
appears 4 times.}
\item[(4)]{$\omega_2$: $3A_1^{L}$, $2A_1^{L}+2A_1^{S}$, $A_1^{L}+A_2^{S}$,
appears 3 times.}
\item[(5)]{$2\omega_1$: $A_2^{L}$, $4A_1^{L}$, appears 2 times.}
\item[(6)]{$\omega_1+2\omega_4$: $A_3^{S}$, $B_2$, appears 2 times.}
\item[(7)]{$2\omega_3$: $A_2^{L}+A_2^{S}$, $A_1^{L}+B_2$, $2A_1^{S}+B_2$,
$A_1^{L}+A_3^{S}$, appears 4 times.}
\item[(8)]{$2\omega_1+2\omega_4$: $A_3^{L}$, $2B_2$, appears 2 times.}
\item[(9)]{$2\omega_3+2\omega_4$: $D_4^{S}$, $C_3$, appears 2 times.}
\end{itemize}

Since $\C_4\subset\F_4$, $\B_4\subset\F_4$, using the following equalities
\[c_1+d_2=2a_2,\quad c_1a_2+d_2a_2=a_2^{2},\quad c_1d_2+d_2^{2}=2a_2d_2, \]
\[c_1^{2}+c_1d_2=2a_2c_1,\quad a_3=c_1d_2,\quad c_1^{3}+c_1^{2}d_2=2a_2c_1^{2},\]
\[c_1a_3=c_1^{2}d_2,\quad a_3+d_2^{2}=2a_2d_2,\quad c_1c_2+d_2c_2=2a_2c_2,\]
and by the second and the third observations, we get
\begin{equation} F_{A_1^{L},W_{\F_4}}+F_{2A_1^{S},W_{\F_8}}-2F_{A_1^{S},W_{\F_4}}=0,\end{equation}
\begin{equation} F_{A_1^{L}+A_1^{S},W_{\F_4}}+F_{3A_1^{S},W_{\F_4}}-2F_{2A_1^{S},W_{\F_4}}=0,\end{equation}
\begin{equation} F_{A_1^{L}+2A_1^{S},W_{\F_4}}+F_{4A_1^{S},W_{\F_4}}-2F_{3A_1^{S},W_{\F_4}}=0,\end{equation}
\begin{equation} F_{2A_1^{L},W_{\F_4}}+F_{A_1^{L}+2A_1^{S},W_{\F_4}}-2F_{A_1^{L}+A_1^{S},W_{\F_4}}=0,\end{equation}
\begin{equation} F_{A_2^{S},W_{\F_4}}-F_{A_1^{L}+2A_1^{S},W_{\F_4}}=0,\end{equation}
\begin{equation} F_{3A_1^{L},W_{\F_4}}+F_{2A_1^{L}+2A_1^{S},W_{\F_4}}-2F_{2A_1^{L}+A_1^{S},W_{\F_4}}=0,\end{equation}
\begin{equation} F_{A_1^{L}+A_2^{S},W_{\F_4}}-F_{2A_1^{L}+2A_1^{S},W_{\F_4}}=0,\end{equation}
\begin{equation} F_{A_2^{L},W_{\F_4}}+F_{4A_1^{L},W_{\F_4}}-2F_{3A_1^{L},W_{\F_4}}=0,\end{equation}]
\begin{equation} F_{A_1^{L}+B_2,W_{\F_4}}+F_{2A_1^{S}+B_2,W_{\F_4}}-2F_{A_1^{S}+B_2,W_{\F_4}}=0.\end{equation}
Moreover, we will prove three more equalities \begin{equation} F_{A_3^{S},W_{\F_4}}-F_{B_2,W_{\F_4}}+
2F_{A_2^{L}+A_1^{S},W_{\F_4}}-2F_{2A_1^{L}+2A_1^{S},W_{\F_4}}=0,\end{equation}
\begin{equation} F_{A_2^{L}+A_2^{S},W_{\F_4}}+F_{A_1^{L}+B_2,W_{\F_4}}-F_{A_1^{S}+B_2,W_{\F_4}}
+2F_{B_2,W_{\F_4}}-3F_{A_2^{L}+A_1^{S},W_{\F_4}}=0,\end{equation}
\begin{equation} F_{A_1^{L}+A_3^{S},W_{\F_4}}-F_{A_1^{L}+B_2,W_{\F_4}}+2F_{A_1^{S}+B_2,W_{\F_4}}-
2F_{A_3^{S},W_{\F_4}}=0.\end{equation}

We will show these relations generate all linear relations
among $\{F_{\Phi,W_{\F_8}}|\ \Phi\subset\F_8\}$.

\begin{lemma}\label{L:F4-short}
Given an integer $k$ with $1\leq k\leq 12$ and $k\neq 9$, there exists a unique
dominant integral weight $\lambda$ for the root system $\Psi=\F_4$ such that
$|\lambda|^{2}=k$.

Given $k=9$, there exist two dominant integral weights $\lambda$ for the root system
$\Psi=\F_4$ such that $|\lambda|^{2}=k$. They are $3\omega_4$ and $\omega_1+\omega_3$.
\end{lemma}

\begin{proof}
The inverse to the Cartan matrix of $\F_4$ is $\left(\begin{array}{cccc}2&3&4&2\\
3&6&8&4\\2&4&6&3\\1&2&3&2\\\end{array}\right)$, so for a dominant integral weight
$\lambda=\sum_{1\leq i\leq 4}a_{i}\omega_{i}$ ($a_{i}\in\bbZ_{\geq 0}$), we have
\[|\lambda|^{2}=2a_1^{2}+6a_2^{2}+3a_3^{2}+a_4^{2}+(6a_1a_2+4a_1a_3+2a_1a_4+
8a_2a_3+4a_2a_4+3a_3a_4).\]

Suppose $|\lambda|^2\leq 12$. First we must have $a_2=0,1$. If $a_2=1$, we have
\[\lambda=\omega_2\ (6),\omega_2+\omega_4\ (11).\] We also have $a_3\leq 2$. When
$a_3\geq 1$, we have \[\lambda=2\omega_3\ (12),\omega_3+\omega_1\ (9),
\omega_3+\omega_4\ (7),\omega_3\ (3).\] When $a_2=a_3=0$, we have
\[\lambda=2\omega_1\ (8),\omega_1+2\omega_4\ (10), \omega_1+\omega_4\ (5),
\omega_1\ (2),3\omega_4\ (9), 2\omega_4\ (4),\omega_4 (1).\]

This proves the lemma.
\end{proof}

\begin{proof}[Proof of the three equalities]
First any term $\chi^{\ast}_{\lambda}$ appearing in these three equalities has $|\lambda|^2\leq 12$.
By Lemma \ref{L:F4-short}, to prove these equalities, we just need to prove the corresponding
equalities about generating functions $f_{\Phi,\F_4}$ and to calculate the coefficients of the
terms $\chi^{\ast}_{\lambda}$ with $|\lambda|^2=9$.

For the functions $f_{\Phi,\F_4}$, we have \begin{eqnarray*}&& f_{A_3^{S},\F_4}-f_{B_2,\F_4}+
2f_{A_2^{L}+A_1^{S},\F_4}-2f_{2A_1^{L}+2A_1^{S},\F_4}\\&=&(1-t)^{3}(1-t^2)^{2}(1-t^3)-
(1-t)(1-t^2)(1-t^3)(1-t^4)\\&&+2(1-t)(1-t^2)^{2}(1-t^4)-2(1-t)^{2}(1-t^{2})^{2}\\&=&
(1-t)(1-t^2)(1-t^3)(-2t+2t^3)+2(1-t)(1-t^2)^{2}(t-t^4)\\&=&0,\end{eqnarray*}
\begin{eqnarray*}&& f_{A_2^{L}+A_2^{S},\F_4}+F_{A_1^{L}+B_2,\F_4}-F_{A_1^{S}+B_2,\F_4}+2F_{B_2,\F_4}
-3F_{A_2^{L}+A_1^{S},\F_4}\\&=&(1-t)^{2}(1-t^{2})^{3}(1-t^4)+(1-t)(1-t^2)^{2}(1-t^3)(1-t^4)\\&&-
(1-t)^{2}(1-t^2)(1-t^3)(1-t^4)+2(1-t)(1-t^2)(1-t^3)(1-t^4)\\&&-3(1-t)(1-t^2)^{2}(1-t^4)\\&=&
(1-t)(1-t^{2})^{2}(1-t^4)(2-t-t^2)+(1-t^2)^{2}(1-t^3)(1-t^4)\\&&-3(1-t)(1-t^2)^{2}(1-t^4)\\&=&
(1-t^{2})^{2}(1-t^4)(3-3t)-3(1-t)(1-t^2)^{2}(1-t^4)\\&=&0,\end{eqnarray*}
\begin{eqnarray*}&&f_{A_1^{L}+A_3^{S},\F_4}-f_{A_1^{L}+B_2,\F_4}+2f_{A_1^{S}+B_2,\F_4}-2f_{A_3^{S},\F_4}
\\&=&(1-t)^{3}(1-t^2)^{3}(1-t^3)-(1-t)(1-t^2)^{2}(1-t^3)(1-t^4)\\&&+2(1-t)^{2}(1-t^2)(1-t^3)(1-t^4)-
2(1-t)^{3}(1-t^{2})^{2}(1-t^3)\\&=&(1-t)(1-t^2)^{2}(1-t^3)(-2t+2t^3)\\&&+
2(1-t)^{2}(1-t^2)(1-t^3)(t+t^2-t^3-t^4)\\&=&0.\end{eqnarray*}

The terms $\chi^{\ast}_{\lambda}$ with $|\lambda|^2=9$ in
\[ F_{A_3^{S},W_{\F_4}}-F_{B_2,W_{\F_4}}+2F_{A_2^{L}+A_1^{S},W_{\F_4}}-
2F_{2A_1^{L}+2A_1^{S},W_{\F_4}},\]
\[ F_{A_2^{L}+A_2^{S},W_{\F_4}}+F_{A_1^{L}+B_2,W_{\F_4}}-F_{A_1^{S}+B_2,W_{\F_4}}
+2F_{B_2,W_{\F_4}}-3F_{A_2^{L}+A_1^{S},W_{\F_4}},\]
\[ F_{A_1^{L}+A_3^{S},W_{\F_4}}-F_{A_1^{L}+B_2,W_{\F_4}}+2F_{A_1^{S}+B_2,W_{\F_4}}-
2F_{A_3^{S},W_{\F_4}}\]
are \[(-2\chi^{\ast}_{\omega_1+\omega_3}-\chi^{\ast}_{3\omega_4})-
(-\chi^{\ast}_{3\omega_4})+2\chi^{\ast}_{\omega_1+\omega_3}=0,\]
\[(4\chi^{\ast}_{\omega_1+\omega_3}+2\chi^{\ast}_{3\omega_4})+(-\chi^{\ast}_{3\omega_4})
-(\chi^{\ast}_{\omega_1+\omega_3}-\chi^{\ast}_{3\omega_4})+
2(-\chi^{\ast}_{3\omega_4})-3\chi^{\ast}_{\omega_1+\omega_3}=0,\]
\[(-6\chi^{\ast}_{\omega_1+\omega_3}-\chi^{\ast}_{3\omega_4})-
(-\chi^{\ast}_{3\omega_4})+2(\chi^{\ast}_{\omega_1+\omega_3}-\chi^{\ast}_{3\omega_4})
-2(-2\chi^{\ast}_{\omega_1+\omega_3}-\chi^{\ast}_{3\omega_4})=0\]
respectively.

Therefore we get the three equalities.
\end{proof}

For the weights $2\omega_1+2\omega_4$ and $2\omega_3+2\omega_4$, since $\Psi=\F_4$ has no sub-root
systems $\Phi$ with $e(\Phi)=19$ or $e(\Phi)=27$, we can show any non-trivial linear combination
$c_1 F_{A_3^{L},W_{\F_4}}+c_2 F_{2B_2,W_{\F_4}}$ is not a linear combination of the characters
$\{F_{\Phi,W_{\F_4}}:\ |\delta'_{\Phi}|^2<|\delta'_{A_4+A_1}|^2=20\}$ and any non-trivial linear
combination $c_1 F_{D_4^{S},W_{\F_4}}+c_2 F_{C_3,W_{\F_4}}$ is not a linear combination of the
characters $\{F_{\Phi,W_{\F_4}}:\ |\delta'_{\Phi}|^2<|\delta'_{A_4+A_1}|^2=28\}$. Thus the relations
as listed above generate all linear relations among $\{F_{\Phi,W_{\F_4}}|\ \Phi\subset\F_4\}$.

\smallskip

\textbf{Type $\G_2$.} For $\Psi=G_2$, the only non-conjugate reduced sub-root systems with conjugate
leading terms are $A_2^{S}$ and $A_1^{L}+A_1^{S}$. We have \begin{equation} F_{A_2^{S},W_{\G_2}}+
F_{A_1^{L}+A_1^{S},W_{\G_2}}+F_{A_1^{L},W_{\G_2}}-3F_{A_1^{S},W_{\G_2}}=0.\end{equation} This is the
unique linear relation between the characters $\{F_{\Phi,W_{\G_2}}|\ \Phi\subset\G_2\}$.

\begin{remark}\label{R:LP-equal-linear relation}
Given an irreducible root system $\Psi$, if $\Psi_0\not\cong\BC_{n}$, then the characters
$\{F_{\Phi,W_{\Psi}}:\Phi\subset\Psi_0,\rank\Phi=\Psi_0\}$ are linearly independent. If $\Psi=\BC_{n}$,
for two sub-root systems $\Phi_1,\Phi_2$ of $\Psi$ with $\rank\Phi_1=\rank\Phi_2=n$,
$F_{\Phi_1,W_{n}}=F_{\Phi_2,W_{n}}$ if and only if $\Phi_1\sim_{W_{n}}\Phi_2$. These follow from
the results we showed above. Note that, these were also proved by Larsen-Pink. Actually this is the
essential part to the proof of Theorem 1 in \cite{Larsen-Pink}.
On the other hand, Larsen-Pink have proved the existence of algebraic relations among the polynomials
$\{b_{n},c_{n},d_{n+1}|n\geq 1\}$. And they used this to construct non-conjugate closed subgroups
with equal dimension data. 
\end{remark}

\begin{proposition}\label{P:linear-exceptional}
Given a compact connectee Lie group $G$ of type $\E_6$, $\E_7$, $\E_8$ or $\F_4$, there exist non-isomorphic
closed connected full rank subgroups with linearly dependent dimension data.
\end{proposition}
\begin{proof}
In the case that $G$ is of type  $\E_6$, $\E_7$, $\E_8$, $G$ possesses a Levi subgroup of type $\D_4$.
Hence the conclusion follows from Proposition \ref{P:linear-BCD}. In the case that $G$ is of type $\F_4$,
$G$ possesses a subgroup isomorphic to $\Spin(8)$. Therefore the conclusion follows from Proposition
\ref{P:linear-BCD}.
\end{proof}

\section{Comparison of Question \ref{Q:linear dependence} and Question \ref{Q:dependent-character}}\label{S:Gamma0}

In this section we give constructions which show that Question \ref{Q:dependent-character} (or
\ref{Q:equal-character}) is not an excessive generalization of Question \ref{Q:linear dependence} (or
\ref{Q:dimension data}), as we promised after introducing Questions \ref{Q:equal-character} and
\ref{Q:dependent-character}. The significance of these constructions is that each equality (or linear relation)
we found in Question \ref{Q:equal-character} (or \ref{Q:dependent-character}) indeed corresponds to an equality
(or a linear relation) of dimension data of closed connected subgroups in a suitable group. Another consequence
of these constructions is showing that the group $\Gamma^{\circ}$ could be quite arbitrary.

Precisely, what we do is as follows. Given a root system $\Psi'$ and a finite group $W$ acting faithfully
on $\Psi'$ and containing $W_{\Psi'}$, we construct some pair $(G,T)$ with $G$ a compact Lie group
and $T$ a closed connected torus in $G$ such that:
\begin{itemize}
\item[(1)]{$\rank\Psi_{T}=\dim T=\rank\Psi'$.}
\item[(2)]{$\Psi'\subset\Psi_{T}$ and is stable under $\Gamma^{\circ}$.}
\item[(3)]{$\Gamma^{\circ}=W$ as groups acting on $\Psi'$.}
\item[(4)]{For each reduced sub-root system $\Phi$ of $\Psi'$, there exists a connected closed subgroup $H$
of $G$ with $T$ a maximal torus of $H$ and with root system $\Phi(H,T)=\Phi$.}
\end{itemize}

Sub-root systems of any given irreducible root system $\Psi_0$ are classified in \cite{Oshima}, in Section
\ref{S:sub-root systems} we have discussed this classification in the case that $\Psi_0$ is a classical
irreducible root system or an exceptional irreducible root system of type $\F_4$ or $\G_2$.



\begin{proposition}\label{P:construction-simple}
Given an irreducible root system $\Psi_0$, there exists a compact connected simple Lie group $G$ and
a closed connected torus $T$ in $G$ such that:
\begin{itemize}
\item[(1)]{$\rank\Psi_{T}=\dim T=\rank\Psi_0$.}
\item[(2)]{$\Psi_0\subset\Psi_{T}$ and is stable under $\Gamma^{\circ}$.}
\item[(3)]{$\Gamma^{\circ}=W_{\Psi_0}$ as groups acting on $\Psi_0$.}
\item[(4)]{For each reduced sub-root system $\Phi$ of $\Psi_0$, there exists a connected closed subgroup $H$
of $G$ with $T$ a maximal torus of $H$ and with root system $\Phi(H,T)=\Phi$.}
\end{itemize}
\end{proposition}

\begin{proof}
{\it Simply laced case.}
In the case that $\Psi_0$ is a simply laced irreducible root system, let $\fru_0$ be a compact simple Lie algebra with
root system isomorphic to $\Psi_0$. Taking $G=\Int(\fru_0)$ and $T$ a maximal torus of $G$, the conclusion
in the proposition is satisfied for $(G,T)$.

{\it Type $\C_{n}$.}
In the case that $\Psi_0=\C_n$, let $G=\SU(2n)$, \[T=\{\diag\{z_1,\dots,z_{n},z_1^{-1},\dots,z_{n}^{-1}\}|\ |z_1|=
|z_2|=\cdots=|z_{n}|=1\}\] and $\theta=\Ad\left(\begin{array}{cc}0&I_{n}\\I_{n}&0\\\end{array}\right)$.
Then $\theta$ is an involutive automorphism of $G$ and $\Lie T$ is a maximal abelian subspace of
$\frp_0=\{X\in\Lie G|\ \theta(X)=-X\}$. In this example, the restricted root system $\Phi(G,T)\cong\C_{n}$
(cf. \cite{Knapp}, Page 424), each long root occurs with multiplicity one, each short root occurs with
multiplicity two and $C_{G}(T)$ is the subgroup of diagonal matrices in $G$. We have
\[\Gamma^{\circ}=N_{G}(T)/C_{G}(T)=\{-1\}^{n}\rtimes S_{n}=W_{C_{n}}.\]
Denote by \[J_{k}=\left(\begin{array}{cccc}\frac{1}{\sqrt{2}}I_{k}&&\frac{1}{\sqrt{2}}iI_{k}
&\\&I_{n-k}&&\\\frac{1}{\sqrt{2}}iI_{k}&&\frac{1}{\sqrt{2}}I_{k}&\\&&&I_{n-k}
\\\end{array}\right),\] which is an $2n\times 2n$ matrix.
Define the subgroups \[A_{k-1}=\big\{\diag\{A,K,\overline{A},\overline{K}\}|\ A\in U(k), K\in T_{n-k}\big\},\]
\[C_{k}=\big\{\left(\begin{array}{cccc}A&0&B&0\\0&K&0&0\\C&0&D&0\\0&0&0&\overline{K}
\\\end{array}\right)|\ \left(\begin{array}{cc}A&B\\C&D\\\end{array}\right)\in\Sp(k),K\in T_{n-k}\big\},\]
\[D'_{k}=\big\{\left(\begin{array}{cccc}A&0&B&0\\0&K&0&0\\C&0&D&0\\0&0&0&\overline{K}
\\\end{array}\right)|\left(\begin{array}{cc}A&B\\C&D\\\end{array}\right)\in
\SO(2k),K\in T_{n-k}\big\}\] and \[D_{k}=J_{k}D'_{k}J_{k}^{-1}.\]
Thus $T$ is a maximal torus of $D_{k}$, $C_{k}$ and $A_{k-1}$, and the root systems
\[\Phi(D_{k},T)=\D_{k},\] \[\Phi(C_{k},T)=\C_{k}\] and \[\Phi(A_{k-1},T)=\A_{k-1}.\]
By \cite{Oshima}, we know that any sub-root system $\Phi\subset\C_{n}$ is of the form
\[\A_{r_1-1}+\cdots+\A_{r_{i}-1}+\C_{s_{1}}+\cdots+\C_{s_{j}}+\D_{t_1}+\cdots+\D_{t_{k}},\]
where $r_{1},..,r_{i},s_1,...,s_{j},t_1,...,t_{k}\geq 1$ and \[r_1+\cdots+r_{i}+s_1+\cdots+
s_{j}+t_1+\cdots+t_{k}=n.\] Here we regard $\A_{0}=\emptyset$.
Using subgroups of block-form, we see that each sub-root system $\Phi$ of $\Psi_0$ is of the form
$\Phi(H,T)$ for some closed connected subgroup $H$ of $G$ with $T$ a maximal torus of $H$.

{\it Type $\BC_{n}$.} In the case that $\Psi_0=\BC_{n}$, let $G=\SU(2n+t)$, $t\geq n$,
\[T=\big\{\diag\{I_{t},z_1,\dots,z_{n},z_1^{-1},\dots,z_{n}^{-1}: |z_1|=|z_2|=\cdots=|z_{n}|=1\big\}\]
and $\theta=\Ad\left(\begin{array}{ccc}I_{t}&0&0\\0&0&I_{n}\\0&I_{n}&0\\\end{array}\right)$.
Then $\theta$ is an involutive automorphism of $G$ and $\Lie T$ is a maximal abelian subspace
of $\frp_0=\{X\in\Lie G|\ \theta(X)=-X\}$. Moreover, the restricted root system \[\Phi(G,T)=\BC_{n}=
\{\pm{e_{i}}\pm{e_{j}}|1\leq i<j\leq n\}\cup\{\pm{e_{i}},\pm{2e_{i}|1\leq i\leq n}\}\] (cf. \cite{Knapp},
Page 424), each root $2e_{i}$ occurs with multiplicity one, each root $\pm{e_{i}}\pm{e_{j}}$ occurs
with multiplicity two and each root $e_{i}$ occurs with multiplicity $2t$. We have
\[\Gamma^{\circ}=N_{G}(T)/C_{G}(T)=W_{\BC_{n}}=\Aut(\BC_{n})=\Gamma.\]
Denote by \[J_{k}=\left(\begin{array}{ccccc}I_{t}&&&&\\&\frac{1}{\sqrt{2}}I_{k}&0&0
\frac{1}{\sqrt{2}}iI_{k}&0\\&0&I_{n-k}&0&0\\&\frac{1}{\sqrt{2}}iI_{k}&0&\frac{1}
{\sqrt{2}}I_{k}&0\\&0&0&0&I_{n-k}\\\end{array}\right),\] which is a $2n+t$ square matrix.
Define the subgroups
\[A_{k-1}=\big\{\left(\begin{array}{ccccc}I_{t}&&&&\\&A&0&0&0\\&0&K&0&0\\&0&0&
\overline{A}&0\\&0&0&0&\overline{K}\\\end{array}\right)|\ A\in\U(k),K\in T_{n-k}\big\},\]
\[B'_{k}=\big\{\left(\begin{array}{ccccc}I_{t-1}&&&&\\&A&0&B&0\\&0&K&0&0
\\&C&0&D&0\\&0&0&0&\overline{K}\\\end{array}\right)|\ \left(\begin{array}{cc}
A&B\\C&D\\\end{array}\right)\in\SO(2k+1),K\in T_{n-k}\big\},\]
\[C_{k}=\big\{\left(\begin{array}{ccccc}I_{t}&&&&\\&A&0&B&0\\&0&K&0&0\\&C&0&D&0\\&0&0&0&\overline{K}
\\\end{array}\right)|\ \left(\begin{array}{cc}A&B\\C&D\\\end{array}\right)\in\Sp(k),K\in T_{n-k}\big\},\]
\[D'_{k}=\big\{\left(\begin{array}{ccccc}I_{t}&&&&\\&A&0&B&0\\&0&K&0&0\\&C&0&D&0
\\&0&0&0&\overline{K}\\\end{array}\right)|\ \left(\begin{array}{cc}A&B\\C&D
\\\end{array}\right)\in\SO(2k),K\in T_{n-k}\big\}, \] \[B_{k}=J_{k}B'_{k}J_{k}^{-1}\] and
\[D_{k}=J_{k}D'_{k}J_{k}^{-1}.\] Thus $T$ is a maximal torus of $D_{k}$, $C_{k}$, $A_{k-1}$ and $B_{k}$,
and the root systems \[\Phi(D_{k},T)=\D_{k},\] \[\Phi(C_{k},T)=\C_{k},\] \[\Phi(A_{k-1},T)=\A_{k-1}\]
and \[\Phi(B_{k},T)=\B_{k}.\] Using subgroups of block form, we see that each reduced root system
$\Phi$ of $\Psi_0$ is of the form $\Phi(H,T)$ for some connected closed subgroup $H$ of $G$ with $T$ a
maximal torus of $H$. Note that the condition $t\geq n$ makes sure that the sub-root system $n\B_1$ corresponds
to a subgroup.

{\it Type $\B_{n}$} In the case that $\Psi_0=\B_{n}$, let $(G,T)$ be the same as the above for $\BC_{n}$. Then
this pair satisfies the desired conclusion.

{\it Type $\F_4$.} In the case that $\Psi_0=\F_4$. Recall that the complex simple Lie algebra $\fre_7(\bbC)$
has a real form with a restricted root system isomorphic to $\F_4$ (cf. \cite{Knapp}, Page 425). Let
$G=\Aut(\fre_7)$ be the automorphism group of a compact simple Lie algebra of type $\E_7$. Then we can
choose an involution $\theta_0$ in $G$ and a closed connected torus $T$ in $G$ such that $\theta|_{T}=-1$,
$\Lie T$ is a maximal abelian subspace of $\fre_7^{-\theta}=\{X\in\fre_7|\Ad(\theta)(X)=-X\}$ and the
restricted root system $\Phi(G,T)\cong\F_4$. Write $\Psi_0=\Phi(G,T)$. Then $\Psi_0\cong\F_4$ and
we show that any sub-root system $\Phi$ of $\Psi_0$ is of the form $\Phi=\Phi(H,T)$ for a closed
connected subgroup $H$ of $G$ with $T$ a maximal torus of $H$. To show this, we use some results from
\cite{Huang-Yu}. Starting with the Klein four subgroup $\Gamma_6$ of $G$ (cf. \cite{Huang-Yu}, Table 4),
we have $G^{\Gamma_6}=(G^{\Gamma_6})_0\times\Gamma$ and $(G^{\Gamma_6})_0\cong\F_4$. From \cite{Huang-Yu},
we know that some involutions in $(G^{\Gamma_6})_0$ are conjugate to $\theta_0$ in $G$. Without loss of
generality we may assume that $\theta_0\in(G^{\Gamma_6})_0$. Choosing any $1\neq\theta_1\in\Gamma_6$, we
have (cf. \cite{Huang-Yu}) \[(G^{\theta_1})_0\cong(\E_6\times\U(1))/\langle(c,e^{\frac{2\pi i}{3}})\rangle,\]
\[(G^{\theta_0\theta_1})_0\cong\SU(8)/\langle iI\rangle\] and
\[\theta_0\in(G^{\theta_1})_0\cap(G^{\theta_0\theta_1})_0.\] Here $1\neq c\in Z(\E_6)$ with $o(c)=3$.
By these, we get three closed connected subgroups $H_1,H_2,H_3$ of $G$ containing $T$ and $\theta_0$,
and such that \[\Lie H_1\cong\mathfrak{su}(8),\] \[\Lie H_2\cong\mathfrak{e}_6\] and
\[\Lie H_3\cong\mathfrak{f}_4.\] Moreover, $\Lie T$ is a maximal abelian subspace
of each of $(\Lie H_{i})^{-\theta_0}$, $i=1,2,3$. By the classification of sub-root systems of
$\F_4$ (cf. Section \ref{S:sub-root systems}), we can show that any sub-root system
$\Phi$ of $\Psi_0\cong\F_4$ is of the form $\Phi(H,T)$ for a connected closed subgroup $H$ contained in
one of $H_1,H_2,H_3$.

{\it Type $\G_2$.}
In the case that $\Psi_0=\G_2$, let $G=\Spin(8)$. There exist (cf. \cite{Helgason}, Page 517) two order three outer
automorphisms $\theta,\theta'$ of $\G$ such that \[H_1=G^{\theta}\cong\G_2,\]
\[H_2=G^{\theta'}\cong\PSU(3)\] and $H_1,H_2$ share a common maximal torus $T$. Thus $\Phi(G,T)\cong\G_2$,
each long root occurs with multiplicity $1$ and each short root occurs with multiplicity $3$. By the
classification of sub-root systems of $\F_4$ (cf. Section \ref{S:sub-root systems}), one can show that
any sub-root system $\Phi$ of $\Psi_0$ is of the form $\Phi(H,T)$ for some connected closed
subgroup $H$ of $H_1$ or $H_2$.
\end{proof}

\begin{remark}\label{R:construction-simple}
When $\Psi_0$ is an irreducible root system not of type $\B_{n}$ ($n\geq 3$), in the above construction we
actually have $\Psi'_{T}=\Psi_0$.
\end{remark}

\begin{proposition}\label{P:Gamma0}
Given a root system $\Psi'$ and a finite group $W$ between $W_{\Psi'}$ and $\Aut(\Psi')$,
there exists a compact (not necessarily connected) Lie group $G$ with a bi-invariant
Riemannian metric $m$ and a connected closed torus $T$ in $G$ such that:
\begin{itemize}
\item[(1)]{$\rank\Psi_{T}=\dim T=\rank\Psi'$.}
\item[(2)]{$\Psi'\subset\Psi_{T}$ and is stable under $\Gamma^{\circ}$.}
\item[(3)]{$\Gamma^{\circ}=W$ as groups acting on $\Psi'$.}
\item[(4)]{For each reduced sub-root system $\Phi$ of $\Psi'$, there exists a connected closed subgroup
$H$ of $G$ with $T$ a maximal torus of $H$ and with root system $\Phi(H,T)=\Phi$.}
\end{itemize}
\end{proposition}

\begin{proof}
Decomposing $\Psi'$ into a disjoint union of (orthogonal) simple root systems and jointing those
isomorphic simple factors, we may write $\Psi'$ as the form
\[\Psi'=\Psi'_1\bigsqcup\cdots\bigsqcup\Psi_{s}\] where each $\Psi'_{i}=m_{i}\Psi_{i,0}$
is a union of $m_{i}$ root systems all isomorphic to an irreducible root system $\Psi_{i,0}$.

For each $i$, by Proposition \ref{P:construction-simple}, we get $(G_{i},T_{i})$ corresponds
to the irreducible root system $\Psi_{i,0}$ and the finite group $W_{\Psi_{i,0}}$ satisfying all
conditions in the Proposition. Let \[G'=G_{1}^{m_{1}}\times\cdots\times G_{s}^{m_{s}}\] and
\[T=T_{1}^{m_{1}}\times\cdots\times T_{s}^{m_{s}}.\] Then $(G',T)$ corresponds to the root system
$\Psi'$ and the finite group $W_{\Psi'}$ with all conditions in the conclusion satisfied.

Among $\{\Psi_{i,0}|\ 1\leq i\leq s\}$, we may assume that $\Aut(\Psi_{i,0})\neq W_{\Psi_{i,0}}$
happens exactly when $1\leq i\leq t$. For each $i\leq t$, since $\Aut(\Psi_{i,0})\neq W_{\Psi_{i,0}}$,
$\Psi_{i,0}$ must be simply laced. In this case we have $G_{i}=\Int(\fru_{i})$ for a compact
simple Lie algebra $\fru_{i}$ with root system $\Psi_{i,0}$.
Let \[G''=(\Aut(\fru_{1})^{m_{1}}\rtimes S_{m_1})\times\cdots\times(\Aut(\fru_{t})^{m_{t}}
\rtimes S_{m_{t}})\times (G_{t+1}^{m_{t+1}}\rtimes S_{m_{t+1}})\times\cdots\times
(G_{s}^{m_{s}}\rtimes S_{m_{s}})\] and \[T=T_{1}^{m_{1}}\times\cdots\times T_{s}^{m_{s}}.\]
Hence $(G'',T)$ corresponds to the root system $\Psi'$ and the finite group $\Aut(\Psi')$ with
all conditions in the conclusion satisfied.

We have $G''/G'\cong\Aut(\Psi')/W_{\Psi'}$. Corresponding to the subgroup $W/W_{\Psi'}$ of
$\Aut(\Psi')/W_{\Psi'}$, we get a subgroup $G$ of $G''$ containing $G'$ and with $G/G'=W/W_{\Psi'}$
in the identification $G''/G'=\Aut(\Psi')/W_{\Psi'}$. Therefore $(G,T)$ corresponds to $(\Psi',W)$
with all conditions in the conclusion satisfied.
\end{proof}


As in the proof of Proposition \ref{P:construction-simple}, given an irreducible root system
$\Psi_0=\BC_{n}$, let $G_1=\SU(2n+t)$ for $t\geq n$ and \[T_1=\{\diag\{z_1,\dots,z_{n},z_1^{-1},\dots,
z_{n}^{-1},\underbrace{1,...,1}_{t}\}|\ |z_1|=|z_2|=\cdots=|z_{n}|=1\}.\] Denote by $S_1$ the group of
diagonal matrices in $G_1$. It is a maximal torus of $G_1$. Let \[\epsilon_{i}
(\diag\{z_1,\cdots,z_{2n+t}\})=z_{i}\] for any $\diag\{z_1,\cdots,z_{2n+t}\}\in S_1$ and
\[e_{i}(\diag\{z_1,\dots,z_{n},z_1^{-1},\dots,z_{n}^{-1},\underbrace{1,...,1}_{t}\})=z_{i}\]
for any $\diag\{z_1,\dots,z_{n},z_1^{-1},\dots,z_{n}^{-1},\underbrace{1,...,1}_{t}\}\in T$.
Thus $$\epsilon_{i}|_{T}=\left\{\begin{array}{rcl} e_{i} &\mbox{if} &1\leq i\leq n\\
e_{i-n}^{-1}&\mbox{if} &n+1\leq i\leq 2n\\ 1 &\mbox{if} &2n+1\leq i\leq 2n+t.
\end{array}\right.$$

The following example is a modification of an example in \cite{Larsen-Pink}, Page 392.

\begin{example}\label{E:LP}
Given $r\geq n$, denote by $G_1=\SU(2n+t)$, $T_1\subset G_1$ the subgroup of diagonal matrices,
$G_2=(G_1)^{r}$ and $T_2=(T_1)^{r}$. Write \[\lambda_{j}=\sum_{1\leq i\leq n}(nj-i+1)\epsilon_{i}\] for
$1\leq j\leq r$. Let $V_{j}=V_{\lambda_{j}}$ be an irreducible representation of $G_1$ with highest
weight $\lambda_{j}$ and \[V=\bigoplus_{\sigma\in A_{r}}V_{\sigma(1)}\otimes\cdots\otimes V_{\sigma(r)}.\]
Denote by $G=\SU(k)$,  where \[k=\frac{r!}{2}\prod_{1\leq j\leq r}\dim V_{j}.\] The representation $V$
gives us an embedding $G_2\subset G$. Write $T$ for the image of $T_2$ under this embedding. Since
$G=\SU(k)$ is simple, a biinvariant Riemannian metric $m$ on $G$ is unique up to a scalar multiple.
Hence $\Psi_{T}$ does not depend on the choice of $m$. In this example, we have:
\begin{itemize}
\item[(1)]{$\rank\Psi_{T}=\dim T=rn$.}
\item[(2)]{$\Psi'=r\BC_{n}\subset\Psi_{T}$ and it is stable under $\Gamma^{\circ}$.}
\item[(3)]{$\Gamma^{\circ}=W_{\Psi'}\rtimes A_{r}=(W_{\BC_{n}})^{r}\rtimes A_{r}$ as groups acting
on $\Psi'$.}
\item[(4)]{For each reduced sub-root system $\Phi$ of $\Psi'$, there exists a closed connected
subgroup $H$ of $G$ with $T$ a maximal torus of $H$ and with root system $\Phi(H,T)=\Phi$.}
\end{itemize}
\end{example}

\begin{proof}
First we have $(W_{\BC_{n}})^{r}\rtimes A_{r}\subset\Gamma^{\circ}$ and $\Psi_{T}=\BC_{rn}$
by our construction of the pair $(G,T)$.

We have $\Psi'=r\BC_{n}\subset\Psi_{T}$ since $r\BC_{n}$ is the root system of
$(G_2,T)$. By Proposition \ref{P:construction-simple}, any reduced sub-root system
$\Phi$ of $r\BC_{n}$ is of the form $\Phi(H,T)$ for a closed connected subgroup
$H$ of $G_2$.

The equality $\Gamma^{\circ}=(W_{\BC_{n}})^{r}\rtimes A_{r}$ can be proven using the idea in
\cite{Larsen-Pink}, Page 392. It proceeds as follows. By the construction of $G_2$ and $V$, the character
\[\chi_{V_1\otimes\cdots\otimes V_{r}}|_{T}\] has $\lambda=(nr,nr-1,\cdots,1)$ as a leading term.
It is regular with respect to $\Aut(T,m)=W_{\BC_{rn}}$. The weights appearing in
$\chi_{V}|_{T}$ of maximal length are in the orbit $((W_{\BC_{n}})^{r}\rtimes A_{r})\lambda$.
Hence $\Gamma^{\circ}=(W_{\BC_{n}})^{r}\rtimes A_{r}$.
\end{proof}


\begin{remark}\label{R:LP}
In Example \ref{E:LP}, moreover we have $\Psi'_{T}=r\BC_{n}=\Psi'$ since otherwise
$\Gamma^{\circ}$ must be larger than $(W_{\BC_{n}})^{r}\rtimes A_{r}$.
\end{remark}


\begin{proposition}\label{P:Gamma0-2}
Given a root system $\Psi'$ and a finite group $W$ between $W_{\Psi'}$ and $\Aut(\Psi')$,
there exists some $G=\SU(k)$ and a closed connected torus $T$ in $G$ such that:
\begin{itemize}
\item[(1)]{$\rank\Psi_{T}=\dim T=\rank\Psi'$.}
\item[(2)]{$\Psi'\subset\Psi_{T}$ and it is stable under $\Gamma^{\circ}$.}
\item[(3)]{$\Gamma^{\circ}=W$ as groups acting on $\Psi'$.}
\item[(4)]{For each reduced sub-root system $\Phi$ of $\Psi'$, there exists a closed connected subgroup
$H$ of $G$ with $T$ a maximal torus of $H$ and with root system $\Phi(H,T)=\Phi$.}
\end{itemize}
\end{proposition}

\begin{proof}
As in the proof of Proposition \ref{P:Gamma0}, we write $\Psi'$ as the form
\[\Psi'=\Psi'_1\bigsqcup\cdots\bigsqcup\Psi_{s},\] where each $\Psi'_{i}=m_{i}\Psi_{i,0}$
is a union of $m_{i}$ root systems all isomorphic to an irreducible root system $\Psi_{i,0}$.

For each $i$, by Proposition \ref{P:construction-simple}, we get $(G_{i},T_{i})$ corresponds
to the root system $\Psi_{i,0}$ and the finite group $W_{\Psi_{i,0}}$ satisfying all conditions
in the Proposition. Moreover, if $\Psi_{i,0}$ is simply laced, then $G_{i}=\Int(\fru_{i})$ for a
compact simple Lie algebra $\fru_{i}$ with root system isomorphic to $\Psi_{i,0}$. Let
\[G'=G_{1}^{m_{1}}\times\cdots\times G_{s}^{m_{s}}\] and
\[T'=T_{1}^{m_{1}}\times\cdots\times T_{s}^{m_{s}}.\]

Among $\Psi_{i,0}$, we assume that $\Aut(\Psi_{i,0})\neq W_{\Psi_{i,0}}$ happens exactly
when $1\leq i\leq t$. We have \[\Aut(\Psi')/W_{\Psi'}=(\Out(\Psi_{1,0})^{m_1}\rtimes S_{m_1})
\times\cdots\times(\Out(\Psi_{t,0})^{m_t}\rtimes S_{m_t})\times S_{m_{t+1}}\times\cdots\times
S_{m_{s}}.\] Here $\Out(\Psi_{i,0})=\Aut(\Psi_{i,0})/W_{\Psi_{i,0}}$. The group
$\Aut(\Psi')/W_{\Psi'}$ acts on $\Rep(G')$ (here $G_{i}=\Int(\fru_{i})$ plays a role) through
its action on the dominant integral weights.

For any $1\leq i\leq s$, choose a maximal torus $S_{i}$ of $G_{i}$ containing $T_{i}$. If
$i\leq t$, we choose $m_{i}$ dominant integral weights $\lambda_{i,1},\cdots,\lambda_{i,m_{i}}$
of $S_{i}=T_{i}$ such that the set $\{\gamma\lambda_{i,j}|\gamma\in\Aut(\Psi_{i,0}),1\leq j\leq m_{i}\}$
has cardinality exactly $m_{i}|\Aut(\Psi_{i,0})|$. If $i\geq t+1$, the restriction
map of weight lattices \[p_{i}: \hat{S_{i}}\longrightarrow\hat{T_{i}}\] is surjective
and it is an orthogonal projection. Choose $m_{i}$ regular dominant integral weights
$\lambda_{i,1},\cdots,\lambda_{i,m_{i}}$ of $S_{i}$ with an additional property:
each $\lambda_{i,j}$ is orthogonal to the weights in $\ker p_{i}$ and their images under $p_{i}$
are regular and distinct to each other. Thus the weights of maximal length in $V_{\lambda_{i,j}}|_{T_{i}}$
are those in the orbit $W_{\Psi_{i,0}}\lambda_{i,j}$, and each occurs with multiplicity one. Here
$V_{\lambda_{i,j}}$ is an irreducible representation of $G_{i}$ with highest weight $\lambda_{i,j}$.
Denote by \[V=\bigoplus_{\sigma\in W/W_{\Psi'}}\sigma(V_{\lambda_1}\otimes\cdots\otimes V_{\lambda_r}).\]
and $G=\SU(k)$, where \[k=\dim V=\frac{|W|}{|W_{\Psi'}|}\prod_{1\leq j\leq r}\dim V_{j}.\]
The representation $V$ gives us an embedding $G'\subset G=\SU(k)$. Write $T$ for the image of $T'$
under this embedding. Arguing similarly as in the proof of Example \ref{E:LP}, we get
$\Gamma^{\circ}=W$. By Proposition \ref{P:construction-simple}, the other conditions in the
conclusion are also satisfied.
\end{proof}





\begin{remark}
In the above construction, we know $\Psi'_{T}\supset\Psi'$ by Property $(4)$ of Proposition
\ref{P:Gamma0-2} and $W_{\Psi'_{T}}\subset\Gamma^{\circ}$ by Proposition \ref{P:Gamma0-Psi'}.
From this $\Psi'_{T}$ is almost determined, which should be close to being equal to $\Psi'$. On the other
hand, $\Psi_{T}$ is probably much larger than $\Psi'$. For example it may happen that $\Psi'=r\BC_{n}$ and
$\Psi'_{T}=\BC_{nr}$.


Can we make an example with $\Psi'_{T}=\Psi'$ in Proposition \ref{P:Gamma0-2}? Moreover, can we
make an example with $\Psi_{T}=\Psi'$?
\end{remark}

\section{Irreducible subgroups}\label{S:Irreducible}

Let $G=\U(n)$. In this section, we study dimension data of closed subgroups $H$ of $G$ acting irreducibly
on $\bbC^{n}$.

Choose a prime $p$, an integer $m\geq 4$ and a prime $q>p^{m}$. Let $n=p^{m}$, $T\subset G$ be the subgroup of
diagonal unitary matrices and \[A=\{\diag\{a_1,a_2,\dots,a_{n}\}|\ a_1^{q}=a_2^{q}=\cdots=a_{n}^{q}=1\}.\]

\begin{lemma}\label{L:GA-normalizer}
$C_{G}(A)=T$ and $N_{G}(A)=T\rtimes S_{n}$, where $S_{n}$ is the subgroup of permutation matrices in $G$.

There exists a unique conjugacy class of subgroups $\overline{N}$ of $S_{n}$ isomorphic to $(C_{p})^{m}$ and with
non-identity elements all conjugate to $(1,2,\dots,p)(p+1,p+2,\dots,2p)\cdots(n-p+1,n-p+2,\dots,n)$.

For any subgroup $\overline{N}$ of $S_{n}$ as above, $C_{S_{n}}(\overline{N})=\overline{N}$ and
\[N_{S_{n}}(\overline{N})/\overline{N}\cong GL(m,\mathbb{F}_{p}).\]
\end{lemma}

\begin{proof}
The first statement is clear. For the second statement, one can prove the uniqueness by induction on $m$.
Again, one can show $C_{S_{n}}(\overline{N})=\overline{N}$ by induction on $m$. Finally, by the uniqueness
of $\overline{N}$, we get
\[N_{S_{n}}(\overline{N})/\overline{N}\cong\Aut(\overline{N})\cong GL(m,\mathbb{F}_{p}).\]
\end{proof}

Now we specify $p=2$ and $m=4$. Let $\overline{H}=N_{S_{n}}(\overline{N})$. Then $\overline{H}/\overline{N}
\cong GL(4,\mathbb{F}_{2})$. Choose a subgroup $\overline{H}_1$ of $\overline{H}$ with
$\overline{H}_1/\overline{N}$ corresponding to
\[\langle\left(\begin{array}{cccc}1&0&1&0\\0&1&0&1\\0&0&1&0\\0&0&0&1\\\end{array}\right),
\left(\begin{array}{cccc}1&0&0&1\\0&1&1&1\\0&0&1&0\\0&0&0&1\\\end{array}\right)\rangle\]
and a subgroup $\overline{H}_2$ of $\overline{H}$ with $\overline{H}_2/\overline{N}$ corresponding to
\[\langle\left(\begin{array}{cccc}1&0&1&0\\0&1&0&1\\0&0&1&0\\0&0&0&1\\\end{array}\right),
\left(\begin{array}{cccc}1&1&0&1\\0&1&0&0\\0&0&1&1\\0&0&0&1\\\end{array}\right)\rangle\]
Let $K=A\rtimes S_{n}$ and $N\subset K$ with $N/A=\overline{N}$. Write $H=N_{K}(N)$. Then $H/N=\overline{H}$.
Let $H_1, H_2\subset H$ with $H_1/N=\overline{H}_1$ and $H_2/N=\overline{H}_2$.

\begin{proposition}\label{P:U16-equal}
The two subgroups $H_1$ and $H_2$ act irreducibly on $\bbC^{16}$, have the the same dimension data and are
non-isomorphic.
\end{proposition}

\begin{proof}
Since $N$ acts irreducibly on $\bbC^{16}$ and $H_1,H_2$ contain $N$, $H_1$ and $H_2$ act irreducibly on
$\bbC^{16}$.

To show $H_1$ and $H_2$ having the same dimension data in $G$, it is sufficient to show they have
the same dimension data in $H$. Since $N$ is a normal subgroup of $H$, $H_1$ and $H_2$, it is sufficient
to show $\overline{H}_1=H_1/N$ and $\overline{H}_2=H_2/N$ have the same dimension data in
$\overline{H}=H/N\cong GL(4,\mathbb{F}_2)$. We have $\overline{H}_1\cong\overline{H}_2\cong C_2\times C_2$.
On the other hand one can show that non-identity elements of $\overline{H}_1$ and $\overline{H}_2$ are all
conjugate to \[\left(\begin{array}{cccc}1&0&1&0\\0&1&0&1\\0&0&1&0\\0&0&0&1\\\end{array}\right).\]
Hence, $\overline{H}_1$ and $\overline{H}_2$ have the same dimension data in
$\overline{H}\cong GL(4,\mathbb{F}_2)$.

We show that the groups $H_1$ and $H_2$ are non-isomorphic. Since $A$ is a characteristic subgroup of
$H_1$ and $H_2$, it is sufficient to show $|Z(H_1/A)|\neq |Z(H_2/A)|$. Obviously $A\lhd N\lhd H_1,H_2$.
Moroever we have \[\overline{N}=N/A\cong (\mathbb{F}_2)^4\] as an abelian group and $H_1$, $H_2$ act on
$N/A$ through the action of $\overline{H}_1$, $\overline{H}_2$ on $(\mathbb{F}_2)^4$. By Lemma
\ref{L:GA-normalizer}, $\overline{N}$ is a maximal abelian subgorup of $S_{n}$. Therefore,
\[Z(H_{i}/A)=\overline{N}^{\overline{H}_{i}},\] the latter means the fixed point subgroup of the
$\overline{H}_{i}$ action on $\overline{N}$. By the definition of $\overline{H}_{1}$ and $\overline{H}_{2}$,
we get \[\overline{N}^{\overline{H}_{1}}\cong(\mathbb{F}_2)^2 \] and
\[\overline{N}^{\overline{H}_{1}}\cong\mathbb{F}_2.\] Hence, $H_1\not\cong H_2$.
\end{proof}

Choose $n=12$ and a prime $q>3$. Let $T\subset G$ be the subgroup of diagonal unitary matrices and
\[A=\{\diag\{a_1,a_2,\dots,a_{n}\}|\ a_1^{q}=a_2^{q}=\cdots=a_{n}^{q}=1\}.\] Write \[N=A\rtimes\langle\left(
\begin{array}{ccccc}0&1&0&\cdots&0\\0&0&1&\cdots&0\\\vdots&\vdots&\vdots&\vdots&\vdots\\0&0&0&\cdots&1\\
1&0&0&\cdots&0\\\end{array}\right)\rangle\] and $K=A\rtimes S_{n}$.

\begin{lemma}\label{L:GA-normalizer2}
$N_{K}(N)/N\cong(\bbZ/12\bbZ)^{\times}\cong C_2\times C_2$.
\end{lemma}

\begin{proof}
By Lemma \ref{L:GA-normalizer}, $C_{G}(A)=T$ and $N_{G}(A)=T\rtimes S_{n}$. Since $N/A\cong C_{12}$ is a maximal
abelian subgroup of $S_{12}$, \[N_{K}(N)/N\cong N_{S_{12}}(C_{12})/C_{12}\cong(\bbZ/12\bbZ)^{\times}.\]
By elementary number theory, $(\bbZ/12\bbZ)^{\times}\cong C_2\times C_2$.
\end{proof}

Let $H=N_{K}(N)$ and $\overline{H}=H/N$. Then $\overline{H}\cong C_2\times C_2$. Write $H_1$,
$H_2$, $H_3$, $H_4$, $H_5$ for the subgroups of $H$ containing $N$ and with $\overline{H}_1=H_1/N$,
$\overline{H}_2=H_2/N$, $\overline{H}_3=H_3/N$, $\overline{H}_4=H_4/N$, $\overline{H}_5=H_5/N$ all the
subgroups of $\overline{H}$, where $|\overline{H}_1|=1$, $|\overline{H}_2|=|\overline{H}_3|=|\overline{H}_4|=2$
and $|\overline{H}_5|=4$.

\begin{proposition}\label{P:U12-linear}
We have \[\mathscr{D}_{H_1}+2\mathscr{D}_{H_5}-(\mathscr{D}_{H_2}+\mathscr{D}_{H_3}+\mathscr{D}_{H_4})=0.\]
\end{proposition}

\begin{proof}
Since $N$ is a normal subgroup of $H_1$, $H_2$, $H_3$, $H_4$ and $H_5$, it is sufficient to show
the dimension data of the subgroups $\overline{H}_1$, $\overline{H}_2$, $\overline{H}_3$, $\overline{H}_4$,
$\overline{H}_5$ of $\overline{H}$ have the corresponding linear relation. The latter follows from
a consideration on irreducible representations of $C_2\times C_2$.
\end{proof}

There exists an example as in Proposition \ref{P:U12-linear} when $n=4$. For $n=4$, one has
$N_{G}(A)/C_{G}(A)\cong S_4$. Chooose a normal subgroup $\overline{N}$ of $S_4$
isomorphic to $C_2\times C_2$ (it is unique) and the subgroups $\overline{H}_1$, $\overline{H}_2$,
$\overline{H}_3$, $\overline{H}_4$ of $S_4$ (in the order of increasing orders) containing $\overline{N}$
Let $H_1$, $H_2$, $H_3$, $H_4$ be the corresponding subgroups of $A\rtimes S_4$. One can show that
\[\mathscr{D}_{H_1}+2\mathscr{D}_{H_4}-(2\mathscr{D}_{H_2}+\mathscr{D}_{H_3})=0.\]

Given a closed connected torus $T$ in $G$, recall that we have a weight lattice $\Lambda_{T}=\Hom(T,\U(1))$
and a finite group $\Gamma^{\circ}=N_{G}(T)/C_{G}(T)$. Let $\Phi$ be the root system of $H$, $\Lambda$ be the
root lattice of $H$ and $\rho_{T}$ be the character of the representation of $T$ on $\bbC^{n}$. Write
$X=\Lambda_{T}\otimes_{\bbZ}\bbR$. The following lemma is identical to Theorem 4 in \cite{Larsen-Pink}.
We state their theorem in a form we needed and sketch the proof.

\begin{lemma}\label{L:torus sharing}
Let $T$ be a closed connected torus in $G$. If there are more than one
conjugacy classes of closed connected subgroups $H$ of $G$ with $T$ a maximal rotus and acting irreducibly
on $\bbC^{n}$, then these representations $(H,\bbC^{n})$ are tensor products of the following list:
\begin{enumerate}
\item[(1),] $n=2^{m}$, $H=(\Spin(2m_1+1)\times\cdots\times\Spin(2m_{s}+1)/Z)$ and
$\bbC^{n}=M_{m_1}\otimes\cdots M_{m_{s}}$, where \[m=m_1+m_2+\cdots+m_{s},\] \[Z=\{(\epsilon_1,\dots,
\epsilon_{s})|\ \epsilon_{i}=\pm{1}, \epsilon_1\epsilon_2\cdots\epsilon_{s}=1\}\]
and $M_{m}$ is Spinor representation of $\Spin(2m+1)$. In this case,
$\{\pm{1}\}^{m}\rtimes S_{m}\subset\Gamma^{\circ}$.
\item[(2),] $n=2^{k^2+k}\prod_{1\leq i\leq k}\frac{i!}{(2i)!}\frac{(m-k-1+2i)!}{(m-k-1+i)!}$,
$H=\Sp(m)$ or $\SO(2m)$, $\bbC^{n}=V_{(k,k-1,\dots,1,0,\dots,0)}$, $1\leq k\leq m-1$
and $\frac{k(k+1)}{2}$ is odd; or, $H=\Sp(m)/\langle-I\rangle$ or $\SO(2m)/\langle-I\rangle$,
$\bbC^{n}=V_{(k,k-1,\dots,1,0,\dots,0)}$, $1\leq k\leq m-1$ and $\frac{k(k+1)}{2}$ is even.
In this case, $\{\pm{1}\}^{m}\rtimes S_{m}\subset\Gamma^{\circ}$.
\item[(3),] $n=27$, $H=\G_2$ or $\PSU(3)$,  $\bbC^{n}=V_{\lambda}$, $\lambda=2(e_1-e_3)$. The weight
$\lambda=2\omega_2$ for $\G_2$ and $\lambda=2\omega_1+2\omega_2$ for $\A_2$. In this case,
$W_{\G_2}\subset\Gamma^{\circ}$.
\item[(4),] $n=2^{12}$, $H=\F_4$, $\Sp(4)/\langle-I\rangle$ or $\SO(8)/\langle-I\rangle$,
$\bbC^{n}=V_{\lambda}$, $\lambda=3e_1+2e_2+e_3$. The weight $\lambda=\omega_3+\omega_4$ for
$\F_4$, $\lambda=\omega_1+\omega_2+\omega_3$ for $\C_4$ and $\lambda=\omega_1+\omega_2+\omega_3+\omega_4$
for $\D_4$, In this case, $W_{\F_4}\subset\Gamma^{\circ}$.
\item[(5),] Any $n\geq 1$ and irreducible semisimple subgroup $H$ of $\U(n)$. In this case,
$W_{\Phi}\subset\Gamma^{\circ}$, where $\Phi$ is the root system of $H$.
\end{enumerate}
\end{lemma}

\begin{proof}
Since $H$ acts irreducibly on $\bbC^{n}$, the center if $H$ is equal to $H\cap Z(\U(n))$. Hence, it is
either finite or 1-dimensional. Therefore, the action of $W_{\Phi}$ on $X$ is multiplicity free.
As $\Gamma^{\circ}\supset W_{\Phi}$, $\Gamma^{\circ}$ acts on $X$ multiplicity freely. Let
$X=\bigoplus_{1\leq i\leq m} X_{i}$ be the decomposition of $X$ into a sum of irredicible summands.
Then we have $\Phi=\bigsqcup_{1\leq i\leq m}\Phi\cap X_{i}$ and
\[\rho_{T}=\bigotimes_{1\leq i\leq m}\rho_{i},\] where $\rho_{i}$ is the character of a representation
of the Lie algebra with root system $\Phi\cap X_{i}$. Since the character ring $\bbQ[\Lambda_{T}]$
is a unque factorization domain and each $\rho_{i}$ has dominat terms with coefficient 1, the characters
$\{\rho_{i}\}$ are determined by $\rho_{T}$. By considering each $\rho_{i}$, we may assume that
$\Gamma^{\circ}$ acts irreducibly and non-trivially on $X$, which forces $H$ being semisimple. We suppose
this since now on.

Since $H$ is assumed to be a closed connected subgroup of $\U(n)$, the root lattice $\Lambda$ is generated
by the differences of elements in $\rho_{T}$ and
the integral weight lattice $\Lambda_{T}$ is generated by elements in $\rho_{T}$. This indicates, both
$\Lambda$ and $\Lambda_{T}$ are determined by $T$. Endowing $G$ with a biinvariant Riemannian metric, this
gives $\Lambda$ a positive definite inner product. We show that $\Phi^{\circ}$ is determined by $\Lambda$.
This can be proved by induction on the rank $l$ of 
$\Lambda$ (which is equal to the rank of $\Phi^{\circ}$). If $l=1$, then $\Phi^{\circ}=\A_1$ and it consists
of the non-zero elements of $\Lambda$ of shortest length. Hence, the statement is clear in this case. If $l>1$,
one can show that the shortest non-zero elements of $\Lambda$ are contained in $\Phi^{\circ}$. For each
element $\lambda$, let \[\lambda=\sum_{1\leq i\leq m}\alpha_{i},\] $\alpha_{i}\in\Phi^{\circ}$ be an
expression of this form with $m$ minimal. We have $(\alpha_{i},\alpha_{j})\geq 0$ for any $1\leq i\leq j\leq s$
since otherwise $\alpha_i+\alpha_{j}\in\Phi^{\circ}$ and we could find an expression with $m$ smaller. Hence,
$|\lambda|\geq|\alpha_{i}|$ for each $1\leq i\leq m$. Therefore, $m=1$ and $\lambda\in\Phi^{\circ}$.
Let $\Phi_1$ be the set of shortest non-zero elements of $\Lambda$. By the above, $\Phi_1$ is a root system
and is contained in $\Phi^{\circ}$. Let $\Lambda'$ be the sublattice of elements in $\Lambda$ orthogonal to
elements in $\Phi_1$ and $\Phi'$ be the sub-root system of elements in $\Phi^{\circ}$ orthogonal to
elements in $\Phi_1$. Then, $\Lambda'$ is determined by $\Lambda$ and $\Phi'$ generates $\Lambda'$. By induction,
$\Phi'$ is determined by $\Lambda'$. Therfore, $\Phi^{\circ}$ is determined by $\Lambda$.

Now we have $\Phi^{\circ}$ determined by $\Lambda$. Since $\Gamma^{\circ}$ acts irreducibly on $X$,
it acts transitively on the irreducible factors of $\Phi^{\circ}$. Hence, $\Phi^{\circ}$ is an orthogonal
sum of isomorphic irredicible root systems. Write $\Phi^{\circ}=m\Omega$ where $\Omega$ is an irredicible
root system. Note that $(B_{l})^{\circ}=lB_1$, $(C_{l})^{\circ}=D_{l}$, $(F_{4})^{\circ}=D_{4}$,
$(G_{2})^{\circ}=A_{2}$ and $\Phi^{\circ}=\Phi$ if $\Phi$ being of type $ADE$. Hence, if $\Omega\neq\A_1$,
then $\Phi=\sum_{1\leq i\leq m}\Phi_{i}$ with $(\Phi_{i})^{\circ}=\Omega$ for each $i$. In this case,
$\rho_{T}$ has a decomposition \[\rho_{T}=\bigotimes_{1\leq i\leq m}\rho_{i}\] where $\rho_{i}$ is the character
of an irreducible representation of a simple Lie algebra of type $\Phi_{i}$. Since each $\rho_{i}$ has dominant
terms of coefficient 1, $\{\rho_{i}|\ 1\leq i\leq m\}$ are determined by $\rho_{T}$. Therefore, we may assume that
$m=1$. That is, we need to consider the ambiguity arsing from the pairs $D_l\subset C_l$ ($l\geq 3$),
$D_4\subset F_4$, $C_4\subset F_4$ and $A_2\subset G_2$. This follows from some detailed study of
the characters of highest weight modules (cf. \cite{Larsen-Pink}, Page 395-396). We omit the details here but
list the results precisely in the items $(2)$-$(4)$ of the conclusion of the theorem. If $\Omega\neq\A_1$, then
\[\{\pm{1}\}^{m}\subset\Gamma^{\circ}\subset\{\pm{1}\}^{m}\rtimes S_{m}.\] The action of $\Gamma^{\circ}$ on
$\Phi^{\circ}$  gives a partition of $\{1,2,\dots,m\}$. This in turn gives a canonical tensor decomposition of
$\rho_{T}$ similar as the above for $\Omega\neq\A_1$ case. By this it is enough to consider the case of
$\Gamma^{\circ}=\{\pm{1}\}^{m}\rtimes S_{m}$. In this case the ambiguity is either as listed in
item $(1)$ of the conclusion, or there is no ambiguity as in item $(5)$.
\end{proof}

\begin{remark}\label{R:torus sharing}
In each item $(1)$-$(4)$, from the above proof, $\Gamma^{\circ}$ contains the Weyl group of the largest
dimensional Lie group appearing in the ambiguity, which is $\Spin(2m+1)$, $\Sp(m)$, $\F_4$, $\G_2$ respectively,
\end{remark}


\begin{theorem}\label{T:irreducible-independence}
Let $H_1,H_2,\dots,H_{s}$ be a list of closed connected subgroups of $G$ acting irreducibly on $\bbC^{n}$.
If they are non-conjugate to each other, then their dimension data are linearly independent.
\end{theorem}

\begin{proof}
By Proposition \ref{P:group-root system}, we may assume that $H_1$, $H_2$,..., $H_{s}$ have a common maximal
torus $T$. Applying Lemma \ref{L:torus sharing}, we get a canonical tensor decompostion of $\rho_{T}$ (the
character of the representation of $T$ on $\bbC^{n}$), which in turn gives a decompostion of each $H_{i}$.
Observe that if each item $(1)$-$(4)$ occurs, the group $\Gamma^{\circ}$ contains a large finite group, i.e.,
the Weyl group of the largest connected compact Lie group appearing in the ambiguity. By calculation one can
show that the dominant weights $2\delta_{\Phi}$ are non-conjugate to each other under this Weyl group, where
$\{\Phi\}$ are the root systems of groups appearing in each item $(1)$-$(4)$. Therefore, the dominant weights
$2\delta_{\Phi_i}$ are non-conjugate to each other under $\Gamma^{\circ}$, where $\Phi_{i}$ is root system of
$H_{i}$. Hence, the dimension data of $H_1,H_2,\dots,H_{s}$ are linearly independent.
\end{proof}



\end{document}